\documentclass[10pt]{amsart}
\usepackage[utf8]{inputenc}
\usepackage{amsmath}
\usepackage{amsfonts}
\usepackage{amssymb}
\usepackage{amsthm}
\usepackage{enumitem}
\usepackage{graphicx}
\usepackage{blindtext}
\usepackage{tgpagella}

\title[Lefschetz fibrations]{Lefschetz fibrations on cotangent bundles and some plumbings}

\author{Sangjin Lee}
\address[Sangjin Lee]{Korea Institute for Advanced Study, 85 Hoegiro Dongdaemun-gu, Seoul 02455, Republic of Korea}
\email{sangjinlee@kias.re.kr (primary), sangjinlee@ibs.re.kr}

\graphicspath{ {images/} }

\newtheorem{theorem}{Theorem}
\newtheorem{corollary}[theorem]{Corollary}
\newtheorem{lemma}[theorem]{Lemma}

\theoremstyle{definition}
\newtheorem{definition}[theorem]{Definition}

\theoremstyle{definition}

\theoremstyle{definition}
\newtheorem{exmp}[theorem]{Example}

\theoremstyle{definition}
\newtheorem{remark}[theorem]{Remark}
\newtheorem*{remark*}{Remark}

\newtheorem{proposition}[theorem]{Proposition}

\newcommand{\ind}{\operatorname{ind}}
\newcommand{\hs}{\hspace{0.2em}}

\makeatletter
\@addtoreset{theorem}{section}
\makeatother

\usepackage{tikz}
\usepackage[framemethod=tikz]{mdframed}

\setlist{noitemsep}
\begin{document}
\maketitle 

\noindent{\bf Abstract}
We construct Lefschetz fibrations of a Weinstein manifold from its Weinstein structure when the Weinstein manifold is one of the following three types: cotangent bundle, plumbing of two cotangent bundles, or plumbing of multiple copies of $T^*S^n$ whose plumbing pattern is a tree.



\section{Introduction}
\label{section introduction}

\subsection{Introduction}
\label{subsect motivation}
{\em Lefschetz fibrations} are powerful tools in studying symplectic topology.  
The following are a few examples of that:
When a Lefschetz fibration is given, one can define the Fukaya-Seidel category as described in \cite{Seidel1}.
In \cite{Maydanskiy, MS, Abouzaid-Seidel}, the authors used Lefschetz fibrations for constructing diffeomorphic pairs of different Weinstein manifolds. 
When Wu \cite{Wu} studied the symplectic mapping class group of the Milnor fiber of $A_n$-type, the well-known Lefschetz fibration of the Milnor fiber played a key role. 
McLean \cite{Mclean, McLean12} showed that one could compute a symplectic homology of a Liouville manifold from a Lefschetz fibration and its monodromy map. 

It is natural to ask which symplectic manifolds admit Lefschetz fibrations. 
Giroux and Pardon \cite{GP} gave a wonderful answer.
They proved that every Stein manifold should admit a Lefschetz fibration. 
Moreover, \cite{GP} proved that every Weinstein manifold should admit a Lefschetz fibration {\em indirectly}, based on the equivalence between Stein and Weinstein manifolds. 

In the present paper, we construct Lefschetz fibrations on some Weinstein manifolds directly from their Weinstein structures. 
More specific results will appear in Section \ref{subsect results}.

\subsection{Results}
\label{subsect results}
The main results of this paper are to construct a Lefschetz fibration of $W$ if $W$ satisfies one of the following three cases:
\begin{itemize}
	\item $W$ is a cotangent bundle of a smooth manifold $M$, i.e., $W = T^*M$, 
	\item $W$ is a plumbing of two cotangent bundles $T^*M_1$ and $T^*M_2$ at one plumbing points, or 
	\item $W$ is a plumbing of multiple copies of $T^*S^n$ such that the plumbing pattern is a tree $T$. 
\end{itemize}

More precisely, we prove Theorems \ref{thm rough statement}--\ref{thm rough statement3}.

\begin{theorem}[Technical statement is Theorem \ref{thm main}]
	\label{thm rough statement}
	Let $M$ be a smooth manifold.
	We give an algorithm producing a Lefschetz fibration on $T^*M$ from a handle decomposition of $M$.  
\end{theorem}

\begin{theorem}[Technical statement is Theorem \ref{thm plumbing}]
	\label{thm rough statement2}
	Let $M_1$ and $M_2$ be smooth manifolds of the same dimension. 
	There is an algorithm producing a Lefschetz fibration on the plumbing of $T^*M_1$ and $T^*M_2$ at one point from a pair of handle decompositions of $M_1$ and $M_2$.  
\end{theorem}

\begin{theorem}[Technical statement is Theorem \ref{thm plumbing along tree}]
	\label{thm rough statement3}
	Let $P$ be a Weinstein manifold obtained by plumbing $T^*S^n$ along a tree $T$.
	Then, we give an algorithm producing a Lefschetz fibration defined on $P$. 
\end{theorem}

\begin{remark}
	We note that \cite{Johns} proved Theorem \ref{thm rough statement} for the case of any closed surface $M$. 
	One can see that \cite{Johns} and Theorem \ref{thm rough statement} give the same Lefschetz fibration for the surface case. 
	For more detail, see Section \ref{subsection the case of cotangetn bundles of surfaces}.
\end{remark}

Before going further, we briefly explain the ideas of Theorems \ref{thm rough statement}--\ref{thm rough statement3}.

\subsubsection{The idea for Theorem \ref{thm rough statement}}
\label{subsubsection idea of the first theorem}
First, we see what a Lefschetz fibration can say about the Weinstein structure on its domain. 

Let $W$ be a Weinstein manifold equipped with a Lefschetz fibration $\pi: W \to \mathbb{C}$. 
Let $F$ be the regular fiber of $\pi$. 
It is well-known that $\pi$ gives a decomposition of $W$ into two parts, one is the subcritical part given by $F \times \mathbb{C}$, and the other is a union of critical Weinstein handles that are attached to the subcritical part. 
We note that to attach critical Weinstein handles, one needs Legendrian attaching spheres on $\partial_\infty (F \times \mathbb{C})$. 
One can have specific attaching spheres from $\pi$, or more precisely, from the (cyclically ordered) collection of vanishing cycles. 
For more detail, we refer the reader to \cite[Section 8]{BEE} or Section \ref{subsect Lefschetz fibration}.

Moreover, if one has a Weinstein handle decomposition of the regular fiber $F$, then it gives a Weinstein handle decomposition of $W$. 
This is because a Weinstein handle decomposition of $F$ induces a Weinstein handle decomposition of the subcritical part $F \times \mathbb{C}$. 

From the above arguments, one can expect the other direction, i.e., producing a Lefschetz fibration on a Weinstein manifold $W$ from a Weinstein handle decomposition of $W$. 
In this paper, we investigate the idea. 

\begin{remark}
	\label{rmk 1}
In this paper, we consider cotangent bundles or some plumbings.
One of the main reasons why we focus on them is that it is easy to obtain their Weinstein handle decomposition. 
For example, if $W = T^*M$ where $M$ is a smooth $n$-dimensional manifold, then Lemma \ref{lem handle decomposition of cotangent bundle} gives an algorithm producing a Weinstein handle decomposition $\mathcal{H}_\mathcal{D}$ of $W$ from a handle decomposition $\mathcal{D}$ of $M$. 
The idea of Lemma \ref{lem handle decomposition of cotangent bundle} is to thicken an $n$-dimensional index $i$-handle in $\mathcal{D}$ in order to construct a $2n$-dimensional index $i$ Weinstein handle in $\mathcal{H}_\mathcal{D}$. 
Similarly, we can construct a Weinstein handle decomposition of a plumbing space $P$ from Weinstein handle decompositions of cotangent bundles if $P$ consisting of the cotangent bundles. 
Then, for more details, see Section \ref{section Lefschetz fibrations on plumbings} and Section \ref{section proof of plumbing tree theorem}.
\end{remark}

In order to investigate the above idea, let us assume that we have a Weinstein handle decomposition $\mathcal{H}$ of $W$.
We would like to consider the following two-step argument.  

The first step is to find a product structure on the subcritical part, i.e., the union of all subcritical Weinstein handles in $\mathcal{H}$. 
To be more precise, let $W_0$ denote the subcritical part.
Then, we would like to find a Weinstein manifold $F$ such that 
\[W_0 \simeq F \times \mathbb{C}.\]
If $\mathcal{H}$ produces a Lefschetz fibration $\pi$ on $W$, then $F$ is the regular fiber of $\pi$. 

The second step is to find the singular value information, i.e., their cyclic order and their vanishing cycles. 
We note that attaching spheres of critical Weinstein handles are Legendrian spheres on 
\[\partial_\infty W_0 \simeq \partial_\infty (F \times \mathbb{C}) \simeq \left(\partial_\infty F \times \mathbb{C}\right) \cup \left(F \times \partial_\infty \mathbb{C}\right).\]
If $\mathcal{H}$ produces a Lefschetz fibration, then one can acquire the singular value information from those attaching Legendrian spheres.

One can easily see that if the collection of Legendrian spheres give the singular value information, there exist some restrictions that the collection is necessarily to satisfy.
For example, the Legendrian spheres should lie on the vertical boundary of $W_0$, i.e., $F \times \partial_\infty \mathbb{C} \subset \partial_\infty W_0$. 
Moreover, each Legendrian sphere should ``correspond'' to a Lagrangian sphere of the Fiber $F$.
We do not explain in what sense the Legendrian and Lagrangian spheres correspond to each other, but if they are related, then the Lagrangian sphere becomes the corresponding vanishing cycle. 

We note that for any Weinstein handle decomposition $\mathcal{H}$ of $W$, the existence of $F$ in the first step is always guaranteed by \cite{Cieliebak, CE}.
However, not every $\mathcal{H}$ can pass the second step, because of the restrictions mentioned above. 
Section \ref{subsect example : the case of T^*S^n} gives an example of a Weinstein handle decomposition of $T^*S^n$, which cannot pass the second step. 
Thus, it is natural to ask what Weinstein handle decomposition can pass the second step. 

It is easy to show that if $\mathcal{H}$ satisfies the following condition $(\star)$, then $\mathcal{H}$ should pass the second step.  
\begin{itemize}
	\item[($\star$)] For any attaching sphere $\Lambda \subset \partial_\infty(F \times \mathbb{C})$ of a critical Weinstein handle, there exists an exact Lagrangian sphere $L \subset F$ such that $\Lambda$ is a Legendrian lift of $L$. 
\end{itemize}
The notion of Legendrian lift of an exact Lagrangian is defined in Definition \ref{def legendrian lift}.

Theorem \ref{thm rough statement} claims that if $W$ is a cotangent bundle, then there exists a Weinstein handle decomposition $\mathcal{H}$ of $W$ satisfying $(\star)$. 
We will construct such $\mathcal{H}$ from a handle decomposition of the zero section of $W$, or a Lagrangian skeleta of $W$. 
Before going further, we briefly explain a relation between a Weinstein handle decomposition and its corresponding Lagrangian skeleta. 

If $W_0$ is the union of all subcritical handles in $\mathcal{H}$ and if $\mathcal{H}$ has $m$ many critical handles, then the Lagrangian skeleta of $W$ induced from $\mathcal{H}$ is given as 
\begin{gather}
	\label{eqn skeleton 2}
	\mathrm{Skel}(W) = \mathrm{Skel}(W_0) \cup \left(\bigcup_{i=1}^m \mathbb{D}^n\right).
\end{gather} 
To be more precise, we note that every critical Weinstein handle has a unique zero of the Liouville vector field, thus we have $m$ many zeros from the critical handles. 
The stable manifolds of $m$ zeros are open Lagrangian disks in $W$. 
In Equation \eqref{eqn skeleton 2}, $\mathbb{D}^n$s mean the closure of the stable Lagrangian disks of zeros.

Also, we note that $W_0 \simeq F \times \mathbb{C}$ as described above.
By taking an isotropic change on the Weinstein structure, one can assume that $W_0$ admits the product Weinstein structure.  
Then, $\mathrm{Skel}(W_0)$ is
\[\mathrm{Skel}(W_0) = \mathrm{Skel}(F) \times \{0\}.\]

Thus, the skeleta of $W$ is determined by attachments of $m$-many disks to $\mathrm{Skel}(F)$. 
For a critical Weinstein handle in $\mathcal{H}$, one can encode the corresponding attaching information as a map 
\begin{gather}
 	\label{eqn attaching map}
 	\partial \mathbb{D}^n = S^{n-1} \to \mathrm{Skel}(W_0) = \mathrm{Skel}(F).
\end{gather}
Then, the image of the map will be 
\[\lim_{t \to -\infty} \phi^t_W(\Lambda),\]
where $\Lambda$ is the attaching sphere of the critical handle and $\phi^t_W$ is the Liouville vector flow of time $t$. 
We note that in the above equation, for each $t$, $\phi^t_W(\Lambda)$ is a closed subset of $W$ and the limit is the limit of closed subsets. 

Let us assume that the map in \eqref{eqn attaching map} is injective. 
Then, the image of the map is homeomorphic to a sphere.
Moreover, if the skeleton of $F$ satisfies technical conditions detailed in Section \ref{subsection recovering Lagrangian skeleta}, then it would be possible to obtain an exact Lagrangian sphere $L \subset F$ by smoothing the image of the map. 
Then, one can expect that the Legendrian lift of $L$ is Legendrian isotopic to $\Lambda$.
 
In Sections \ref{section the proof of Theorem main} and \ref{section examples}, we prove Theorem \ref{thm rough statement}. 
More precisely, for a given handle decomposition of $M$, we construct a Weinstein handle decomposition of $T^*M$ such that the maps in \eqref{eqn attaching map} are injective for all critical handles in $\mathcal{H}_\mathcal{D}$.
Then, we prove that this Weinstein handle decomposition passes two steps described above.

\subsubsection{The idea for Theorems \ref{thm rough statement2}--\ref{thm rough statement3}}
\label{subsubsectioin idea for plumbings}
Theorem \ref{thm rough statement2} considers the plumbing Weinstein manifolds $W$ of two cotangent bundles $T^*M_1, T^*M_2$. 
Roughly, in order to prove Theorem \ref{thm rough statement2}, we produce Lefschetz fibrations for $T^*M_1$ and $T^*M_2$ respectively, by applying Theorem \ref{thm rough statement}, then we combine those two Lefschetz fibrations. 

More precisely, we construct Weinstein handle decomposition $\mathcal{H}_i$ of $T^*M_i$, which produces a Lefschetz fibration.  
Let 
\[\mathcal{H}_1 =\{A_0, \dots, A_{m_1}\}, \mathcal{H}_2 = \{B_0, \dots, B_{m_2}\},\]
and let $A_0$ and $B_0$ be index 0 Weinstein handles. 

We note that 
\[A_0, B_0 \simeq \mathbb{D}^n \times \mathbb{D}^n.\]
Then, by attaching Weinstein handles $A_1, \dots, A_{m_1}$ (resp.\ $B_1, \dots, B_{m_2}$) along $(\partial \mathbb{D}^n) \times \mathbb{D}^n$ (resp.\ $\mathbb{D}^n \times (\partial \mathbb{D}^n)$), one obtains the plumbing space $W$. 

This gives a Weinstein handle decomposition of the plumbing space $W$. 
Moreover, the Weinstein handle decomposition produces a Lefschetz fibration of $W$. 

We would like to point out that the idea will work for plumbings of multiple cotangent bundles, after a slight modification, even if we plumb three or more cotangent bundles. 
As the modification, we need to identify a ``critical" handle of one cotangent bundle and the unique zero handle of another cotangent bundle. 
The modified idea will be explained with details at the beginning of Section \ref{section sketch of the proof of theorem}. 

Even though the modified idea seems work, when one applies the idea to a plumbing of three or more cotangent bundles, the Weinstein handle decomposition of the plumbing space could be complicated. 
And, the complexity of the handle decomposition will affect the complexity of the resulting Lefschetz fibration. 
Especially, the resulting Lefschetz fibration can have a complicated fiber. 

In order to avoid the complexity, we restrict our attention to {\em simple} plumbings, i.e., plumbings satisfying the following two conditions:
\begin{itemize}
	\item The plumbing spaces consist of multiple copies of $T^*S^n$. 
	\item The plumbing patterns are trees. 
\end{itemize}
Theorem \ref{thm rough statement3} and its proof given in Sections \ref{section sketch of the proof of theorem}--\ref{section proof of plumbing tree theorem} will give us an algorithm producing a Lefschetz fibration of a such plumbing.

\subsubsection{Other results}
\label{subsubsect other results}
The current paper consists of two parts, except Sections \ref{section introduction} and \ref{section preliminaries} which are the introduction and the preliminaries. 

The main theorem of the first part is Theorem \ref{thm rough statement}. 
We note that Theorem \ref{thm rough statement} gives multiple Lefschetz fibrations for a Weinstein manifold.
Thus, one can ask about the relationship between them.
Proposition \ref{prop rough statement} is an answer to that when the Weinstein manifold is a cotangent bundle of a surface.

\begin{proposition}[Technical statement is Proposition \ref{prop handle moves}]
	\label{prop rough statement}
	If $M$ is a $2$-dimensional smooth manifold, then the Lefschetz fibration on $T^*M$ obtained by applying Theorem \ref{thm rough statement} is unique up to four moves that are given in Section \ref{subsect four moves}.
\end{proposition}

The second part of this paper considers plumbing spaces of two types. 
The first type is plumbings of two cotangent bundles. Theorem \ref{thm rough statement2} considers the first type. 
The second type is plumbings of copies of the cotangent bundle of a sphere along a tree. Theorem \ref{thm rough statement3} considers the second type. 

Even though we consider some restricted plumbings, there exist possible applications. 
One of the possible applications is Corollary \ref{cor diffeomorphic family} which gives diffeomorphic families of Weinstein manifolds.
The members of a diffeomorphic family are plumbings of cotangent bundles of spheres.
Moreover, the families contain some Milnor fibers of simple singularities. 
For example, one can see that Minor fibers of $A_{4k+3}, D_{4k+3}$-singularities are diffeomorphic to each other if their dimension is $2n$ with odd $n \geq 3$. 
For more detail, see Corollary \ref{cor diffeomorphic family}.

\subsection{Acknowledgment}
\label{subsect acknoledgment} 
The author appreciates Hongtaek Jung for the helpful discussions. 
Also, the author appreciates Cheol-Hyun Cho for the discussion initiating the second part of the present paper. 
The author would also like to thank an anonymous referee for helpful comments and suggestions.

This work was partially supported by the Institute for Basic Science (IBS-R003-D1) and also by a KIAS Individual Grant (MG094401) at Korea Institute for Advanced Study.

\section{Preliminaries}
\label{section preliminaries}
In Section \ref{section preliminaries}, we review preliminaries and partially set notation.

\subsection{Handle decomposition}
\label{subsect handle decomposition}
In the present subsection, we explain what notion we mean by ``handle decomposition''.

\begin{definition}
	\label{def standard handle}
	\mbox{}
	\begin{enumerate}
		\item {\em An $n$-dimensional standard handle $h^i$ of index $i$} is a subspace 
		\[
		h^i=\mathbb{D}^i \times \mathbb{D}^{n-i}
		\]
		in $\mathbb{R}^n$, where $\mathbb{D}^k$ is the disk of radius $1$ in $\mathbb{R}^k$.  
		\item The {\em attaching region} of $h^i$ is $\partial \mathbb{D}^i \times \mathbb{D}^{n-i} = S^{i-1} \times \mathbb{D}^{n-i}$.
		Let $\partial_R h^i$ denote the attaching region of $h^i$.
	\end{enumerate}
\end{definition}
If there is no chance of confusion, we sometimes omit the dimension of a handle and simply call it $i$-handle. 

Let $M$ be an $n$-dimensional manifold with boundary. 
If there is a map $\phi: \partial_R h^i \to \partial M$, then one can attach the $n$-dimensional standard handle $h^i$ to $M$.
As the result of the attaching, one obtains another $n$-dimensional manifold, given as follows:
\[M \sqcup h / \sim, x \sim \phi(x) \text{  for all } x \in \partial_R h.\] 
Based on this, the notion of {\em handle decomposition of $M$} means data explaining the construction of $M$ as a union of handles. 
More precise definition is following bellow.  

\begin{definition}
	\label{def handle decomposition}
By a {\em handle decomposition of an $n$-dimensional smooth manifold $M$}, we mean a finite, ordered set of $n$-dimensional handles $\{h_0, \dots, h_m\}$ together with the injective maps $\phi_i : \partial_R h_i \to \partial(\cup_{j=0}^{i-1} h_j)$ satisfying the following:
\begin{itemize}
	\item $h_0$ is the unique index $0$-handle;
	\item there exists a natural number $N$ such that for $i \leq N$ (resp.\ $i >N$), $h_i$ is subcritical (resp.\ critical), i.e., $\ind(h_i) < n $ (resp.\ $\ind(h_i) = n$);
	\item two different critical handles are disjoint, or equivalently, every critical handle are attached to the union of subcritical handles;
	\item $\cup_{i=0}^m h_i$ is diffeomorphic to $M$. 	  
\end{itemize}
The maps $\phi_i$ are called {\em gluing maps}.
\end{definition}
We note that the word ``union" mentioned in the above definition does not mean the disjoint union of standard handles.
The union means the gluing by the gluing maps $\phi_i$. 

\begin{remark}
	\label{rmk nonusual definition of handle decomposition}
	We also note that Definition \ref{def handle decomposition} is not a definition which is usually used in literature. 
	However, we use Definition \ref{def handle decomposition} for some technical reasons which will appear later.
\end{remark}

We also define the following notation for the later use.
\begin{definition}
	\label{def set of handle decompositions}
	Let $\mathcal{H}(M)$ be the set of handle decomposition of a smooth manifold $M$. 
\end{definition}

\subsection{Weinstein Handle}
\label{subsect Weinstein handle}
We review the notion of Weinstein handle and their attachment in Section \ref{subsect Weinstein handle}.
For more detail, we refer the reader to Weinstein \cite{Weinstein}.

In order to define a standard Weinstein handle, we fix a smooth function $F : \mathbb{R}^2 \to \mathbb{R}$ such that
\begin{itemize}
	\item $F(0,0) \neq 0$,
	\item whenever $F(x,y)=0$, the partial derivatives of $F$, $\frac{\partial F}{\partial x}, \frac{\partial F}{\partial y}$ do not have the same sign,
	\item $\frac{\partial F}{\partial x} \neq 0$ when $y =0$, and
	\item $\frac{\partial F}{\partial y} \neq 0$ when $x =0$.
\end{itemize}

Let fix an integer $i$, in order to define the Weinstein handle of index $i$.
Let the standard symplectic Euclidean space $(\mathbb{R}^{2n}, \omega_{std})$ be equipped with a Liouville form 
\begin{gather}
	\label{eqn liouville form}
	\lambda_i = \sum_{j=1}^i -(x_j dy_j + 2 y_j dx_j) + \sum_{j=1}^{n-i}\frac{1}{2}(p_j dq_j - q_j d p_j).
\end{gather}
Here $(x_1, \dots, x_i, y_1, \dots, y_i, p_1, \dots, p_{n-i}, q_1, \dots, q_{n-i})$ are coordinates of $\mathbb{R}^{2n}$.
Then, the Liouville vector field corresponding to $\lambda_i$ is the gradient vector field, with respect to the standard Euclidean metric, of the Morse function
\[f_i = \sum_{j=1}^i(y_j^2 - \frac{1}{2}x_j^2) + \sum_{j = 1}^{n-i}\frac{1}{4}(p_j^2 + q_j^2). \] 

Weinstein \cite{Weinstein} defined the notion of Weinstein handle as follows:
\begin{definition}
	\label{def Weinstein handle}
	{\em The standard $2n$-dimensional Weinstein $i$-handle $H^i$} is a region of $(\mathbb{R}^{2n}, \omega_{std}, \lambda_i)$ satisfying that
	\begin{itemize}
		\item the region is bounded by 
		\[\left\{f_i^{-1}\left(-\tfrac{1}{2}\right)\right\} \text{  and  } \left\{F\left(\sum_{j=1}^i x_j^2, \sum_{j=1}^{n-i} p_j^2 + \sum_{j=1}^i y_j^2 +\sum_{j=1}^{n-i} q_j^2 = 0\right)\right\},\] 
		\item the region contains the origin point.
	\end{itemize} 
\end{definition}
\cite[Lemma 3.1]{Weinstein} proved that the choice of a specific function $F$ does not affect a standard handle up to symplectic completion. 

\begin{remark}
	\label{rmk comparision of handles and Weinsteion handles}
	It is easy to check that as a smooth manifold, the $2n$-dimensional standard Weinstein $i$-handle $H^i$ is diffeomorphic to a smooth $2n$-dimensional $i$-handle $h^i$. 
	In order to avoid confusion, we will use the uppercase $H$ (resp.\ the lower case $h$) for a Weinstein handle (resp.\ smooth handle).
\end{remark}

The following notion are necessarily to discuss the attachment of Weinstein handles:
\begin{definition}
	\label{def attaching spheres and regions of Weinsteion handle}
	\mbox{}
	\begin{enumerate}
		\item The {\em attaching region} of $H^i$ is the intersection of $\partial H^i$ and $f_i^{-1}(-\frac{1}{2})$. As similar to the case of smooth handles, let $\partial_R H^i$ denote the attaching region.
		\item The {\em attaching sphere} of $H^i$ is the intersection of $\partial_R H^i$ and the isotropic subspace 
		\[\{y_1 = \dots = y_i = p_1 = \dots = p_{n-i} = q_1 = \dots = q_{n-i} = 0 \} \subset \mathbb{R}^{2n}.\]
		Let $\partial_S H^i$ denote the attaching sphere. 
	\end{enumerate} 
\end{definition}

In order to attach a Weinstein handle $H$ to a Weinstein domain $W$, one needs a gluing map $\phi: \partial_R H \to \partial W$.
The difference from the smooth handle attachment is that one should consider the Weinstein structures on $H$ and $W$. 
Thus, the gluing map should preserve the contact structure, or more precisely, $\phi$ should be a contactomorphism between $\partial_R H$ and the image of $\phi$. 

\begin{remark}
	\label{rmk contact isotopic}
	Let $W$ be a Weinstein manifold. 
	Let us assume that there are two gluing maps $\phi_0, \phi_1 : \partial_R H \to \partial W$ which are contacto isotopic in the following sense:
	there exists a one-parameter family $f_t : W \stackrel{\sim}{\to} W$ of symplectomorphisms, such that $f_0$ is the identity and $\phi_1 = f_1 \circ \phi_0$.
	
	If $W_i$ denotes the Weinstein manifold obtained by attaching $H$ to $W$ with $\phi_i$, it is easy to check that $W_0$ and $W_1$ have symplectomorphic symplectic completions.  
	One can show that by using the one-parameter family $W_t$ of Weinstein manifolds which are obtained by attaching $H$ to $W$ with $f_t \circ \phi_0$. 
\end{remark}

\cite{Weinstein} showed that in order to attach a Weinstein handle $H^i$ of index $i$, it is enough to remember some local information, rather than the gluing map defined on the whole attaching region. 
More precise statement will appear at the last part of the present subsection.

The local information consist of a pair of an isotropic $(i-1)$ sphere $\Lambda$, which the attaching sphere of $H^i$ will be attached along, and a trivialization of the ``conformal symplectic normal bundle of $\Lambda$''. 
In the rest of Section \ref{subsect Weinstein handle}, we review the notion of conformal symplectic normal bundle. 

Let $(X, \xi)$ be a $(2n-1)$-dimensional contact manifold where $\xi$ is the given contact structure.
(Or one could consider a $2n$-dimensional Weinstein domain and let $X$ be the boundary of $W$.) 
If $\alpha$ is a contact form on $X$, then, it is well-known that $(\xi_x, d\alpha)$ is a symplectic vector space.

Let $\Lambda$ be an isotropic $(i-1)$-dimensional sphere in $X$.
Then, $T_x \Lambda$ is an isotropic subspace of a symplectic vector space $(\xi_x, d\alpha)$. 
Thus, if $T_x\Lambda^{\perp'}$ means the symplectic dual of $T_x\Lambda$, i.e., 
\[T_x\Lambda^{\perp'} := \{ v \in \xi_x \hspace{0.2em} | \hspace{0.2em} d \alpha(v, w) = 0 \text{  for all  } w \in T_x\Lambda\},\]
then, 
\[T_x\Lambda \subset T_x \Lambda^{\perp'}.\]

One can easily check that the quotient 
\begin{gather}
	\label{eqn def of csn}
	T\Lambda^{\perp'}/T\Lambda 	
\end{gather}
is a $(2n-2i)$-dimensional vector bundle over $\Lambda$ which carries a conformal symplectic structure naturally induced from $d \alpha$. 

\begin{definition}
	\label{def csn}
	The quotient in Equation \eqref{eqn def of csn} is called {\em the conformal symplectic normal bundle} of $\Lambda$.
	Let $CSN(\Lambda)$ denote the conformal symplectic normal bundle of $\Lambda$.
\end{definition}

The result of \cite{Weinstein} is to determine a contact isotopy class of a gluing map $\phi: \partial_R H \to X$ from a pair of $\Lambda$ and $CSN(\Lambda)$.
Thus, one could attach a Weinstein handle from the information given by the pair $(\Lambda, CSN(\Lambda))$ uniquely up to symplectomorphic symplectic completion.
Remark \ref{rmk contact isotopic} explains briefly how the contact isotopy class induces the uniqueness. 

Conversely, if there is a gluing map $\phi: \partial_R H \to X$, then $\phi$ induces an isotropic sphere $\Lambda := \phi(\partial_S H)$ and the differential $D\phi$ induces a trivialization of $CSN(\Lambda)$, which the pair recovers the contact isotopy class of $\phi$. 

\subsection{Weinstein handle decomposition}
\label{subsect Weinstein handle decomposition}
It is well-known that every Weinstein domain can be broken down into Weinstein handles, or equivalently, every Weinstein domain admits a Weinstein handle decomposition.
In Section \ref{subsect Weinstein handle decomposition}, we defined the notion of Weinstein handle decomposition that we use in the present paper. 

We recall that Definition \ref{def handle decomposition} defines a handle decomposition of $M$ as a collection of handles and gluing information of them.
In other words, a handle decomposition of $M$ explains how to construct $M$ as an attachment of handles.
In the context, constructing $M$ actually means that constructing a smooth manifold which is diffeomorphic to $M$, i.e., Definition \ref{def handle decomposition} is defined up to diffeomorphisms. 

As similar to Definition \ref{def handle decomposition}, we define a handle decomposition of a Weinstein domain $W$ as a collection of Weinstein handles together with gluing information.
Thus, a Weinstein handle decomposition of $W$ constructs a Weinstein domain by gluing Weinstein handles, which is equivalent to $W$. 
Before defining the notion of a Weinstein handle decomposition, we discuss which equivalence we consider in the current paper. 

A technical difficulty of studying Weinstein domains arises from the incompleteness of Weinstein domains. 
In order to resolve the difficulty, one could take the symplectic completions of them.
For more details, we refer the reader to \cite[Section 11]{CE}.
Based on this, we define the equivalence as follows:
\begin{definition}
	\label{def equivalence of Weinstein domains}
	We say that two Weinstein domains are {\em equivalent} to each other if their symplectic completions are exact symplectomorphic.
\end{definition}
We note that if two finite type Weinstein manifolds are symplectomorphic, then they are exact symplectomorphic by \cite[Theorem 11.2]{CE}.

\begin{definition}
	\label{def Weinstein handle decomposition}
	By a {\em Weinstein handle decomposition} of a Weinstein domain $W$, we mean a finite, ordered set of Weinstein handles ${H_0, \dots, H_m}$ together with the injective maps $\Phi_i : \partial_S H_i \to \partial (\cup_{j=0}^{i-1} H_j)$ whose images are isotropic spheres, and trivializations of the conformal symplectic normal bundle of $\Phi_i(\partial_S H_i)$ satisfying the following:
	\begin{itemize}
		\item $H_0$ is the unique index $0$-handle;
		\item there is a natural number $N$ such that for $i \leq N$ (resp.\ $i >N$), $H_i$ is subcritical (resp.\ critical), i.e., $\ind(H_i) < n $ (resp.\ $\ind(H_i) = n$);
		\item $\cup_{i=0}^m H_i$ and $W$ have symplectomorphic symplectic completions.
	\end{itemize}
\end{definition}
We note that the gluing information in Definition \ref{def handle decomposition} are given by gluing maps, defined on the whole attaching regions of each handle.
However, in Definition \ref{def Weinstein handle decomposition}, the gluing information are given as maps on attaching spheres and trivializations of the conformal symplectic normal bundles.

\subsection{Lefschetz fibration}
\label{subsect Lefschetz fibration}
We move on to our main interest, i.e., Lefschetz fibrations.
Since the definition and its properties are well-known in the literature, we briefly describe it in here. 
For more details, see \cite{ Seidel1, Maydanskiy, BEE}, etc.

To define the notion of Lefschetz fibration, we need some preparations. 
First, we see the unit disk $\mathbb{D}^2$ as a Liouville domain equipped with the one form $\mu = \tfrac{1}{2} \left(xdy - ydx\right) = \tfrac{1}{2}r^2 d\theta$. 
We also give the standard complex structure on $\mathbb{D}^2$, and then the Liouville completion of $\mathbb{D}^2$ is $\mathbb{C}$ equipped with the standard symplectic structure. 
One can find the following definition in \cite[Section 3.1]{Maydanskiy}.

\begin{definition}
	\label{def Lefschetz fibration}
	Let $(W,\omega = d \lambda)$ be an exact symplectic manifold with corners. 
	A {\em Lefschetz fibration} on $W$ is a map $\pi : W\to \mathbb{D}^2$ is a $J$-holomorphic map, where $J$ is a $\omega$-compatible almost complex structure, satisfying the following properties:
	\begin{itemize}
		\item ({\em Transversality to $\partial \mathbb{D}^2$.})
		At every point $x \in W$ such that $y :=\pi(x) \in \partial \mathbb{D}^2$, 
		\[T_y\mathbb{D}^2 = T_y\left(\partial \mathbb{D}^2\right) \oplus D_\pi\left(T_xW\right).\]
		Because of this condition, $\pi^{-1}(\partial \mathbb{D}^2)$ is a boundary stratum of $W$ of codimension $1$.
		The part of $\partial W$ will be called the {\em vertical boundary} and denoted by $\partial^v W$ , and the union of boundary faces not contained in $\partial^v W$ will be called the {\em horizontal boundary} and denoted by $\partial^h W$.
		\item ({\em Regularity along $\partial^h W$.}) 
		If $H$ is a boundary face of $W$ such that $H \not\subset \partial^v W$, then $\pi|_H:H \to \mathbb{D}^2$ is a smooth fibration. 
		\item ({\em Horizontality of $\partial^h W$ with respect to the symplectic connection.})
		At any point $x \in W$, let $T_x^v W := \operatorname{ker}\left(D \pi_x\right)$. 
		Away from critical points, since $\pi$ is $J$-holomorphic, the symplectic complement, denoted by $T_x^h W$, defines a connection. 
		We require that if $x \in H$ for a boundary face $H$ such that $H \not\subset \partial^v W$, then $T_x^h W$ is contained in $T_x H$. 
		\item ({\em Lefschetz type critical points.})
		There are only finitely many points where $D\pi$ is not surjective, and for any such critical point $p$, there exist complex Darboux coordinates $(z_1, \dots, z_n)$ centered at $p$ so that $\pi(z_1, \dots, z_n) = \pi(p) + z_1^2 + \dots + z_n^2$. 
		Moreover, there is at most one critical point in each fiber of $\pi$.
	\end{itemize}
\end{definition}
We note that it would be more precise to use the term ``exact Lefschetz fibration" in Definition \ref{def Lefschetz fibration}. 
However, in this paper, this is the only type which we considered here. 
Thus, we omit the adjective for convenience.
 
Definition \ref{def Lefschetz fibration} is classical, but \cite{GP} suggested an alternative definition.

\begin{definition}
	\label{def abstract Lefschetz fibration}
	An {\em abstract Weinstein Lefschetz fibration} is a tuple 
	\[W = (F: L_1, \dots, , L_m)\]
	consisting of a Weinstein domain $F^{2n-2}$ (the ``central fiber'') along with a finite sequence of exact parameterized Lagrangian spheres $L_1, \dots, L_m \subset F$ (the ``vanishing cycles'').
\end{definition}

Definitions \ref{def Lefschetz fibration} and \ref{def abstract Lefschetz fibration} are interchangeable. 
In the rest of Section \ref{subsect Lefschetz fibration}, we explain how to obtain a Lefschetz fibration of a Weinstein manifold when an abstract Weinstein Lefschetz fibration is given briefly. 
For more details on the equivalence of Definitions \ref{def Lefschetz fibration} and \ref{def abstract Lefschetz fibration}, we refer the reader to \cite[Section 8]{BEE} or \cite[Section 6]{GP}.

Let $W = (F: L_1, \dots, , L_m)$ be a given abstract Weinstein Lefschetz fibration.
Then, one can construct a Weinstein domain as follows:
First, we consider the product of $F$ and $\mathbb{D}^2$. 
Then, the vertical boundary $F \times \partial \mathbb{D}^2$ admits a natural contact structure. 
Moreover, the vanishing cycle $L_i$ can be lifted to a Legendrian $\Lambda_i$ near $2 \pi i/m \in S^1$. 
The lifting procedure is given in Section \ref{subsubsection some contact topology}.
We note that by assuming that the disk $\mathbb{D}^2$ has a sufficiently large radius, one could assume that the projection images of $\Lambda_i$ onto the $S^1$ factor are disjoint to each other. 
Finally, one could attach critical Weinstein handles along $\Lambda_i$ for all $i = 1, \dots, m$.
Then, the completion of the resulting Weinstein domain admits a Lefschetz fibration satisfying that the regular fiber is $F$, and that there are exactly $m$ singular values near $2 \pi i /m \in S^1$.

\clearpage

\part{Lefschetz fibrations on cotangent bundles}
\label{part Lefschetz fibrations on cotangent bundles}
The main goal of the first part is to investigate the idea given in Section \ref{subsect motivation} and to prove Theorem \ref{thm rough statement}. 
We review the idea in Section \ref{section the idea} and prove Theorem \ref{thm rough statement} in Sections \ref{section the proof of Theorem main} and \ref{section examples}. 
Before those, in Section \ref{section preparations}, we define some notions and set notation. 

\section{Preparations}
\label{section preparations}
Before discussing our construction of Lefschetz fibrations on cotangent bundles, we prepare the later sections that prove Theorem \ref{thm rough statement}.
In the first two subsections of the section, we discuss an algorithm producing a Weinstein handle decomposition of a cotangent bundle $T^*M$ from a handle decomposition of a smooth manifold $M$.
In the last subsection, Section \ref{subsection recovering Lagrangian skeleta}, we introduce technical terms and settings. 
 
\subsection{Attaching Legendrian}
\label{subsect attaching Legendrian}
We review that gluing information of a Weinstein handle consist of two things, an attaching map (or equivalently an attaching sphere) and a trivialization of the conformal symplectic normal bundle of the attaching sphere. 
In Section \ref{subsect attaching Legendrian}, we introduce a notion that combines these two gluing information.
The notion is called attaching Legendrian.  
The {\em attaching Legendrian} (resp.\ {\em core Lagrangian}) is defined on a standard Weinstein handle $H^i \subset \mathbb{R}^{2n}$, where $\mathbb{R}^{2n} = \mathbb{R}^{2k} \times \mathbb{R}^{2(n-k)}$ is coordinated by 
\[(x_1, \dots, x_k, y_1, \dots, y_k, p_1, \dots, p_{n-k}, q_1, \dots, q_{n-k}),\]
as we did in Equation \eqref{eqn liouville form}.

\begin{definition}
	\label{def attaching Legendrian}
	\mbox{}
	\begin{enumerate}
		\item 	The {\em attaching Legendrian $\partial_L H^k$} of the standard $2n$-dimensional Weinstein $k$-handle $H^k$ is the intersection of $\partial_R H^k$ and the region 
		\[\{y_1 = \dots = y_k = 0 = q_1 = \dots = q_{n-k}\}.\]
		\item The {\em core Lagrangian} of the standard $2n$-dimensional Weinstein $k$-handle $H^k$ is the intersection of the handle and the region
		\[\{y_1 = \dots = y_k = 0 = q_1 = \dots = q_{n-k}\}.\]
	\end{enumerate}
\end{definition}

\begin{remark}
	\label{rmk attaching Legendrian is not intrinsic}
	We note that the Liouville vector field has only one zero in a Weinstein handle, and that the attaching sphere is the boundary of the stable manifold of the unique zero with respect to the Liouville vector flow. 
	Thus, the attaching sphere of a Weinstein handle could be defined on the Weinstein handle without using coordinates.
	
	Differently from the attaching sphere, in order to define the notions of attaching Legendrian and core Lagrangian, a choice of coordinate charts is necessarily. 
	Thus, for a general Weinstein handle $H$, $\partial_L H$ is defined with respect to an identification with $H$ and the standard handle. 
	For convenience, we use the notions of attaching Legendrians and core Lagrangians without mentioning a choice of identifications.
\end{remark}

\begin{lemma}
	\label{lemma attaching legendrian}
	Let $X$ be a $(2n-1)$-dimensional contact manifold. 
	If there is a map $\phi: \partial_L H^k \hookrightarrow X$, where $H^k$ is the standard $2n$-dimensional Weinstein $k$-handle such that 
	\begin{itemize}
		\item $\phi$ is an embedding, and 
		\item $\mathrm{Im}(\phi)$ is a Legendrian in $X$, 
	\end{itemize}
	then $\phi$ induces a trivialization on $CSN(\Lambda)$ where $\Lambda := \phi(\partial_S H^k)$.
\end{lemma}
\begin{proof}
	Simply, \cite[Proposition 4.2]{Weinstein} proves Lemma \ref{lemma attaching legendrian}.
	
	In order to give more precise proof, let us note that, for any Legendrian $\Lambda$ in a contact manifold, there is a neighborhood that is contactomorphic to a neighborhood of $\Lambda$ in the Jet $1$ bundle of $\Lambda$. 
	We also note that $\phi$ identifies two Legendrians $\partial_L H^k$ and $\textrm{Im}\phi$.
	Thus, they have neighborhoods that are contactomorphic to each other. 
	Since $\partial_L H^k$ admits a natural trivialization induced from the coordinate of the standard handle.
	It induces a trivialization of $CSN(\Lambda)$. 
\end{proof}

Lemma \ref{lemma attaching legendrian} concludes that if there is a map $\phi: \partial_LH^k \to \partial_\infty W$ satisfying the setting in Lemma \ref{lemma attaching legendrian}, then the map $\phi$ encodes gluing information of a Weinstein handle to a Weinstein manifold $W$. 
Moreover, Lemma \ref{lemma Legendrian isotopy} shows that it is enough to consider the Legendrian isotopy class of the image of the attaching Legendrian. 

\begin{lemma}
	\label{lemma Legendrian isotopy}
	Let $W$ be a Weinstein manifold and there is a map $\phi : \partial_LH^k \to \partial_\infty W$ satisfying the conditions in Lemma \ref{lemma attaching legendrian}.
	Let $\Lambda_t$ be an Legendrian isotopy connecting $\Lambda_0 := \phi(\partial_L H^k)$ and $\Lambda_1$. 
	If $W_i$ denotes the Weinstein domain obtained by attaching $H^k$ along $\Lambda_i$ for $i= 0, 1$, then $W_0$ and $W_1$ have symplectomorphic completions. 
\end{lemma}
\begin{proof}
	On the contact manifold $\partial_\infty W$, the Legendrian isotopy $\Lambda_t$ can be extended to the contact isotopy $\psi_t$ of $\partial_\infty W$. 
	For the extension procedure, we refer the reader to \cite[Section 2.5]{Geiges}. 
	By \cite[Lemma 12.5]{CE}, there is a Liouville form on $\partial_\infty W \times [0,1]$ such that the corresponding holonomy from $\partial_\infty W \times \{0\}$ to $\partial_\infty W \times \{1\}$ is the contact isotopy $\psi_1$. 
	Since a Weinstein homotopic change does not affect on the equivalence class of the resulting symplectic manifold, it completes the proof. 
\end{proof}

\subsection{Weinstein handle decomposition of $T^*M$}
\label{subsect Weinstein handle decomposition of T^*M}
We prove Lemma \ref{lem handle decomposition of cotangent bundle} in Section \ref{subsect Weinstein handle decomposition of T^*M}, which gives us a Weinstein handle decomposition of $T^*M$ from a handle decomposition of $M$. 
Before stating and proving Lemma \ref{lem handle decomposition of cotangent bundle}, we set notation first.
\vskip0.2in

\noindent{\em Notation.}
A handle decomposition $D$ of $M$ means an ordered collection of handles $\{h_0, \dots, h_m \}$ together with the gluing information as defined in Definition \ref{def handle decomposition}.
Let 
\[\phi_i : \partial_R h_i \to \partial (\cup_{j < i} h_j). \]
denote the gluing map for $h_i$.

For a given handle decomposition $D$ of $M$, we consider a collection of Weinstein handles $\{H_0, \dots, H_m\}$ such that 
\[\dim(H_i) = 2n \text{  and  } \ind(H_i) = \ind(h_i)\]
for all $i = 0, \dots, m$. 
Then, one can easily construct an embedding $\iota_i : h_i \hookrightarrow H_i$ such that 
\begin{enumerate}
	\item $\iota_i(h_i)$ is the core Lagrangian of $H_i$, and
	\item $\iota_i$ sends $\partial_R h_i$ to the attaching Legendrian $\partial_L H_i$ of $H_i$.
\end{enumerate}
The core Lagrangian and the attaching Legendrian are defined in Definition \ref{def attaching Legendrian}. 

As discussed in Section \ref{subsect attaching Legendrian}, the gluing information for Weinstein handles can be given by maps defined on the attaching Legendrians of Weinstein handles. 
Then, the following map 
\[\Phi_i : \partial_L H_i \stackrel{\iota_i^{-1}}{\to} \partial_R h_i \stackrel{\phi_i}{\to} \partial (\cup_{j<i} h_j) \stackrel{\cup \iota_j}{\to} \partial (\cup_{j<i} \partial H_j)\]
gives the attaching information by Lemma \ref{lemma attaching legendrian}. 

Let $W_D$ denote the resulting Weinstein domain by attaching Weinstein handles in $\mathcal{D}$.
Now we can state Lemma \ref{lem handle decomposition of cotangent bundle}.
 
\begin{lemma}
	\label{lem handle decomposition of cotangent bundle}
	The cotangent bundle $T^*M$ and $W_D$ are symplectomorphic up to symplectic completions. 
\end{lemma}
\begin{proof}
	In order to prove Lemma \ref{lem handle decomposition of cotangent bundle}, we observe that the Lagrangian skeleta of $W_D$ is a smooth manifold $M$. 
	To be more precise, we note that $W_D$ admits a specific Liouville form induced from the Weinstein handle decomposition $\mathcal{D}$.
	According to the Liouville form, it is easy to observe that the corresponding Lagrangian skeleta of $W_D$ is the union of all core Lagrangians of handles in $W_D$, i.e., $M$. 
	Since the Lagrangian skeleta is a smooth manifold, $W_D$ is equivalent to the cotangent bundle of the Lagrangian skeleta, i.e., $T^*M$.  
\end{proof}

\subsection{Recovering Lagrangian skeleton}
\label{subsection recovering Lagrangian skeleta}
We start the section by recalling the notion of {\em arboreal Lagrangians} which was first defined in \cite{Nadler17}.
In the paper, we will use the definitions given in \cite{Alvarez-Eliashberg-Nadler23}.

\begin{definition}[Definition 1.1 of \cite{Alvarez-Eliashberg-Nadler23}]
	\label{def arboreal singularities}
	{\em Arboreal Lagrangian (resp.\ Legendrian) singularities} form the smallest class $\mathrm{Arb}_n^{symp}$ (resp.\ $\mathrm{Arb}_n^{cont}$) of germs of closed isotropic subsets in $2n$-dimensional symplectic (resp.\ $(2n+1)$-dimensional contact) manifolds such that the following properties are satisfied:
	\begin{enumerate}
		\item[(i)] (Invariance) $\mathrm{Arb}_n^{symp}$ is invariant with respect to symplectomorphisms and $\mathrm{Arb}_n^{cont}$ is invariant with respect to contactomorphisms. 
		\item[(ii)] (Base case) $\mathrm{Arb}_0^{symp}$ consists $pt = \mathbb{R}^0 \subset T^*\mathbb{R}^0 = pt$. 
		\item[(iii)] (Stabilizations) If $L \subset (X, \omega)$ is in $\mathrm{Arb}_n^{symp}$, then the product $L \times \mathbb{R} \subset (X \times T^*\mathbb{R}, \omega + dp \wedge dq)$ is in $\mathrm{Arb}_{n+1}^{symp}$. 
		\item[(iv)] (Legendrian lifts) If $L \subset T^*\mathbb{R}^n$ is in $\mathrm{Arb}_n^{symp}$, then its Legendrian lift $\widehat{L} \subset J^1\mathbb{R}^n$ is in $\mathrm{Arb}_n^{cont}$. 
		\item[(v)] (Liouville cones) Let $\Lambda_1, \dots, \Lambda_k \subset S^*\mathbb{R}^n$ be a finite disjoint union of arboreal Legendrian germs form $\mathrm{Arb}_{n-1}^{cont}$ centered at points $z_1, \dots, z_k \in S^*\mathbb{R}^n$. 
		Let $\pi: S^*\mathbb{R}^n \to \mathbb{R}^n$ be the front projection. 
		Suppose 
		\begin{itemize}
			\item $\pi(z_1) = \dots = \pi(z_k)$;
			\item For any $i$, and smooth submanifold $Y \subset \Lambda_i$, the restriction $\pi|_Y: Y \to \mathbb{R}^n$ is an embedding (or equivalently, an immersion, since we only consider germs);
			\item For any distinct $i_1, \dots, i_l$, and any smooth submanifolds $Y_{i_1} \subset \Lambda_{i_1}, \dots, Y_{i_l} \subset \Lambda_{i_l}$, the restriction $\pi_{Y_{i_1} \cup \dots \cup Y_{i_l}}: Y_{i_1} \cup \dots \cup Y_{i_l} \to \mathbb{R}^n$ is self-transverse. 
		\end{itemize}
		Then the union $\mathbb{R}^n \cup C(\Lambda_1) \cup \dots \cup C(\Lambda_k)$ of the Liouville cones with the zero-section form an arboreal Lagrangian germs form $\mathrm{Arb}_n^{symp}$. 
	\end{enumerate}
	With the above classes defined, we can also allow boundary by additionally taking the product $L \times \mathbb{R}_{\geq 0} \subset \left(X \times T^*\mathbb{R}, \omega +dp \wedge dq\right)$ for any arboreal Lagrangian $L \subset (X, \omega)$, and similarly for arboreal Legendrians. 
\end{definition}

One of the main results in \cite{Alvarez-Eliashberg-Nadler23} is to show that for fixed dimension $n$, the classes of arboreal singularities contain only finitely many local models up to symplectomorphisms or contactomorphisms.
Moreover, to each class in $\mathrm{Arb}_n^{symp}$, one can assign a {\em signed rooted tree} $\mathcal{T}$. 
\begin{theorem}[Theorem 1.2 of \cite{Alvarez-Eliashberg-Nadler23}]
	\label{thm arboreal singularity}
	If two arboreal Lagrangian singularities $L \subset (X, \omega), L' \subset (X', \omega')$ of the class $\mathrm{Arb}_n^{symp}$ have the same dimension and signed rooted tree $\mathcal{T}$, then there is (the germ of) a symplectomorphism $(X, \omega) \simeq (X', \omega')$ identifying $L$ and $L'$.
\end{theorem}

A Lagrangian subset, possibly with singularities, is called {\em arboreal Lagrangian} if its singularities are arboreal in the sense of Definition \ref{def arboreal singularities}. 
\begin{definition}[Definition 3.1 of \cite{Alvarez-Eliashberg-Nadler23}]
	\label{def arboreal Lagrangian}
	A closed subset $L \subset X$ of a $2n$-dimensional symplectic manifold $(X, \omega)$ is called an {\em arboreal Lagrangian} if the germ of $(X, L)$ at any point $\lambda \in L$ is symplectomorphic to the germ of the pair $(T^*\mathbb{R}^m \times T^*\mathbb{R}^{n-m}, L_\mathcal{T} \times \mathbb{R}^{n-m})$ at the origin, for a signed rooted tree $\mathcal{T}$ having $(m+1)$-many vertices (including a root and $m \leq n$). 
\end{definition}
We refer the reader to \cite[Section 2.2]{Alvarez-Eliashberg-Nadler23} for the local models assigned to a signed rooted tree $\mathcal{T}$. 

Now, we introduce some technical definitions that we need to prove Theorem \ref{thm rough statement}. 
\begin{definition}
	\label{def recovering skeleton} 
	Let $W$ be a Weinstein manifold of dimension $2n$, equipped with a Lagrangian skeleton $\mathbb{L}$. 
	The Lagrangian skeleton is {\em recovering} if the following hold:
	\begin{enumerate}
		\item[(i)] (Global condition) $\mathbb{L}$ is a disjoint union of smooth Lagrangians $\{L_0, \dots, L_k\}$ such that 
		\begin{itemize}
			\item $L_0$ is a closed Lagrangian without boundary; 
			\item For all $i =1, \dots, k$, the boundary of the closure $\overline{L}_i$ should be contained in $\cup_{j=0}^{i-1} L_i$, i.e., $\partial \overline{L}_i \subset \cup_{j=0}^{i-1} L_j$.  
		\end{itemize}
		\item[(ii)] (Local condition) $\mathbb{L}$ is an arboreal Lagrangian such that the germ of $(W, \mathbb{L})$ at any point $\lambda \in \mathbb{L}$ is symplectomorphic to the germ of the pair $(T^*\mathbb{R}^m \times T^*\mathbb{R}^{n-m}, L_\mathcal{T} \times \mathbb{R}^{n-m})$ at the origin, for a positively rooted $A_m$-tree $\mathcal{T}$ having  $(m+1)$-many vertices (including a root and $m \geq n$).
	\end{enumerate}
\end{definition}

\begin{remark}
	\label{rmk recovering skeleton}
	\mbox{}
	\begin{enumerate}
		\item We first note that if $\mathbb{L} = \cup_{i=0}^k L_i$ is a recovering Lagrangian skeleton of a Weinstein manifold $W$, then one can {\em recover} $W$ by gluing cotangent bundles $\left\{T^*\widetilde{L}_i | i = 0, \dots, k\right\}$ where $\widetilde{L}_j$ is defined to be the complement of a small neighborhood of $\partial \overline{L}_i$ from $\overline{L}_i$. 
		Recall that $\overline{L}_i$ could be a Lagrangian with {\em corners}, but $\widetilde{L}_i$ could be a manifold with boundary, but not corner. 
		Then, one can easily see that $T^*\widetilde{L}_i$ is a Weinstein sector. 
		Moreover, a small neighborhood of $\mathbb{L}$ in $W$, which is equivalent to $W$, can be obtained by gluing $T^*\widetilde{L}_i$s along Weinstein hypersurfaces {\em inductively}. 
		See \cite{Avdek21, Eliashberg18} for the gluing procedure.
		
		More precisely, we could obtain a Weinstein manifold $\cup_{j=0}^i T^*\widetilde{L}_j$ {\em inductively}. 
		The base step is $T^*L_0$ since $L_0$ is a closed Lagrangian. 
		Using the local condition of Definition \ref{def recovering skeleton} and the condition that $\partial \overline{L}_{i+1} \subset \cup_{j=0}^{i-1} L_j$, one can define a Legendrian $\Lambda_{i+1}$ of the asymptotic contact manifold of $\cup_{j=0}^i T^*\widetilde{L}_j$ as follows: 
		For every $\lambda \in \partial \overline{L}_{i+1}$, we have a local coordinate modeled by a singed rooted tree $A_m$ for some $m$. 
		The local model gives us a positive conormal direction at every $\lambda \in \partial \overline{L}_{i+1}$, and the lift of $\partial \overline{L}_{i+1}$ in the conormal direction will define a Legendrian $\Lambda_{i+1}$. 
		Then, one can glue $T^*\widetilde{L}_{i+1}$ to $\cup_{j=0}^i T^*\widetilde{L}_j$ along $\Lambda_{i+1}$. 
		
		\item The above item (1) explains the geometric meaning of Definition \ref{def recovering skeleton} (i) (Global condition). 
		The meaning of the second condition, i.e., local condition, is that one can use the local model to get an exact Lagrangian smoothing of a Lagrangian $A_m$-singularities. 
		
		We sketch below how one could get a smoothed version, and for details, see \cite[Section 2]{Alvarez-Eliashberg-Nadler20}.
		An (possibly singular) arboreal singularity of $A_m$-type lives in $T^*\mathbb{R}^m$.
		It has a Legendrian boundary on the asymptotic boundary of $T^*\mathbb{R}^m$.
		One can consider the Legendrian front projection and its a regular neighborhood in $\mathbb{R}^{m+1}$.
		The front projection is a hypersurface in $\mathbb{R}^{m+1}$, and one can lift the hypersurface to a smooth Legendrian in $\mathbb{R}^{2m+1}$. 
		It has a Lagrangian projection in $\mathbb{R}^{2m} = T^*\mathbb{R}^m$ that serves as a smoothed version of Lagrangian $A_m$-singularity. 
		
		\item We note that in Definition \ref{def recovering skeleton}, we require $\mathbb{L}$ to have a singularity of a {\em specific} type, i.e., positively rooted $A_m$-types. 
		Even if $\mathbb{L}$ have singularities of other types, we could recover $W$ from $\mathbb{L}$ as we did in (1) if we have a corresponding {\em conormal} direction.  
		But we only allow the simplest type, i.e., $A_m$-type, since we only need the type in the proof of Theorem \ref{thm rough statement}.
	\end{enumerate}
\end{remark}

In Remark \ref{rmk recovering skeleton} (1), we come up with a Legendrian $\Lambda_{i+1}$ form a subset $\partial \overline{L}_{i+1} \subset \cup_{j=0}^{i-1} L_j$ in order to attach $T^*\widetilde{L}_{i+1}$ to $\cup_{j=0}^i T^*\widetilde{L}_j$.
In the proof of Theorem \ref{thm rough statement}, we will attach a Weinstein sector to a Weinstein manifold equipped with a recovering skeleton.
Similar to Remark \ref{rmk recovering skeleton} (1), we will use a subset of the recovering skeleton, which satisfies the following condition:
\begin{definition}
	\label{def transversally embedded} 
	Let $W$ be a Weinstein manifold of dimension $2n$, equipped with a recovering Lagrangian skeleton $\mathbb{L} = \cup_{i=0}^k L_i$. 
	\begin{enumerate}
		\item A subset $S$ of $\mathbb{L}$ is an {\em embedded submanifold (of codimension $r$)} if there exists an injective continuous map from a manifold of dimension $(n-r)$ whose image is $S$. 
		In other words, there exists an injective continuous map 
		\[\iota: M^{n-r} \to \mathbb{L}\]
		from an $(n-r)$-dimensional manifold $M$, such that $\iota(M) = S$. 
		\item An embedded submanifold $S \subset \mathbb{L}$ of codimension $r$ is said to be {\em transversally embedded}, if for any $I_0 \subset \{0, \dots, k\}$, the intersection of $S$ and $L(I_0):= \cap_{i \in I_0} \overline{L}_i$ is a submanifold of $L(I_0)$ such that the dimension of submanifold $S \cap L(I_0)$ is $\dim L(I_0) -r$. 
	\end{enumerate}
\end{definition}
We note that the notion of embedded submanifold uses the word ``submanifold'' incorrectly (in a rigorous sense), since an embedded submanifold $S$ is not necessarily to be a smooth manifold. 
More precisely, it could be a subset that is {\em homeomorphic} to a manifold, without a manifold structure. 
We note also that $L(I_0)$ is a submanifold of $L_{i_0}$ where $i_0$ is the minimal element of $I_0$ because of the local condition in Definition \ref{def recovering skeleton}. 
The dimension of $L(I_0)$ is $(n - |I_0|)$ where $|I_0|$ means the cardinal of $I_0$. 

For later use, we prove the following lemma:
\begin{lemma}
	\label{lem disk bundle}
	Let $W$ be a Weinstein manifold of dimension $2n$, equipped with a recovering Lagrangian skeleton $\mathbb{L}$. 
	Let $S$ be a transversally embedded submanifold of positive codimension $r > 0$ so that there exists an injective map $\iota: M^{n-r} \to \mathbb{L}$. 
	Then, there exists an extension of $\iota$, denoted by $\iota$ again,
	\[\iota: \mathbb{D}^r \times M \to \mathbb{L},\] 
	satisfying the following:
	\begin{enumerate}
		\item The extended $\iota$ is also a continuous injective map and $\iota(\mathbf{0}, x) = \iota(x)$ where $\mathbf{0}$ is the origin of the disk $\mathbb{D}^r$ and $x \in M$. 
		\item For any $x \in M$ and $p \in \mathbb{D}^r$, if $\iota(x) \in L(I_0)$, then $\iota(p,x) \in L(I_0)$. 
		We also note that if $|I_0| = m+1$, then near $x \in L(I_0)$, $\mathbb{L}$ is locally modeled by $L_{\mathcal{A}_{m+1}} \times \mathbb{R}^{n-m}$ where $L_{\mathcal{A}_{m+1}}$ is a local model in $T^*\mathbb{R}^m$ by a positively rooted $A_{m+1}$-tree. 
		Moreover, $L(I_0)$ is locally modeled by $\mathbf{0} \times \mathbb{R}^{n-m}$. 
		Then, $D \iota \partial p$ is tangential to $\mathbf{0} \times \mathbb{R}^{n-m}$, where $\partial p$ denotes a tangent vector of $\mathbb{D}^r \times M$ in a $\mathbb{D}^r$-direction.
	\end{enumerate}
\end{lemma}
\begin{proof}
	Let us fix a metric $g$ on $W$. 
	Then, by choosing a small neighborhood of $S$ in $\mathbb{L}$, one can find a normal disk bundle of $S$ satisfying the conditions of Lemma \ref{lem disk bundle}.
	We note that every fiber of the disk bundle is $\mathbb{D}^r$ because $S$ is a transversally embedded submanifold of $\mathbb{L}$ having codimension $r$. 
\end{proof}

\begin{remark}
	\label{rmk transversally embedded}
	We note that if one restrict $\iota$ in Lemma \ref{lem disk bundle} to the boundary of $\partial \mathbb{D}^r \times M \simeq S^{r-1} \times M$, the restriction of $\iota$ provides a transversally embedded submanifold of $\mathbb{L}$ of {\em codimension $1$}. 
	Since it is codimension $1$, one can specify a {\em positive conormal} direction along the embedded submanifold, so that the union of $\mathbb{L}$ with the positive conormal of the embedded submanifold of codimension $1$ is an arboreal Lagrangian whose singularities are of $A_k$-types for some $k$. 
	We will use this in the proof of Theorem \ref{thm rough statement} in Section \ref{section the proof of Theorem main}.
\end{remark}

\section{The idea}
\label{section the idea}
Now, we state a technical statement of Theorem \ref{thm rough statement}.
The statement will be given in Section \ref{subsect technical statement}. 
Before that, we explain the main idea in Section \ref{subsection the main idea}. 

\subsection{The main idea}
\label{subsection the main idea}
Section \ref{subsubsection idea of the first theorem} briefly explained the main idea of Theorem \ref{thm rough statement}.
We review the idea in Section \ref{section the idea} with more details.  

The idea is motivated by the fact that the following two induce a Weinstein handle decomposition of $W$: 
\begin{itemize}
	\item The first one is a Lefschetz fibration $\pi : W \to \mathbb{C}$, and 
	\item the other is a Weinstein handle decomposition of a regular fiber of $\pi$.
\end{itemize}

We would like to go through the other direction.
In other words, we would like to construct a Lefschetz fibration of $W$ from a Weinstein handle decomposition of $W$. 
However, it is easy to find a counter example, i.e., a Weinstein handle decomposition {\em cannot} produce a Lefschetz fibration. 
We will give an example in Section \ref{subsect example : the case of T^*S^n}. 

Reminding the existence of counter examples, we look over necessary conditions first.
If a Weinstein handle decomposition $\mathcal{H}$ can produce a Lefschetz fibration, one can construct the Lefschetz fibration through the following two steps:
\begin{enumerate}
	\item[(i)] First, one can consider the union of all subcritical handles in $\mathcal{H}$.
	Let $W_{sub}$ denote the union of all subcritical handles. 
	Then, one can find a codimension 2 Weinstein manifold $F$ such that 
	\[W_{sub} \simeq F \times \mathbb{C},\]
	where $\mathbb{C}$ admits the standard radial structure, i.e., its symplectic $2$-form and Liouville $1$-form are given as $dx \wedge dy$ and $\tfrac{1}{2}\left(xdy - ydx\right)$, respectively. 
	We will simply use $F \times \mathbb{C}$ in order to indicate the subcritical part in the rest of the section.
	We note that $\simeq$ symbol means {\em Weinstein homotopic} between both sides. 
	\item [(ii)] Then, the critical handles in $\mathcal{H}$ should {\em respect} the product structure $W_{sub} \simeq F \times \mathbb{C}$.
	To be more precise, let $\{\Lambda_i\}_{i \in I}$ be the set of attaching Legendrian spheres of the critical handles and let $pr_2 : F \times \mathbb{C} \to \mathbb{C}$ be the projection map. 
	Then, $pr_2(\{\Lambda_i\}_{i \in I})$ should be a collection of pairwise disjoint intervals contained in $\partial_\infty \mathbb{C}$. 
	Moreover, the fiber $F$ must have a collection of exact Lagrangian spheres $\{L_i\}_{i \in I}$ such that, on the asymptotic boundary of $F \times \mathbb{C}$, $\Lambda_i$ and 
	\[\left\{\left(p, \theta_i(p)\right) | p \in \times \{\theta_i\} \right\}\]
	are Legendrian isotopic where $\theta_i : L_i \to S^1$ is a function.
	(More precisely, $\theta_i$ is obtained from the primitive function of the exact Lagrangian sphere $L_i \in F$.
	See Section \ref{subsubsection some contact topology} and Definition \ref{def legendrian lift} for more details.)
	In other words, $L_i$ is the vanishing cycle of the singular value corresponding to the critical handle. 
\end{enumerate}
We note that every Weinstein handle decomposition satisfies (i) by \cite{CE, Cieliebak}. 
Thus, in order to construct a Lefschetz fibration, it is enough to find a Weinstein handle decomposition passing (ii).

The next question is how to find the Lagrangian sphere $L_i \subset F$ for each critical handle. 
To answer the question, we observe how one can construct a Lagrangian skeleton of $W$.
The Lagrangian skeleton of $W$ corresponding to $\mathcal{H}$ is obtained by attaching the {\em closures of stable manifolds of the zeros in critical handles} to {\em the Lagrangian skeleton of $F$}.
We note that the stable manifolds of critical handles are defined by the corresponding Liouville structure on $W$, and the stable manifold of each critical handle should be an open disk of dimension $n$. 
In other words, if one has the collection of attaching maps, i.e., 
\begin{gather}
	\label{eqn attaching map 2}
	\text{Att}_i: S^{n-1} = \text{the boundary of the closure of the stable manifold of each critical handle} \to \text{the skeleton of  } F,
\end{gather}
then one can recover the skeleton of $W$. 

We expect that, if the maps $\text{Att}_i$ are {\em injective} maps, then by smoothing the images of $\text{Att}_i$ (note that the images could intersect singularities of the skeleton of $F$), one can obtain a Lagrangian sphere $L_i$ satisfying the above conditions. 
Based on this idea, we define the notion of {\em injective Weinstein handle decomposition}.
\begin{definition}
	\label{def injectiveness}
	For simplicity, we say that a Weinstein handle decomposition $\mathcal{H}$ is {\em injective}, if the maps $\text{Att}_i$ in \eqref{eqn attaching map 2} are injective for all critical handles in $\mathcal{H}$.
\end{definition}

When we have an injective Weinstein handle decomposition, then the expected Lefschetz fibration has the fiber given from the subcritical handles and vanishing cycles that are smoothed Lagrangian sphere $L_i$ mentioned above. 

To prove Theorem \ref{thm rough statement}, we first construct an injective Weinstein handle decomposition of a cotangent bundle. 
To do that, we recall that Lemma \ref{lem handle decomposition of cotangent bundle} gives an algorithm generating a Weinstein handle decomposition $W_D$ of a cotangent bundle $T^*M$ from a handle decomposition $D$ of the zero section $M$. 
One can easily check that if $D$ satisfies the property $(\star)$ below, then $W_D$ is injective. 

In order to describe the property $(\star)$, we note that there exists a Morse function on $M$ corresponding to $D$, or more clearly, there exists a Morse function such that the handle decomposition of $M$ induced from the Morse function is $D$. 
Each handles in $D$ has a unique zero of the gradient flow of the Morse function. 
Then, $M$ can be written as a union of stable manifolds (with respect to the gradient flow) of the zeros in handles of index $n$.
The property $(\star)$ can be stated as follows: 

\begin{itemize}
	\item[$(\star)$] For any handle of index $n$ in $D$, the closure of the corresponding stable manifold is homeomorphic to a disk $\mathbb{D}^n$. 
\end{itemize}

\begin{exmp}
	\label{Example 1}
	To find an example of $D$ such that $D$ does not satisfy $(\star)$, let us assume that a handle decomposition $D$ of $M$ has only one handle of index $n$. 
	Then, the closure of the stable manifold corresponding to the unique index $n$-handle should be $M$. 
	If $M$ is not a disk, then $D$ cannot satisfy $(\star)$. 
\end{exmp}

In order to find a handle decomposition of $M$ satisfying $(\star)$, we consider the following strategy. 
First, we start with any handle decomposition $D$ of $M$.
Then, we add a canceling pair of indices $(n-1, n)$-handles for each handles of index $< n$. 
The details will be given in Section \ref{subsect technical statement} with Figure \ref{figure example dividing a subscritical handle}. 
Let $\tilde{D}$ be the new handle decomposition of $M$.
Then, $\tilde{D}$ satisfies $(\star)$.  

Again, Lemma \ref{lem handle decomposition of cotangent bundle} gives an injective Weinstein handle decomposition $W_{\tilde{D}}$ from $\tilde{D}$. 
The left is to show that $W_{\tilde{D}}$ passes two steps (i) and (ii). 
In order to show this, we use Legendrian isotopies of attaching Legendrians of Weinstein handles in $\tilde{\mathcal{D}}$. 
These are main contents of Sections \ref{section the proof of Theorem main} and \ref{section examples}.

\subsection{Technical statement of Theorem \ref{thm rough statement}}
\label{subsect technical statement}
Now, we are ready to state the technical statement of Theorem \ref{thm rough statement} based on the idea explained in Section \ref{subsection the main idea}.
First, we state an algorithm $\mathcal{A}$ that produces an injective Weinstein handle decomposition of $T^*M$ from any handle decomposition of $M$.

We fix notations first. 
Let $D = \{h_0, \dots, h_m\}$ be a handle decomposition of an $n$-dimensional manifold $M$.
By Definition \ref{def handle decomposition}, there is a natural number $N$ such that $h_i$ is of index $<n$ (resp.\ $n$) if $i \leq N$ (resp.\ $i > N$).
Lemma \ref{lem handle decomposition of cotangent bundle} constructs a Weinstein handle decomposition of $T^*M$ from $D$.  
By abuse of notation, let $W_D$ denote the Weinstein handle decomposition of $T^*M$.

The algorithm $\mathcal{A}$ consists of two steps. 
The first step is to construct another handle decomposition $\tilde{D}$ of $M$ from $D$, and the second step is to apply Lemma \ref{lem handle decomposition of cotangent bundle} to $\tilde{D}$. 
\vskip0.2in

\noindent{\em Step 1.}
The first step of $\mathcal{A}$ is to add a canceling pair of index $(n-1, n)$-handles for each of handles of index $<n$ in $D$.
By adding a canceling pair, one can imagine that $h_i$ is divided into three handles.
To be more precise, let $h_i^{ori}$ denote the handle of index $\ind(h_i)$, and let $h_i^{n-1}$ (resp.\ $h_i^n$) denote the handle of index $(n-1)$(resp.\ $n$)-handle in the added canceling pair. 
Then, we would like to say that 
\[h_i = h_i^{ori} \cup h_i^{n-1} \cup h_i^n \simeq \mathbb{D}^k \times \mathbb{D}^{n-k}, h_i^n \simeq \mathbb{D}^k \times \mathbb{D}^{n-k}_\epsilon,\]
where $k$ is the index of $h_i$ and $\mathbb{D}^{n-k}_\epsilon$ is the $(n-k)$-dimensional disk with radius $\epsilon < \tfrac{1}{2}$.
See Figure \ref{figure example dividing a subscritical handle}.

More precisely, we can say that $h_i^{ori}, h_i^{n-1}, h_i^n$ satisfy the following:
\begin{enumerate}
	\item[(i)] We note that for $k := \ind(h_i)$, there is a map $f$ such that $\partial_R h \stackrel{f}{\simeq} S^{k-1} \times \mathbb{D}^{n-k}$.
	Then, $f(\partial_R h_i \cap \partial_R h_i^n) \simeq S^{k-1} \times \mathbb{D}_{\epsilon}^{n-k}$, where $\mathbb{D}_{\epsilon}^{n-k}$ is a smaller disk with a radius $\epsilon <1$;
	\item[(ii)] $\partial h_i \setminus \partial_R h_i$ does not intersect $h_i^n$. 
\end{enumerate}
An example for $3$-dimensional $1$-handle is given in Figure \ref{figure example dividing a subscritical handle}.

\begin{figure}[h]
	\centering
	\input{example_dividing_a_subcritical_handle.txt}		
	\caption{The left is a $3$-dimensional $1$-handle $h$, and the right is a division of $h$ into a $3$-handle $h^3$ (red), a $2$-handle $h^2$ (blue), and the other $1$-handle $h^{ori}$ (complement of red and blue). One can observe that the red and blue handles are in a canceling pair.}
	\label{figure example dividing a subscritical handle}
\end{figure} 

\begin{remark}
	\label{rmk the choice of order}
	We note that if $\ind(h_i) = n-1$, then there are two $(n-1)$-handles after dividing. 
	Thus, in order to use the notation $h_i^{ori}$ and $h_i^{n-1}$, it is necessarily to choose one of two possibilities. 
	However, at the end, the choice does not affect on the resulting Lefschetz fibration.
	We ignore this issue in the rest of the paper for convenience.
\end{remark}

After dividing all subcritical handles in $D$, one obtains another handle decomposition $\tilde{D}$ of $M$ such that
\[\tilde{D} := \{h_0^{ori}, h_0^{n-1}, h_1^{ori}, h_1^{n-1}, \dots, h_N^{ori}, h_N^{n-1}, h_0^{n}, \dots, h_N^{n}, h_{N+1}, \dots, h_m\}.\]
We note that $\tilde{D}$ consists of $(2N+2)$-many handles of index $<n$ and $(m+1)$-many handles of index $n$.
\vskip0.2in

\noindent{\em Step 2.}
The second step of the algorithm is to apply Lemma \ref{lem handle decomposition of cotangent bundle} for $\tilde{D}$. 
Then, one obtains a Weinstein handle decomposition $W_{\tilde{D}}$ of $T^*M$. 
By abuse of notation, let
\[W_{\tilde{D}} = \{H_0^{ori}, H_0^{n-1}, H_1^{ori}, H_1^{n-1}, \dots, H_N^{ori}, H_N^{n-1}, H_0^{n}, \dots, H_N^{n}, H_{N+1}, \dots, H_m\}.\]
We remark that there is an one-to-one relation between the handles in $\tilde{D}$ and Weinstein handles in $W_{\tilde{D}}$.
In the above, $H_i^{ori}, H_i^{n-1}, H_i^n$ for $i \leq N$ and $H_j$ for $j > N$ correspond to $h_i^{ori}, h_i^{n-1}, h_i^n$ and $h_j$ by the one to one relation.

One can easily check that for any handle decomposition $D$ of $M$, $W_{\tilde{D}}$ is an injective Weinstein handle decomposition from the viewpoint of Definition \ref{def injectiveness}. 
Theorem \ref{thm main} claims that $W_{\tilde{D}}$ produces a Lefschetz fibration. 

\begin{theorem}[=Theorem \ref{thm rough statement}]
	\label{thm main}
	Let $M$ be a smooth manifold of dimension $n$.
	For any handle decomposition $D$ of $M$, let $W_{\tilde{D}}$ be the Weinstein handle decomposition of $T^*M$ obtained by applying the above algorithm $\mathcal{A}$ to $D$.
	Then, one can produce a Lefschetz fibration from $W_{\tilde{D}}$.  
\end{theorem}

\begin{remark}\mbox{}
	\label{rmk M_0}
	\begin{enumerate}
		\item We note that the regular fiber $F$ of the resulting Lefschetz fibration is determined by the subcritical handles in $W_{\tilde{D}}$. 
		Since every subcritical handle in $W_{\tilde{D}}$ arises from a handle of index $<n$ in $D$, the handles of index $<n$ in $D$ determine $F$.
		
		For more detail, let $M_{sub}:= \cup_{i=0}^N \partial h_i$. 
		We note that $M_{sub}$ is not a smooth manifold, but a manifold with corners. 
		One can easily see that a tubular neighborhood of $M_{sub}$ in $T^*M$ is the union of all subcritical handle of index $<n$ in $W_{\tilde{D}}$. 
		Thus, the neighborhood of $M_{sub}$ is equivalent to $F \times \mathbb{C}$.
		\item The number of critical handles in $W_{\tilde{D}}$ is the same as the number of all handles in $D$. 
		Thus, the number of singular values of the resulting Lefschetz fibration is also the same as the number of handles in $D$.  
		This gives an upper bound of the minimal number of singular values over the set of Lefschetz fibrations on a cotangent bundle.
	\end{enumerate}
\end{remark}

\subsection{A counter example}
\label{subsect example : the case of T^*S^n}
Now, we give an example of Weinstein handle decomposition that cannot produce a Lefschetz fibration. 

It is easy to prove that $T^*S^n$ admits a Weinstein handle decomposition consisting of one Weinstein $0$-handle and one Weinstein $n$-handle. 
It is because $S^n$ admits a decomposition into one $0$-handle and one $n$-handle. 
Then, Lemma \ref{lem handle decomposition of cotangent bundle} gives the desired Weinstein handle decomposition of $T^*S^n$.

Let us assume that the Weinstein handle decomposition produces a Lefschetz fibration $\pi$. 
Then, the regular fiber $F$ should be $\mathbb{D}^{2n-2}$ since the only subcritical handle is the zero handle, i.e., 
\[F \times \mathbb{C}^2 \simeq \text{  the Weinstein $0$-handle  } \simeq \mathbb{D}^{2n}.\]

Since the Weinstein handle decomposition has one critical handle, the Lefschetz fibration $\pi$ has one critical value. 
Let $L$ be the vanishing cycle corresponding to the critical value.
Then, $L$ should be an exact Lagrangian submanifold of $F$.
However, it is well-known that there is no exact Lagrangian in $\mathbb{D}^{2n-2}$. 
Thus, it is a contradiction. 

\begin{remark}
	\label{rmk number of critical values}
	From the above arguments, one can conclude that every Lefschetz fibration of $T^*S^n$ has at least $2$ or more critical values. 
	Since there exists a well-known Lefschetz fibration of $T^*S^n$ having exactly $2$ critical values, $2$ is the minimal number of critical values of a Lefschetz fibration of $T^*S^n$.
	
	Moreover, the same arguments work for the case of Milnor fibers of $A_n$-singularities.
	As the result of the same arguments, any Lefschetz fibrations of those Milnor fibers have at least $(n+1)$ critical values. 
	Since there is a well-known Lefschetz fibration of the Milnor fiber with exactly $(n+1)$ singular values, $(n+1)$ is the minimum number of singular values.
\end{remark}

\section{Proof of Theorem \ref{thm main}}
\label{section the proof of Theorem main}

In the section, we prove Theorem \ref{thm main}. 
More precisely, we will construct a Lefschetz fibration in an {\em inductive} manner. 
The first subsection will give an induction hypotheses, and the second subsection will prove the base and the final steps of the construction.
Since the proof of the induction step is complicated, we first give the sketch in Section \ref{subsection sketch of the induction step}, we prepare the full proof in Section \ref{subsection preparations for Legendrian isotopy}, and the full proof will be given in Section \ref{subsection Legendrian isotopy}.

We note that in Section \ref{section examples}, we will give an explicit example of the inductive construction with figures. 
It could be a good idea to skip the complicated part of the abstract proof, i.e., the induction step, read the example section first, and come back to the present section.

\subsection{Induction hypotheses}
\label{subsection induction hypotheses}
First, we note that as we used in the previous subsection, we use the same notation
\begin{gather*}
	D = \{h_0, \dots, h_m\}, \\
	\tilde{D} := \{h_0^{ori}, h_0^{n-1}, h_1^{ori}, h_1^{n-1}, \dots, h_N^{ori}, h_N^{n-1}, h_0^{n}, \dots, h_N^{n}, h_{N+1}, \dots, h_m\}, \\
	W_{\tilde{D}} = \{H_0^{ori}, H_0^{n-1}, H_1^{ori}, H_1^{n-1}, \dots, H_N^{ori}, H_N^{n-1}, H_0^{n}, \dots, H_N^{n}, H_{N+1}, \dots, H_m\}.
\end{gather*}
Moreover, let 
\[M_0 \subset M_1 \subset \dots \subset M_N \subset M_{N+1} = M\]
be an increasing sequence of closed subsets defined as 
\begin{gather}
	\label{eqn def of M_i}
	M_i := \cup_{j=0}^i \big(h_j^{ori} \bigcup h_j^{n-1}\big), \text{  if  } i \leq N, \text{  and  } M_{N+1} := M.
\end{gather}
Similarly, we also define 
\[\overline{M}_i = \cup_{j=0}^i h_j \text{  for all  } i \leq N.\]
Then, it is easy to observe that 
\begin{gather}
	\label{eqn indcution hypotheses}
	\overline{M}_i = M_i \cup \left(\bigcup_{j=0}^i h_i^n\right), \partial M_i = \partial \overline{M}_i \cup \left(\bigcup_{j=0}^n \partial h_i^n\right) \text{  for all  } i \leq N.
\end{gather}

Our strategy is to find $F_i$ such that, for all $i \leq N$,
\begin{gather}
	\label{eqn induction hypothesis}
	T^*M_i \simeq F_i \times \mathbb{C},
\end{gather}
i.e., the product Weinstein manifold $F_i \times \mathbb{C}$ equipped with the product Weinstein structure is Weinstein homotopic to $T^*M_i$ equipped with the standard cotangent bundle Weinstein structure, and moreover, satisfies the following induction hypotheses: 
\begin{enumerate}
	\item[(IH 1)] The skeleton of $F_i$, $\mathrm{Skel}(F_i)$, is a recovering Lagrangian skeleton. 
	For the definition of recovering skeleton, see Definition \ref{def recovering skeleton}.
	Especially, there exists a collection of Lagrangians $\{L_0, \dots, L_i\}$ such that 
	\[\mathrm{Skel}(F_i)= \cup_{j=0}^i L_j.\]
	Moreover, we also require that $F_i$ is obtained by attaching a Weinstein sector $T^*\widetilde{L}_i$ to $F_{i-1}$ satisfying that $\mathrm{Skel}(F_{i-1}) = \cup_{j=0}^{i-1}L_j$.
	For the attachment of $T^*L_i$ to $F_{i-1}$, see Remark \ref{rmk recovering skeleton};
	\item[(IH 2)] We note that as seen in Equation \eqref{eqn indcution hypotheses}, $\partial h_j^n \subset \partial M_i$ for all $0 \leq j \leq i$. 
	Since $T^*M_i \simeq F_i \times \mathbb{C}$, $\partial h_j^n \subset \partial_\infty \left(F_i \times \mathbb{C}\right)$, where $\partial_\infty$ means the asymptotic boundary.
	We note that $\partial_\infty \left(F_i \times \mathbb{C}\right)$ contains $F_i \times \partial_\infty \mathbb{C} = F_i \times S^1$. 
	The {\em second induction hypothesis} is that as a subset of $\partial_\infty \left(F_i \times \mathbb{C}\right)$, 
	\[\partial h_j^n \subset F_i \times \left[\pi-j\theta_0, \pi-(j-1) \theta_0\right),\]
	for a constant $\theta_0>0$ satisfying $N \theta_0 < \pi$.
	We note that $\left[\pi-j\theta_0, \pi-(j-1) \theta_0\right) \subset \mathbb{R}/\mathbb{Z} \simeq S^1$. 
	\item[(IH 3)] Similar to (IH 2), we note that $\partial \overline{M}_i \subset \partial M_i \subset \partial_\infty \left(F_i \times \mathbb{C}\right)$.
	For the same $\theta_0$ in (IH 2), 
	\[\partial \overline{M}_i \subset F_i \times [-i \theta_0, -(i-1)\theta_0).\]
	We note that $[-i \theta_0, -(i-1)\theta_0) \subset \mathbb{R}/(2 \pi \mathbb{Z}) \simeq S^1$;
	\item[(IH 4)] We would like to remind that 
	\[\partial_R h_{i+1} \subset \partial \overline{M}_i \subset F_i \times [-i \theta_0, -(i-1)\theta_0).\]
	It is possible to attach a Weinstein handle of index $k =\mathrm{ind} (h_{i+1})$ to $F_i \times \mathbb{C}$ along the attaching sphere of $h_{i+1}$. 
	After the attachment, the attached subcritical handle induces a stable manifold that is homeomorphic to a $k$-dimensional open disk. 
	Then, the resulting manifold admits the Lagrangian skeleton obtained by attaching $\mathbb{D}^k$ to $\mathrm{Skel}(F_i)$.
	The attachment can be encoded by an attaching map 
	\[A_{i+1}: \partial \mathbb{D}^k \simeq S^{k-1} \to \mathrm{Skel}(F_i).\]
	The {\em fourth induction hypothesis} is that $A_{i+1}$ is a transversally embedding map and so that the image $A_{i+1}(S^{k-1})$ is a transversally embedded submanifold of the skeleton of $F_i$. 
	See Definition \ref{def transversally embedded} for the definition of transversal embedded subset of a recovering skeleton. 
	\item[(IH 5)] For every $0 \leq j \leq i$, there exists an exact Lagrangian $V_j \subset F_i$ such that a Legendrian lift (see Definition \ref{def legendrian lift}) of $V_j$ is Legendrian isotopic to $\partial h_j^n$.
	\item[(IH 6)] There exists an exact Lagrangian $\overline{V}_{i+1} \subset F_i$ such that $\partial \overline{M}_i$ is a Legendrian lift of $\overline{V}_{i+1} \subset \mathrm{Skel}(F_i)$. 
\end{enumerate}

\begin{remark}
	\label{rmk product structure}
	\mbox{}
	\begin{enumerate}
		\item Before going further, we would like to set more notations, and would like to remark the following: 
		Let 
		\[W_i := F_i \times \mathbb{C}.\]
		Then, $W_i$ admits a product Lefschetz fibration 
		\[\pi_i : W_i = F_i \times \mathbb{C} \to \mathbb{C}.\]
		We note that the symplectic completion of $\mathbb{D}^2$ is $\mathbb{C}$. 
		Based on this, we abuse notation as follows: 
		We replace $\mathbb{C}$ with $\mathbb{D}^2$. 
		Similarly, we say that $W_i = F_i \times \mathbb{D}^2$.and that the target space of a product Lefschetz fibration $\pi_i$ is $\mathbb{D}^2$.   
		In a similar manner, we replace $F_i$ with a Weinstein domain whose completion is $F_i$. 
		The Weinstein domain is also denoted by $F_i$ by abuse of notation. 
		By the replacement, we achieve that $F_i$ has an actual boundary, rather than the asymptotic boundary. 
		\item We would like to recall that, for $0 \leq j \leq i$, $\partial h_j^n$ is a Legendrian sphere of the contact boundary of $F_i \times \mathbb{C}$. 
		When one takes the time-$t$ Liouville flow of $F_i \times \mathbb{C}$ of the Legendrian $\partial h_j^n$ as $t \to -\infty$, the limit defines a subset of the skeleton of $F_i$.
		We note that the Lagrangian sphere $V_j$ in (IH 5) would be obtained by smoothing the limit subset at an intuitive level.  
		Similarly, the exact Lagrangian $\overline{V}_{i+1} \subset F_i$ would be determined from $\partial \overline{M}_i$. 
		It is compatible to the condition in (IH 6) that $\overline{V}_{i+1} \subset \mathrm{Skel}(F_i)$. 
		Moreover, the image of $A_{i+1}$ in (IH 4) is a subset of $\overline{V}_{i+1}$. 
		\item We note that the induction hypotheses (IH 2) and a(IH 3) are related to the cyclic order of the singular values. 
		On the other hand, (IH 5) and (IH 6) are related to the vanishing cycles. 
		\item We note that in (IH 4), before describing the fourth induction hypothesis, we described an attachment of a subcritical handle of index $k$ and the corresponding change on the skeleton. 
		On the skeleton, the corresponding change is an attachment of $\mathbb{D}^k$, but by applying a technique given in \cite[Section 3.1]{Starkston}, one can thicken $\mathbb{D}^k$.
		After the thickening, one would attach $\mathbb{D}^k \times \mathbb{D}^{n-k}$ instead of $\mathbb{D}^k$. 
	\end{enumerate}
\end{remark}

\subsection{The base and the final steps}
\label{subsection the base and the final steps}
As mentioned in the previous subsection, we construct $F_i$ such that $T^*M_i \simeq F_i \times \mathbb{C}$ for $i =0, \dots, N$ inductively. 
We note that the final result of the inductive construction, i.e., $F_N$, is determined by the union of all subcritical Weinstein handles in $W_{\tilde{D}}$, and the final step of the inductive construction is to attach all critical Weinstein handles in $W_{\tilde{D}}$.
In this subsection, we discuss the base step of the inductive construction as well as the final step of the proof. 

\subsubsection{The base step}
The base step is to construct a Weinstein manifold $F_0$ such that $T^*M_0 \simeq F_0 \times \mathbb{C}$.
Note that the symbol $\simeq$ means an equivalence {\em up to Weinstein homotopic}. 
The symbol means the same equivalence in the rest of the paper. 
By the above construction, $M_0$ is an $n$-dimensional annulus, i.e., $M_0 = S^{n-1} \times [0,1]$.
Thus, $T^*M_0$ is equivalent to $T^*S^{n-1} \times \mathbb{D}^2$. 

Let $W_0$ be the total space of an abstract Lefschetz fibration $\pi_0$ given as  
\[ \pi_0 := (F_0 = T^*S^{n-1} ; \varnothing). \]
Since $T^*M_0$ and $W_0$ both are equivalent to $T^*S^{n-1} \times \mathbb{D}^2$, $T^*M_0$ is equivalent to $W_0$. 

Now, we need to check that $F_0 = T^*S^{n-1}$ satisfies the induction hypotheses (IH 1--6).
The first condition (IH 1) is easy to show, because the Lagrangian skeleton of $F_0 = T^*S^{n-1}$ is the zero section $S^{n-1}$.

For the second and the third conditions (IH2) and (IH3), we recall that 
\[S^{n-1} \sqcup S^{n-1} = S^{n-1} \times \{0, 1\} = \partial M_0 = \partial h_0^n \cup \partial \overline{M}_0.\]
Since one can identify $M_0$ with 
\[\text{the zero section  } S^{n-1} \times \{p \in \mathbb{D}^2 | p \text{  lies on the $x$-axis.}\} \subset F_0 \times \mathbb{D}^2.\]
Thus, one can identify $\partial h_0^n$ (resp.\ $\partial \overline{M}_0$) with the zero section of the fiber $\pi_0^{-1}(-1,0) \simeq F_0 = T^*S^{n-1}$ (resp.\ $\pi^{-1}_0(1,0)$), where $(\pm 1, 0) \in \mathbb{D}^2 \subset \mathbb{R}^2$. 

From the above, the other conditions (IH 4--6) hold trivially.
Especially, both of $V_0$ in (IH 5) and $\overline{V}_1$ in (IH 6) would be the zero section of the fiber $F_0 = T^*S^{n-1}$, since $\partial h_0^n$ and $\partial \overline{M}_1$ are the zero sections of the fibers at $(-1,0) \in \mathbb{D}^2$ and $(1, 0) \in \mathbb{D}^2$, respectively. 

We note also that the Lagrangian skeleton of $F_0 = T^*S^{n-1}$ is the zero section of the cotangent bundle. 
Thus, the skeleton is a smooth, closed Lagrangian submanifold. 

\subsubsection{The final step}
By repeating the inductive step $N$-times, one obtains $W_N$ such that $T^*M_N \simeq W_N = F_N \times \mathbb{D}^2$. 
We note that $T^*M_N$ is the union of all subcritical handles in $W_{\tilde{D}}$. 
Thus, in order to finish the proof, we need to attach critical handles to $W_N$. 
The attachment of critical handles to a subcritical Weinstein manifolds equipped with a product Lefschetz fibration has been studied well. 
See, for example, \cite[Proposition 8.1]{BEE} and \cite[Section 6]{GP}.

In our Weinstein handle decomposition $W_{\tilde{D}}$, there exist $(m+1)$-many critical Weinstein handles, labeled by 
\[H_0^n, \dots, H_N^n, H_{N+1}, \dots, H_m.\]
For $0 \leq j \leq N$, $H_j^n$ is attached to $W_n$ along $\partial h_j^n$ that is located in $F_i \times \left[\pi - j \theta_0, \pi-(j-1)\theta_0\right)$ by (IH 2). 
For $j \geq N+1$, $H_j$ is attached to $W_n$ along $\partial h_j \subset \partial \overline{M}_N$. 
We note that $\partial \overline{M}_N$ is a disjoint union of $\cup_{j \geq N+1} \partial h_j$.
Thus, one can see $\partial \overline{M}_N$ as a disjoint union of Legendrian spheres in $F_i \times \left( - (N-1) \theta_0, -N \theta_0\right]$.

Since $N \theta_0 < \pi$, it is easy to check that the attaching Legendrian spheres for the critical handles 
\[H_0^n, \dots, H_N^n, H_{N+1}, \dots, H_m\]
are pairwise disjoint. 
Thus, based on \cite[Proposition 8.1]{BEE}, one can obtain a Lefschetz fibration for $T^*M$.
Moreover, the vanishing cycles are determined by (IH5) and (IH 6).

\subsection{Sketch of the induction step}
\label{subsection sketch of the induction step}
Since the inductive step is complicated, we discuss it without details in this subsection, and the details will be discussed in Section \ref{subsection Legendrian isotopy}.

Let assume that the induction hypotheses (IH 1--6) hold for some $0 \leq i \leq N-1$. 
We would like to show that one can construct $F_{i+1}$ satisfying (IH 1--6) from $F_i$. 
To do that, let us remark that
\[T^*M_i \simeq \cup_{k=0}^i\left(H_k^{ori} \cup H_k^{n-1}\right), \text{  for all  } i \leq N.\]
And, for convenience, we will use the following notation: $\check{H}_{i+1}^{ori}$ (resp.\ $\check{H}_{i+1}^{n-1}$) denotes the Weinstein handle such that 
\begin{gather*}
	\dim(\check{H}_{i+1}^{ori}) = \dim(H_{i+1}^{ori}) -2, \ind(\check{H}_{i+1}^{ori}) = \ind(H_{i+1}^{ori}),\\
	\dim(\check{H}_{i+1}^{n-1}) = \dim(H_{i+1}^{n-1}) -2, \ind(\check{H}_{i+1}^{n-1}) = \ind(H_{i+1}^{n-1}).
\end{gather*}
Then, one can easily see that
\[H_{i+1}^{ori} \simeq \check{H}_{i+1}^{ori} \times \mathbb{D}^2, \hspace{0.2em}H_{i+1}^{n-1} \simeq \check{H}_{i+1}^{n-1} \times \mathbb{D}^2,\]
where the equivalence is up to Weinstein homotopy and $\mathbb{D}^2$ has the standard radial Weinstein structure.

By the definition of $M_{i+1}$, 
\[T^*M_{i+1} \simeq (F_i \times \mathbb{C}) \cup H_{i+1}^{ori} \cup H_{i+1}^{n-1}.\]
Our strategy is to take some Legendrian isotopies of the attaching Legendrians $\partial_L H^{ori}_{i+1}$ and $\partial_L H_{i+1}^{n-1}$ so that 
the Weinstein handle attachments along the attaching Legendrians after the isotopies {\em respects} the product structure in Equation \eqref{eqn induction hypothesis}.
In other words, 
\[W_{i+1} = (F_i \times \mathbb{C}) \cup H_{i+1}^{ori} \cup H_{i+1}^{n-1} \simeq \left(F_i \cup \check{H}_{i+1}^{ori} \cup \check{H}_{i+1}^{n-1}\right) \times \mathbb{C}.\]

In order to refer later, we remark the following: 
\begin{remark}
	\label{rmk attaching two at one time}
	\mbox{}
	\begin{enumerate}
		\item We recall that $h_{i+1}$ is divided into three handles $h_{i+1}^{ori}, h_{i+1}^{n-1}$, and $h_{i+1}^n$.
		By the construction in Section \ref{subsect technical statement}, if one identifies $h_{i+1}$ with $\mathbb{D}^k \times \mathbb{D}^{n-k}$ where $k = \mathrm{ind} (h_{i+1})$, $h_{i+1}^n \simeq \mathbb{D}^k \times \mathbb{D}^{n-k}_\epsilon$ with $\epsilon < \tfrac{1}{2}$.
		And similarly, the union of $h_{i+1}^{ori} \cup h_{i+1}^{n-1} \simeq S^{n-k-1} \times [0,1] \times \mathbb{D}^k$. 
		We also recall that, roughly speaking, $H_{i+1}^{ori} \simeq D^*h_{i+1}^{ori}$ and $H_{i+1}^{n-1} \simeq D^*h_{i+1}^{n-1}$.
		Thus, one can conclude that 
		\begin{gather}
			\label{eqn subcritical parts}
			H_{i+1}^{ori} \cup H_{i+1}^{n-1} \simeq D^*\left(S^{n-k-1} \times [0,1] \times \mathbb{D}^k\right) \simeq D^*\left(S^{n-k-1} \times \mathbb{D}^k\right) \times D^*[0,1] \simeq \left(\check{H}_{i+1}^{ori} \cup \check{H}_{i+1}^{n-1}\right) \times \mathbb{D}^2.
		\end{gather}
		Then, by taking a proper Legendrian isotopy, we would like to show that the attachment of 
		\[H_{i+1}^{ori} \cup H_{i+1}^{n-1} \simeq \left(\check{H}_{i+1}^{ori} \cup \check{H}_{i+1}^{n-1}\right) \times \mathbb{D}^2\]
		to $W_i = F_i \times \mathbb{D}^2$ respects the product structure.
		\item We also note that the attachment of $\check{H}_{i+1}^{ori} \cup \check{H}_{i+1}^{n-1} \simeq D^*(S^{n-k-1} \times \mathbb{D}^k)$ to $F_i$ could be seen as the attachment discussed in Section \ref{subsection recovering Lagrangian skeleta}.
		For that, we require the first and fourth induction hypotheses (IH 1) and (IH 4).
		More precisely, according to the discussion in Section \ref{subsection recovering Lagrangian skeleta}, $D^*(S^{n-k-1} \times \mathbb{D}^k)$ can be attached along 
		\[\partial (S^{n-k-1} \times \mathbb{D}^k) = S^{n-k-1} \times S^{k-1},\]
		if there exists an transversally embedded $S^{n-k-1} \times S^{k-1}$ in the Lagrangian skeleton of $F_i$. 
		By (IH 1, 4), Lemma \ref{lem disk bundle}, and Remark \ref{rmk transversally embedded}, one can find a transversally embedded $S^{n-k-1} \times S^{k-1}$ as the boundary of the embedded disk bundle of $A_{i+1}(\partial \mathbb{D}^k)$ given in (IH 4).
	\end{enumerate}
\end{remark}

If we can take a proper Legendrian isotopy and if we can attach two Weinstein handles $H_{i+1}^{ori}$ and $H_{i+1}^{n-1}$ as described above, then the resulting Weinstein domain $W_{i+1}$ has a product Lefschetz fibration 
\[\pi_{i+1} : W_{i+1} \to \mathbb{C},\]
whose fiber is 
\[F_{i+1} = F_i \cup \check{H}_{i+1}^{ori} \cup \check{H}_{i+1}^{n-1}.\]
Thus, it is enough to find a proper Legendrian isotopy. 

We will construct a proper Legendrian isotopy in Section \ref{subsection Legendrian isotopy}. 
In the rest of the present subsection, we describe the properties that the Legendrian isotopy should satisfy. 

For convenience, we let 
\[\Lambda_{i+1} := \partial \overline{M}_i \subset F_i \times [-i \theta_0, -(i-1)\theta_0).\]
Then, $\Lambda_{i+1}$ is divided into three parts, one contained in $\partial_L H_{i+1}^{ori} \cup \partial_L H_{i+1}^{n-1}$, one in $\partial_L H_{i+1}^n$, and their compliment. 
Along the first part, i.e., one in $\partial_L H_{i+1}^{ori} \cup \partial_L H_{i+1}^{n-1}$, we attach the subcritical handles $H_{i+1}^{ori}$ and $H_{i+1}^{n-1}$, along the second one, i.e., one in $\partial_L H_{i+1}^n$, we attach the critical handle $H_{i+1}^n$, and the last one, i.e., the complement, would be a part of $\partial \overline{M}_{i+1}$ in the next induction step. 
Moreover, at the end of the induction, $\partial \overline{M}_N$ would be the union of attaching Legendrian spheres for the critical handles $H_{N+1}, \dots, H_m$.
From the above description, we set the following notation:	
\begin{itemize}
	\item We let $\Lambda_{i+1}^{sub}$ denote the part of $\Lambda_{i+1}$ contained in $\partial_L H_{i+1}^{ori} \cup \partial_L H_{i+1}^{n-1}$;
	\item $\Lambda_{i+1}^{cri}$ denote the part of $\Lambda_{i+1}$ contained in  $\partial_L H_{i+1}^n$;
	\item $\Lambda_{i+1}^{comp}$ denote the complement of $\Lambda_{i+1}^{sub} \cup \Lambda_{i+1}^{cri}$ in $\Lambda_{i+1}$.
\end{itemize}

We will take the Legendrian isotopy of $\Lambda_{i+1}$ so that after the Legendrian isotoping, the following hold: 
\begin{enumerate}
	\item[(A)]$\Lambda_{i+1}^{sub}$ is lying on the horizontal boundary of $W_i$, i.e., $(\partial_\infty F_i) \times \mathbb{D}^2$. 
	Moreover, $\Lambda_{i+1}^{sub}$ is a product of a Legendrian in the boundary of $F_i$ and a diameter of $\mathbb{D}^2$. 
	For the future use, let $\theta_i$ be the number such that the diameter connects $e^{-\sqrt{-1} \theta_i}$ and $e^{-\sqrt{-1}(-\theta_i + \pi)}$.
	\item[(B)] $\Lambda_{i+1}^{cri}$ is lying on the vertical boundary of $W_i$, i.e., $F_i \times (\partial \mathbb{D}^2)$.
	Moreover, that part is projected down to $\{e^{\sqrt{-1}\theta} \hs | \hs \theta \in \left[\pi-(i+1)\theta_0, \pi-i \theta_0\right)\} \subset S^1 = \partial \mathbb{D}^2$ by $\pi_i$.
\end{enumerate} 
A conceptual picture for the lowest-dimensional case is given in Figure \ref{figure legendrian modification}.
\begin{figure}[h]
	\centering
	\input{legendrian_modification.txt}		
	\caption{a). An example of handle decomposition $D$ of an annulus with an index $0$-handle $h_0$ and an index $1$-handle $h_1$.  b). A handle decomposition $\tilde{D}$ induced from $D$. c). The Lefschetz fibration $\pi_0$. We note that the zero sections of two fiber $\pi_0^{-1}(\pm 1)$ are identified with $\partial (h_0^1 \cup h_0^{ori}) = \partial h_0 \cup \partial h_0^2$. d). The projected image of the Legendrian that we would like to achieve by Legendrian isotoping. We note that the colored parts in d) correspond to the same colored part in b). The red parts correspond to $\Lambda_1^{sub}$ and the blue part corresponds to $\Lambda_1^{cri}$.}
	\label{figure legendrian modification}
\end{figure} 

\begin{remark}
	\label{rmk reason of IH 4 and 6}
	In (IH 6), $\Lambda_{i+1} \simeq \partial \overline{M}_i$ is a Legendrian lift (see Definition \ref{def legendrian lift} in the next subsection) of $\overline{V}_{i+1}$.
	In the later sections, our strategy for taking a Legendrian isotopy is to take a Lagrangian isotopy of $\overline{V}_{i+1}$ (and its Legendrian lift). 
	Moreover, in the process, roughly speaking, we move $\Lambda_{i+1}^{sub}$ and $\Lambda_{i+1}^{cri}$, or equivalently, the part of $\Lambda_{i+1} \simeq \partial \overline{M}_i$ corresponding to $\partial_R h_{i+1}$, and the compliment $\Lambda_{i+1}^{comp}$ does not change through the Legendrian isotopy. 
\end{remark}

\subsection{Preparations for Section \ref{subsection Legendrian isotopy}}
\label{subsection preparations for Legendrian isotopy}
In Section \ref{subsection preparations for Legendrian isotopy}, we set notation before discussing Legendrian isotopies.

\subsubsection{Product structure of $W_i$}
\label{subsubsection product structure}
Since we would like to take Legendrian isotopies on $\partial_\infty W_i$ for $i = 0, \dots, N-1$, we need to review the contact structure on $\partial_\infty W_i$. 
The contact structure is the restriction of the Liouville structure, thus we start from the Liouville structure of $W_i$. 

For $i \leq N$, $W_i$ admits a product structure, i.e., 
\[W_i = F_i \times \mathbb{C}.\]
We note that as mentioned in Remark \ref{rmk product structure}, sometimes we assume 
\[W_i = F_i \times \mathbb{D}^2.\]
We can assume that because $\mathbb{C}$ is the symplectic completion of $\mathbb{D}^2$, and thus, two different $W_i$s are equivalent up to symplectic completion.
From this point of view, $W_i$ is also equivalent to $F_i \times \mathbb{D}_R^2$ where $\mathbb{D}_R^2$ means the $2$-dimensional disk of radius $R>0$, so by abuse of notation, we say
\begin{gather*}
	W_i = F_i \times \mathbb{D}_R^2.
\end{gather*} 

The product Liouville form on $W_i = F_i \times \mathbb{D}_R^2$ is given by 
\begin{gather}
	\label{eqn Liouville form of W_i}
	\lambda_{F_i} + \frac{1}{2}(x dy - y dx),
\end{gather}
where $\lambda_{F_i}$ is a Liouville form of $F_i$, and where $x, y$ are the standard coordinates of $\mathbb{D}_R^2 \subset \mathbb{R}^2$. 
For convenience, we simply use $\lambda_i$ for $\lambda_{F_i}$ if there is no chance of confusion. 
Also, by rescaling, we assume that $\mathbb{D}_R^2$ has the radius $1$.
However, the rescaling changes the Liouville form that is given in \eqref{eqn Liouville form of W_i}, and one obtains
\[\lambda_i + \frac{1}{c}(x dy - y dx),\]
where $c = 2R$. 

We note that the Lefschetz fibration $\pi_i : W_i \to \mathbb{D}^2$ is a projection to the second factor. 
For the projection to the first factor $F_i$, we set  
\[pr_i : W_i = F_i \times \mathbb{D}^2 \to F_i.\]

\subsubsection{Some contact topology on $\partial W_i$}
\label{subsubsection some contact topology}
Under the product structure, $\partial_\infty W_i$ consists of two parts, the vertical boundary $F_i \times \partial \mathbb{D}^2$ and the horizontal boundary $\partial_\infty F_i \times \mathbb{D}^2$. 
The contact forms on the vertical boundary and the horizontal boundary are given by 
\begin{gather}
	\label{eqn contact form on the vertical boundary}
	\lambda_i + \frac{1}{c}d \theta, \\
	\label{eqn contact form on the horizontal boundary} 
	\alpha_{F_i} + \frac{1}{c}(x dy- y dx),
\end{gather}
where $\theta \in \mathbb{R}/2\pi$ is the standard coordinate of $\partial \mathbb{D}^2$, and where $\alpha_{F_i}$ denotes the restriction of $\lambda_{F_i}$ on $\partial F_i$. 
We simply use $\alpha_i$ instead of $\alpha_{F_i}$ if there is no chance of confusion. 

Let $L$ be an exact Lagrangian of $F_i$, i.e., there is a function $f : L \to \mathbb{R}$ such that $df = \lambda_i|_L$. 
Then, together with a choice of $\theta_* \in \mathbb{R} / 2\pi$, one could lift $L$ to a Legendrian $\Lambda$ on the vertical boundary such that 
\begin{gather}
	\label{eqn lifted Legendrian}
	\Lambda := \{\big(p, \cos (- c f(p) - \theta_*), \sin ( - c f(p) - \theta_*) \big) \in F_i \times \mathbb{D}^2 \hs | \hs p \in L\}.
\end{gather} 

\begin{remark}
	To prove that $\Lambda$ is a Legendrian, we observe that $T L$ is identified with $T\Lambda$ by 
	\[V \in TL \mapsto V + c V(f) \sin (- c f(p) - \theta_*) \partial x - c V(f) \cos( - c f(p) - \theta_*) \partial y. \]
	The vector is in the kernel of the contact form, i.e., the one-form in Equation \eqref{eqn contact form on the vertical boundary}, since
	\begin{gather*}
		\lambda_i(V) - \frac{1}{c} c V(f) = df(V) - V(f) = 0.
	\end{gather*}
	We note that the first equality comes from $\lambda_i|_L = df$. 
	Then, it proves that $\Lambda$ is a Legendrian. 
\end{remark}

\begin{definition}
	\label{def legendrian lift}
	The Legendrian $\Lambda$ in Equation \eqref{eqn lifted Legendrian} is called the {\em Legendrian lift of $L$ with respect to $\lambda_i$ and $\theta_0$}.
\end{definition}

We note that for a Lagrangian $L$, its Legendrian lift depends on the choice of $\lambda_i$ and $\theta_0$ in Definition \ref{def legendrian lift}. 
However, the choice of $\theta_0$ does not change the Legendrian isotopy class of a Legendrian lift. 
To be more precise, for $\theta_0$ and $\theta_1$, let $\theta_s = (1-s) \theta_0 + s \theta_1$. 
Then, a Legendrian lift with respect to $\theta_s$ defines a Legendrian isotopy between Legendrian lifts with respect to $\theta_0$ and $\theta_1$.
Similarly, one can easily check that the choice of $\lambda_i$ changes the contact structure on the asymptotic boundary and thus, the choice effects Legendrian lifts. 
However, for two different choice of Liouville one forms, Gray's Theorem guarantees that there exists a $1$-parameter family of diffeomorphisms connecting two different contact structures and also Legendrian lifts with respect to two different Liouville one forms. 
In this sense, we roughly say that there exists a {\em unique} Legendrian lift of an exact Lagrangian up to aforementioned isotopy, not depending on the choices of $\lambda_i$ and $\theta_0$. 

\subsubsection{A Hamiltonian flow on $F_i$}
\label{subsubsection a hamiltoian flow}
As we replaced a Weinstein manifold $\mathbb{C}$ with a Weinstein manifold $\mathbb{D}^2$, we replace a Weinstein manifold $F_i$ with a Weinstein domain for convenience. 
The Weinstein domain will be denoted by $F_i$ again by abusing notation. 
We end the current subsection by constructing a Hamiltonian flow on $F_i$.
We will use the Hamiltonian isotopy later, especially in Section \ref{subsubsection crossing the base}.

Since $F_i$ is a Weinstein domain, there is a small tubular neighborhood of $\partial F_i$ which is symplectomorphic to $\partial F_i \times (-\epsilon, 0]$. 
The symplectic form on $\partial F_i \times (- \epsilon, 0]$ is $d(e^r \alpha_i)$ where $r \in (-\epsilon,0]$.
Moreover, the Liouville form $\lambda_i$ agrees with $e^r \alpha_i$ on $\partial F_i \times (-\epsilon,0]$.

Let $H : F_i \to \mathbb{R}$ be a function such that 
\begin{itemize}
	\item $H|_{F_i \setminus \partial F_i \times (-\epsilon,0]} \equiv 0$, 
	\item $H|_{\partial F_i \times (-\frac{\epsilon}{2}, 0]} \equiv e^r$, and
	\item $H|_{\partial F_i \times [-\epsilon, -\frac{\epsilon}{2}]}$ is smooth and increasing with respect to $r$-coordinate. 
\end{itemize}
Let $\Phi_i^t$ denote the time $t$ Hamiltonian flow associated to $H$. 

\begin{remark}
	\label{rmk reeb and hamiltonain flow}
	It is easy to check that on $\partial F_i$, $\Phi_i^t$ is the time $t$ Reeb flow with respect to the contact form $\alpha_i$. 
\end{remark}

\subsection{Legendrian isotopy}
\label{subsection Legendrian isotopy}
Now, we prove the induction step in the subsection. 

\subsubsection{More detailed sketch}
\label{subsubsection more detailed sketch}
We note that the sketch for the induction step is given in Section \ref{subsection sketch of the induction step}, but we can give a little bit more detailed sketch by using the notations defined in Section \ref{subsection preparations for Legendrian isotopy}. 

As mentioned in the sketch, we first construct a Legendrian isotopy of $\Lambda_{i+1}$, which satisfies the conditions (A) and (B) given at the end of Section \ref{subsection sketch of the induction step}.
For (A), after Legendrian isotoping, $\Lambda_{i+1}^{sub}$ should lie on the horizontal boundary of $W_i = F_i \times \mathbb{D}^2$. 
Since the original $\Lambda_{i+1}$ is contained in the vertical boundary of $W_{i+_1}$, the starting point is to push $\Lambda_{i+1}^{sub}$ to the corner of $W_i$.
This process will be explained in Section \ref{subsubsection push to the corner}.

Even after pushing to the corner, $\Lambda_{i+1}^{sub}$ does not satisfy (A), i.e., $\Lambda_{i+1}^{sub}$ is not a product of Legendrians in $\partial F_i$ and a diameter of the base $\mathbb{D}^2$. 
Thus, we need another Legendrian isotopy making $\Lambda_{i+1}^{sub}$ to satisfy (A). 
This process will be explained in Section \ref{subsubsection crossing the base}.
And, one can also achieve the condition (B) in this step.

The next step is to attach two subcritical Weinstein handles $H_{i+1}^{ori}$ and $H_{i+1}^{n-1}$. 
We note that as described in Equation \eqref{eqn subcritical parts}, 
\[H_{i+1}^{ori} \cup H_{i+1}^{n-1} \simeq D^*\left(S^{n-k-1} \times [0,1] \times \mathbb{D}^k\right) \simeq D^*\left(S^{n-k-1} \times \mathbb{D}^k\right) \times D^*[0,1] \simeq \left(\check{H}_{i+1}^{ori} \cup \check{H}_{i+1}^{n-1}\right) \times \mathbb{D}^2.\]
Then, attaching $H_{i+1}^{ori} \cup H_{i+1}^{n-1}$ along 
\[\text{a Legendrian of  } \partial F_i \times \text{  a diameter of the base  } \mathbb{D}^2,\]
is equivalent to attach 
\[\check{H}_{i+1}^{ori} \cup \check{H}_{i+1}^{n-1} \simeq D^*\left(S^{n-k-1} \times \mathbb{D}^k\right)\]
to $F_i$. 
It will give us $F_{i+1}$ satisfying (IH 1) and we can also show that $F_{i+1}$ satisfies (IH 4) with an extra argument. 

For the other induction hypotheses, we note that 
\[\partial M_i = \partial \overline{M}_i \cup \left(\bigcup_{j \leq i} \partial h_j^n\right), \partial M_{i+1} = \partial \overline{M}_{i+1} \cup \left(\bigcup_{j \leq i+1} \partial h_j^n\right).\]
When one compares $\partial M_i$ and $\partial M_{i+1}$, one of the changes is that $\partial h_{i+1}^n$ is added. 
We note that by the construction of $h_{i+1}^n$, $\partial h_{i+1}^n$ consists of two parts; one is denoted by $\Lambda_{i+1}^{cri}$ and the other is contained in $\partial h_{i+1}^{ori}$ and $\partial h_{i+1}^{n-1}$.
We also note that the second part contained in $\partial h_{i+1}^{ori}$ and $\partial h_{i+1}^{n-1}$ can be seen as a boundary part of the core Lagrangians of $H_{i+1}^{ori}$ and $H_{i+1}^{n-1}$. 
Then, together with the above explained attachments of $H_{i+1}^{ori}$ and $H_{i+1}^{n-1}$, (B) proves that the $(I+1)^{th}$ fiber $F_{i+1}$ satisfies the induction hypotheses (IH 2, 5).

The similar arguments also prove that $(i+1)^{th}$ step of the inductive construction satisfies (IH 3, 6) since $\partial \overline{M}_{i+1}$ consists of $\Lambda_{i+1}^{comp}$ and a part contained in $\partial h_{i+1}^{ori}$ and $\partial h_{i+1}^{n-1}$.

\subsubsection{Push to the corner}
\label{subsubsection push to the corner}
First, we recall that by (IH 3), $\Lambda_{i+1} \simeq \partial \overline{M}_i \subset F_i \times \left( -(i-1) \theta_0, -i \theta_0 \right]$. 
Thus, every point in $\partial \overline{M}_i$ can be coordinated by $(x, \theta)$ with $x \in F_i$ and $\theta \in \left( -(i-1) \theta_0, -i \theta_0 \right]$.
By taking the isotopy sending $(x,\theta)$ to $(x, \theta + t)$ for $t \in [- \theta_0, 0]$, one can assume that 
\[\Lambda_{i+1} \subset \partial \overline{M}_i \subset F_i \times \left( -i \theta_0, -(i+1) \theta_0 \right].\]

We note that by (IH 6), $\Lambda_{i+1}$ is equivalent to a Legendrian lift of $\overline{V}_{i+1} \subset \mathrm{Skel}(F_i)$.
Our strategy is to take a Lagrangian isotopy of $\mathrm{Skel}(F_i)$, especially, in a small neighborhood of the image of $A_{i+1}$ in (IH 4). 

We first recall that since $\mathrm{Skel}(F_i)$ is a recovering Lagrangian skeleton of $F_i$, for any $x \in \mathrm{Skel}(F_i)$, $\mathrm{Skel}(F_i)$ is locally modeled by $L_{\mathcal{A}_{m+1}} \times \mathbb{R}^{n-m-1} \subset T^*\mathbb{R}^m \times T^*\mathbb{R}^{n-m-1}$, where $\mathcal{A}_{m+1}$ is the positively rooted $A_{m+1}$-tree.
We will take an isotopy along $\mathbb{R}^{n-m-1}$ in $T^*\mathbb{R}^{n-m-1}$ and the product with $L_{\mathcal{A}_{m+1}}$ with the isotopy will provide a Lagrangian isotopy of $\mathrm{Skel}(F_i)$. 

To do that, we fix two auxiliary data. 
The first one is a disk bundle of $A_{i+1}(S^{k-1})$.
Note that by applying Lemma \ref{lem disk bundle}, one could extend the map $A_{i+1}:S^{k-1} \to \mathrm{Skel}(F_i)$ to
\[A_{i+1} : \mathbb{D}^{n-k}_{3 \epsilon} \times S^{k-1} \to \mathrm{Skel}(F_i),\] 
where $\mathbb{D}^{n-k}_{3 \epsilon}$ is the $(n-k)$-dimensional closed disk with radius $3 \epsilon$. 
We emphasize that by Lemma \ref{lem disk bundle} (2), $A_{i+1}$ sends $\mathbb{D}_{3\epsilon}^{n-k}$-factor to the $\mathbb{R}^{n-m-1}$-factor of the local model. 

The second auxiliary data is a function $\varphi: [0,3\epsilon] \to \mathbb{R}$ such that
\begin{itemize}
	\item $\varphi(3\epsilon) =0$, and
	\item the graph of $\varphi'$ is given in Figure \ref{figure graph}.
\end{itemize} 
\begin{figure}[h]
	\centering
\begingroup%
  \makeatletter%
  \providecommand\color[2][]{%
    \errmessage{(Inkscape) Color is used for the text in Inkscape, but the package 'color.sty' is not loaded}%
    \renewcommand\color[2][]{}%
  }%
  \providecommand\transparent[1]{%
    \errmessage{(Inkscape) Transparency is used (non-zero) for the text in Inkscape, but the package 'transparent.sty' is not loaded}%
    \renewcommand\transparent[1]{}%
  }%
  \providecommand\rotatebox[2]{#2}%
  \newcommand*\fsize{\dimexpr\f@size pt\relax}%
  \newcommand*\lineheight[1]{\fontsize{\fsize}{#1\fsize}\selectfont}%
  \ifx\svgwidth\undefined%
    \setlength{\unitlength}{198.42519685bp}%
    \ifx\svgscale\undefined%
      \relax%
    \else%
      \setlength{\unitlength}{\unitlength * \real{\svgscale}}%
    \fi%
  \else%
    \setlength{\unitlength}{\svgwidth}%
  \fi%
  \global\let\svgwidth\undefined%
  \global\let\svgscale\undefined%
  \makeatother%
  \begin{picture}(1,0.67142857)%
    \lineheight{1}%
    \setlength\tabcolsep{0pt}%
    \put(0,0){\includegraphics[width=\unitlength,page=1]{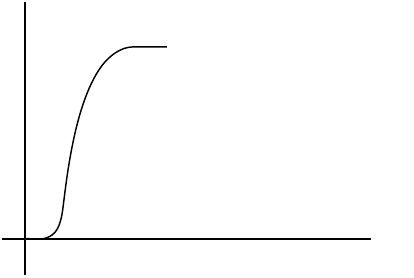}}%
    \put(0.02870664,0.56239146){\makebox(0,0)[lt]{\lineheight{1.25}\smash{\begin{tabular}[t]{l}$1$\end{tabular}}}}%
    \put(0.31071391,0.04760796){\makebox(0,0)[lt]{\lineheight{1.25}\smash{\begin{tabular}[t]{l}$\epsilon$\end{tabular}}}}%
    \put(0,0){\includegraphics[width=\unitlength,page=2]{graph.pdf}}%
    \put(0.47549681,0.04760977){\makebox(0,0)[lt]{\lineheight{1.25}\smash{\begin{tabular}[t]{l}$2\epsilon$\end{tabular}}}}%
    \put(0.72398216,0.04760977){\makebox(0,0)[lt]{\lineheight{1.25}\smash{\begin{tabular}[t]{l}$3\epsilon$\end{tabular}}}}%
    \put(0.02871067,0.04760977){\makebox(0,0)[lt]{\lineheight{1.25}\smash{\begin{tabular}[t]{l}$0$\end{tabular}}}}%
  \end{picture}%
\endgroup%

	\caption{}		
	\label{figure graph}
\end{figure}

One can also define $\tilde{\phi}: \mathbb{D}^{n-k}_{3\epsilon} \times S^{k-1} \to \mathbb{R}$ as 
\[\tilde{\phi}(q,p) = \varphi(|q|).\]
And, since $A_{i+1}$ is injective, there exists a function $\phi$ on the image of $A_{i+1}$ such that $\phi = \tilde{\phi} \circ A_{i+1}^{-1}$. 
Moreover, since $\phi$ has the function value $0$ along the boundary of $\mathrm{Im}(A_{i+1})$, $\phi$ extends to the whole Lagrangian skeleton $\mathrm{Skel}(F_i)$. 

With these auxiliary data, we could take an isotopy of a local model $L_{\mathcal{A}_{m+1}} \times \mathbb{R}^{n-m-1} \subset T^*\mathbb{R}^m \times \mathbb{R}^{n-m-1}$. 
We note that $\phi$ depends only on the $\mathbb{R}^{n-m-1}$-factor of the local model.
Thus, if we take a Lagrangian isotopy deforming $\mathrm{Skel}(F_i)$ to the graph of $-d \phi$, the isotopy modifies $\mathbb{R}^{n-m-1}$-factors in $T^*\mathbb{R}^{n-m-1}$. 
We note that the term ``graph of $-d \phi$'' makes sense only in the local model since it depends on the cotangent bundles structure, and thus, rigorously, we are abusing terminologies.
To be more precise, we could fix a finite collection of local coordinate charts modeled by a local model and take an isotopy in each local chart. 
The existence of such a finite collection can be guaranteed by the compactness of $\mathrm{Skel}(F_i)$. 

With the abused terminologies, the Legendrian lift of the Lagrangian after taking an isotopy can be coordinated as 
\begin{gather}
	\label{eqn lifted Legendrian in piece}
	\left\{\Big(x, -d\phi(x), \big(\cos ( c \phi(x) + \theta_*), \sin ( c \phi(x) + \theta_*)\big)\Big) | x \in \mathrm{Skel}(F_i)\right\}.
\end{gather}
We would like to point out is that one can choose sufficiently small $c$ in Equation \eqref{eqn lifted Legendrian in piece}, as mentioned in Section \ref{subsubsection product structure}. 
Thus, one can choose $c$ and $\theta_*$ satisfying that 
\begin{gather}
	\label{eqn inequality for IH 2} 
	-(i+1) \theta_0 < -c + \theta_* < \theta_* < -i \theta_0.
\end{gather}
Later, it implies that the $(i+1)^{th}$ step satisfies (IH 3). 

For the later use, we set notation. 
So far, we take a Legendrian isotopy of $\Lambda_{i+1}$ pushing some parts of it to the corner.
Let $\Lambda_{i+1}'$ denote the Legendrian we obtained after taking isotopy. 
As $A_{i+1}$ coordinates $\Lambda_{i+1}$, we can coordinate $\Lambda_{i+1}$ as follows:
\begin{gather}
	\label{eqn parametrization}
	j_{i+1}: \mathbb{D}_{3 \epsilon}^{n-k} \times S^{k-1} \to F_i\\
	\notag
	(p,q) \mapsto \left((p,q), -d\phi_(p,q)\right), \\
	\notag
	J_{i+1}: \mathbb{D}_{3 \epsilon}^{n-k} \times S^{k-1} \to F_i \times \mathbb{D}^2\\
	\notag
	(p,q) \mapsto \Big((p,q), -d\phi_{(p,q)}, \big(\cos (c \phi(p,q) + \theta_*), \sin (c \phi(p,q) + \theta_*)\big)\Big).
\end{gather}

One can easily see that through the above isotopy, we push the part 
\[\left(\mathbb{D}_{2\epsilon}^{n-k} \setminus \mathbb{D}_\epsilon^{n-k}\right) \times S^{k-1}\]
to the corners. 
If we identify $\mathbb{D}_{2\epsilon}^{n-k} \times S^{k-1} \subset \Lambda_{i+1}$ with $\partial_R h_{i+1}$, the part pushed to the corners can be identified with $\Lambda_{i+1}^{sub}$.

\subsubsection{Crossing the base}
\label{subsubsection crossing the base}
In the previous steps, we take a Legendrian isotopy of $\Lambda_{i+1}$ pushing $\Lambda_{i+1}^{sub}$ to the corner of $W_i$.
The resulting Legendrian is denoted as $\Lambda_{i+1}'$.
The next step is to isotope $\Lambda_{i+1}'$ so that the Legendrian after isotoping satisfies conditions (A) and (B) of Section \ref{subsection sketch of the induction step}.
After taking such a Legendrian isotopy, one can attach the subcritical handles $H_{i+1}^{ori} \simeq \check{H}_{i+1}^{ori} \times \mathbb{D}^2$ and $H_{i+1}^{n-1} \simeq \check{H}_{i+1}^{n-1} \times \mathbb{D}^2$ to $W_i = F_i \times \mathbb{D}^2$ while preserving the product structure.  

In order to do that, we will construct two one-parameter families of maps $\gamma_1^s$ and $\gamma_2^s$ for all $s \in [0,\pi]$. 
These two families will be defined on $\partial \mathbb{D}^{n-k}_{2\epsilon} \times S^{k-1} \times [0,1]$ and $\mathbb{D}^{n-k}_{2\epsilon} \times S^{k-1}$ respectively.
We note that 
\begin{align*}
	\partial \mathbb{D}^{n-k}_{2\epsilon} \times [0,1] \xrightarrow{\text{an identification}} \mathbb{D}_{2\epsilon}^{n-k} \setminus \mathbb{D}_\epsilon^{n-k} \\
	\partial \mathbb{D}^{n-k}_{2\epsilon} \times \{0\}, \partial \mathbb{D}^{n-k}_{2\epsilon} \times \{1\} \mapsto \partial \mathbb{D}_{2\epsilon}^{n-k}, \partial \mathbb{D}_{\epsilon}^{n-k}.
\end{align*}
If one rescales the domain of $\gamma_2^s$, the above identification will allow us to {\em concatenate} two families $\gamma_1^s$ and $\gamma_2^s$. 
At the end, the concatenation of them, defined on $\mathbb{D}_{2\epsilon}^{n-k} \times S^{k-1}$ will give a Legendrian isotopy connecting $\Lambda_{i+1}'$ and the desired Legendrian. 

The first family $\gamma_1^s$ is defined as follows:
\begin{gather}
	\label{eqn gamma_1}
	\gamma_1^s : \partial \mathbb{D}^{n-k}_{2\epsilon} \times S^{k-1} \times [0,1] \to \partial (F_i \times \mathbb{D}^2),\\
	\notag (p, q, t) \mapsto \big(\Phi_i^{-\frac{1}{c} t \sin s}(j_{i+1}(p,q)), (1-t) \cos (c\phi(p,q)+\theta_*) + t \cos(c\phi(p,q)+\theta_*+s), \\
	\notag (1-t) \sin (c\phi(p, q) + \theta_*) + t \sin(c\phi(p,q)+\theta_*+s)\big).
\end{gather}
We note that $j_{i+1}(p,q)$ is defined in Equation \eqref{eqn parametrization}, and $\Phi_i^t$ is defined in Section \ref{subsubsection a hamiltoian flow}.

One can check the followings:
\begin{enumerate}
	\item[(i)] $\operatorname{Im}(\gamma_1^s)$ is a Legendrian for any $s \in [0,\pi]$, and
	\item[(ii)] $\gamma_1^s(p,q,0) = J_{i+1}(p,q)$ for all $(p,q) \in \partial \mathbb{D}^{n-k}_{2\epsilon} \times S^{k-1}$. 
\end{enumerate}

The second item (ii) is easy to check by definition. 
In order to prove (i), one need to compute 
\begin{gather}
	\label{eqn gamma_1 legendrian}
	\left(\alpha_i + \tfrac{1}{c}(xdy-ydx)\right)\left(\gamma_{1*}^{s}(\partial t)\right)=0, \\
	\label{eqn gamma_1 legendrian 2}
	\left(\alpha_i + \tfrac{1}{c}(xdy-ydx)\right)(V) = 0 \text{  for all  } V \in T\left(\partial \mathbb{D}^{n-k}_{2\epsilon} \times S^{k-1}\right).
\end{gather}

For Equation \eqref{eqn gamma_1 legendrian}, we observe that 
\begin{gather*}
	\gamma_{1*}^{s}(\partial t) = \frac{\partial}{\partial t}\left(\Phi_i^{-\frac{1}{c} t \sin s}(j_{i+1}(p,q)\right) + \Big(- \cos \big(c\phi(p,q)+\theta_* \big) +  \cos\big(c\phi(p, q) + \theta_* + s\big)\Big)\partial x \\ + \Big(- \sin \big(c\phi(p, q) + \theta_*\big) +  \sin\big(c\phi(p, q) +\theta_* + s\big)\Big) \partial y.
\end{gather*}
When one plugs this vector into the contact form on the horizontal boundary, i.e., the one-form in Equation \eqref{eqn contact form on the horizontal boundary}, one obtains 
\begin{gather*}
	\left(\alpha_i + \tfrac{1}{c}(x dy- y dx)\right) \left(\gamma^s_{1*}(\partial t)\right) = \alpha_i\Big(\frac{\partial}{\partial t}\Phi_i^{-\frac{1}{c} t \sin s}\big(j_{i+1}(p,q)\big)\Big) +\\
	\tfrac{1}{c}\big((1-t) \cos (c\phi(p,q)+\theta_*) + t \cos(c\phi(p,q)+\theta_*+s)\big)\Big(- \sin \big(c\phi(p, q) + \theta_*\big) +  \sin\big(c\phi(p, q) +\theta_* + s\big)\Big)  \\
	-\tfrac{1}{c} \big((1-t) \sin (c\phi(p, q) + \theta_*) + t \sin(c\phi(p,q)+\theta_*+s)\big)\Big(- \cos \big(c\phi(p,q)+\theta_* \big) +  \cos\big(c\phi(p, q) + \theta_* + s\big)\Big) \\ 
	=- \tfrac{1}{c} \sin s + \tfrac{1}{c} \sin s = 0.
\end{gather*}
We note that 
\[\alpha_i\Big(\frac{\partial}{\partial t}\Phi_i^{-\frac{1}{c} t \sin s}\big(j_{i+1}(p,q)\big)\Big) = - \tfrac{1}{c} \sin s, \]  
since $\Phi_i^t$ is the Reeb flow on $\partial F_i$. 
Thus, Equation \eqref{eqn gamma_1 legendrian} holds. 

For Equation \eqref{eqn gamma_1 legendrian 2}, we observe that
\begin{gather*}
	\gamma_{1*}^s(V)= (\Phi_{i}^{-\frac{1}{c} t \sin s} \circ j_{i+1})_*(V) \\ + \Big((1-t)\sin\big(c\phi(p,q)+\theta_*\big) +t \sin\big(c\phi(p,q)+\theta_*+s\big)\Big) c V(\phi) \partial x \\ + \Big(-(1-t)\cos(c\phi(p,q)+\theta_*) -t \cos(c\phi(p,q)+\theta_*+s)\big)\Big) cV(\phi) \partial y \\
	= \left(\Phi_i^{-\frac{1}{c} t \sin s} \circ j_{i+1}\right)_*(V)
\end{gather*}
The last equality comes from the fact that $g$ is constant on $\partial \mathbb{D}^{n-k}_{2\epsilon} \times S^{k-1}$, so that $V(g)=0$.
Thus,
\begin{gather*}
	\left(\alpha_i + \tfrac{1}{c}(x dy - y dx)\right)(\gamma_{1*}^s(V)) = \alpha_i\left((\Phi_{i}^{-\frac{1}{c} t \sin s}\circ j_{i+1})_*(V)\right) \\
	= \left((\Phi_{i}^{-\frac{1}{c} t \sin s})^*\alpha_i\right)(j_{i+1 *} V) = \alpha_i (j_{i+1 *} V) = \lambda_i (j_{i+1 *} V) =0.
\end{gather*} 
The third equality holds since $\Phi_i^t$ is the Reeb flow on $\partial F_i$, and the others hold by definitions. 
This proves Equation \eqref{eqn gamma_1 legendrian 2}. 

In order to construct the second one-parameter family $\gamma_2^s$, we observe the following:
we note that $\Phi_i^{-\tfrac{1}{c} t \sin s}$ is a symplectomorphism on $F_i$.
By \cite[Lemma 11.2]{CE}, there is a function $h_s : F_i \to \mathbb{R}$ such that 
\begin{align}
	\label{eqn h_s}
	(\Phi_i^{-\frac{1}{c} t \sin s})^*(\lambda_i) = \lambda_i +d h_s.
\end{align}
Since on $\partial F_i$, $\Phi_i^{-\tfrac{1}{c} t \sin s}$ is the Reeb flow, $h_s|_{\partial F_i}$ is a constant function.
Thus, 
\[(\Phi_i^{-\frac{1}{c} t \sin s})^*(\alpha_i) = \alpha_i \text{  on  } \partial F_i.\]
We note that $\alpha_i := \lambda_i|_{\partial F_i}$.
Since $h_s$ is unique up to constant in Equation \eqref{eqn h_s}, we can choose $h_s$ such that $h_s|_{\partial F_i} \equiv 0$. 

We set $\gamma_2^s$ for $s \in [0,\pi]$ as follows:
\begin{gather}
	\label{eqn gamma_2}
	\gamma_2^s : \mathbb{D}^{n-k}_{2\epsilon} \times S^{k-1} \to \partial (F_i \times \mathbb{D}^2), \\
	\notag (p,q) \mapsto \Big(\left(\Phi_i^{-\frac{1}{c} \sin s}\circ j_{i+1}\right)(p,q), \cos\big(c\phi(p,q) + \theta_* + s + h_s(p,q)\big), \sin\big(c\phi(p,q) + \theta_* +s + h_s(p,q)\big) \Big).
\end{gather}

As similar to the case of $\gamma_1^s$, the following facts hold:
\begin{enumerate}
	\item[(iii)] $\operatorname{Im}(\gamma_2^s)$ is a Legendrian for any $s \in [0,\pi]$, and
	\item[(iv)] $\gamma_1^s(p,q,1) = \gamma_2^s(p,q)$ for all $(p,q) \in \partial \mathbb{D}^{n-k}_{2\epsilon} \times S^{k-1}$. 
\end{enumerate}
Since $\gamma_2^s$ is in the form of a lifted Legendrian, (iii) holds, and since $h_s|_{\partial F_i} \equiv 0$, (iv) holds.

Finally, we note that 
\[\mathrm{Im}\gamma_1^0 \cup \mathrm{Im}\gamma_2^0 = J_{i+1}(\mathbb{D}^{n-k}_{2\epsilon} \times S^{k-1}) \subset \Lambda_{i+1}'.\]
Then, we consider one-parameter family 
\[\left(\Lambda_{i+1}' \setminus (\mathrm{Im}\gamma_1^0 \cup \mathrm{Im}\gamma_2^0)\right) \cup (\mathrm{Im}\gamma_1^s \cup \mathrm{Im}\gamma_2^s),\]
that is parameterized by $s$. 
By smoothing each members of the one-parameter family, one have an one-parameter family of Legendrians starting $\Lambda_{i+1}'$. 
Let $\tilde{\Lambda}_{i+1}$ be the final Legendrian, i.e., the Legendrian for $s = \pi$.

\subsubsection{Attaching subcritical handles and induction hypotheses}
\label{subsubsection attaching subcritical handles}
Now, the rest for the inductive step is to attach two subcritical Weinstein handles $H_{i+1}^{ori}$ and $H_{i+1}^{n-1}$ and to check that the induction hypotheses (IH 1--6) hold after the attachments. 

We recall that by the construction above, after taking the Legendrian isotopy, a Legendrian $\Lambda_{i+1} \simeq \partial \overline{M}_i$ in the contact boundary of $W_i = F_i \times \mathbb{D}^2$ can be divided into three parts, $\Lambda_{i+1}^{sub}, \Lambda_{i+1}^{cri}$ and $\Lambda_{i+1}^{comp}$. 
After the above Legendrian isotopy, $\Lambda_{i+1}^{sub}$ is lying on the horizontal boundary of $W_i$ and is given as a product of a Legendrian on $\partial F_i$ and a diameter of $\mathbb{D}^2$.
The diameter is connecting two points 
\[\left(\cos(-c + \theta_*), \sin (-c + \theta_*)\right) \text{  and  } \left(\cos(-c + \theta_* + \pi), \sin (-c + \theta_* + \pi)\right).\]
See the formula in \eqref{eqn gamma_1}, which defines $\gamma_1^s$ for $s \in [0, \pi]$.
See also inequalities in \eqref{eqn inequality for IH 2} for the conditions on $\theta_*$.
Moreover, the Legendrian on $\partial F_i$ is also given by $\gamma_1^s$, and with the coordinates used in Section \ref{subsubsection crossing the base}, the Legendrian is written as 
\[\left\{\left(j_{i+1}(p,q)\right) | (p,q) \in \partial \mathbb{D}^{n-k}_{2\epsilon} \times S^{k-1}\right\}.\]

We attach two subcritical Weinstein handles 
\[H_{i+1}^{ori} \cup H_{i+1}^{n-1} \simeq D^*\left(S^{n-k-1} \times \mathbb{D}^k\right) \times D^*[0,1] \simeq \left(\check{H}_{i+1}^{ori} \cup \check{H}_{i+1}^{n-1}\right) \times \mathbb{D}^2,\]
along $\Lambda_{i+1}^{sub}$, or equivalently, the product of a Legendrian $\left(S^{n-k-1} \times \mathbb{D}^k\right)$ in $\partial F_i$ and the above diameter. 

It means that the fiber $F_{i+1}$ in the $(i+1)^{th}$ step of the inductive construction is obtained by attaching $\left(\check{H}_{i+1}^{ori} \cup \check{H}_{i+1}^{n-1}\right)$ to $F_i$ along the Legendrian on $\partial F_i$. 
Or equivalently, it is the same as attaching 
\[\left(\check{H}_{i+1}^{ori} \cup \check{H}_{i+1}^{n-1}\right) \simeq D^*\left(S^{n-k-1} \times \mathbb{D}^k\right)\]
along the boundary of the disk bundle of $A_{i+1}(S^{k-1})$. 
Note that the existence of the disk bundle is guaranteed by (IH 4) and Lemma \ref{lem disk bundle}. 
Moreover, the Lagrangian skeleton of $F_{i+1}$ is given as 
\[\mathrm{Skel}(F_{i+1}) = \mathrm{Skel}(F_i) \cup \left(L_{i+1} = S^{n-k-1} \times \mathbb{D}^k\right).\]
This construction also proves that at each boundary point of $\partial \overline{L}_{i+1}$, we would have a local chart that satisfies an arboreal singularity of $A_k$-tree type. 
We note that the local coordinate chart is given in Section \ref{subsubsection push to the corner}. 
It implies that {\em (IH 1) holds for the $(i+1)^{th}$ step}.

For (IH 2) and (IH 5), we observe that the Legendrian isotopy does not change $\partial h_j^n \subset \partial M_i$ for $j =0, \dots, i$. 
Thus, after taking the isotopy and attaching the subcritical handles, {\em for $j = 0, \dots, i$, (IH 2) and (IH 5) hold}, and it is enough to show (IH 2) and (IH 5) for $\partial h_{i+1}^n$.
We note that from the construction of $h_{i+1}^n$, one can see that $\partial h_{i+1}^n$ is a union of $\Lambda_{i+1}^{cri}$ and a part of $\partial h_{i+1}^{ori}$ and $\partial h_{i+1}^{n-1}$.

Since $h_{i+1}^{ori}$ and $h_{i+1}^{n-1}$ can be seen as the core Lagrangians of $H_{i+1}^{ori}$ and $H_{i+1}^{n-1}$, we would like to find the core Lagrangians first. 
From the equivalence given in \eqref{eqn subcritical parts}, we recall 
\[H_{i+1}^{ori} \cup H_{i+1}^{n-1} \simeq D^*\left(S^{n-k-1} \times \mathbb{D}^k\right) \times D^*[0,1],\]
again.
Then, the union of core Lagrangians of $H_{i+1}^{ori}$ and $H_{i+1}^{n-1}$ is given as the zero section of 
\[D^*\left(S^{n-k-1} \times \mathbb{D}^k\right) \times D^*[0,1],\]
i.e., the product of $L_{i+1} \subset F_{i+1}$ and the diameter connecting
\[\left(\cos(-c + \theta_*), \sin(-c + \theta_*)\right) \text{  and  } \left(\cos(-c + \theta_*+\pi), \sin(-c + \theta_*+\pi)\right).\]

The other part of $\partial h_{i+1}^n$, i.e., $\Lambda_{i+1}^{cri}$ is isotoped by $\gamma_2^s$. 
When one considers $\gamma_2^\pi$, the image is the isotoped $\Lambda_{i+1}^{cri}$.
And the Legendrian sphere equivalent to $\partial h_{i+1}^n$ is the union of the isotoped $\Lambda_{i+1}^{cri}$ and 
\[\left(L_{i+1}, \cos(-c + \theta_* + \pi), \sin (-c + \theta_* + \pi)\right) \subset F_{i+1} \times \mathbb{D}^2 \simeq W_{i+1}.\]

From Equations \eqref{eqn gamma_1} and \eqref{eqn gamma_2} defining $\gamma_1^s$ and $\gamma_2^s$, one can check the location of $\partial h_{i+1}^n$, and {\em it satisfies (IH 2)}. 
We note that the $F_i$ factor of $\gamma_1^\pi(p,q,t) \in \partial (F_i \times \mathbb{D}^2)$ is independent of $t$.
Thus, $h_\pi$ in Equation \eqref{eqn gamma_2} is constant. 
Moreover, $h_\pi$ could be defined to be the zero function along the contact boundary of $F_i$.
  
Moreover, one can easily see that $\partial h_{i+1}^n$ could be written as a Legendrian lift of $V_{i+1}$, where $V_{i+1}$ is a smoothing of $L_{i+1} \subset \mathrm{Skel}(F_{i+1})$ and the disk bundle of 
\[A_{i+1}(S^{k-1}) \subset \mathrm{Skel}(F_i) \subset \mathrm{Skel}(F_{i+1}).\]
It proves that {\em $\partial h_{i+1}^n$ satisfies (IH 5)}.

In order to check (IH 3) and (IH 6), we recall that 
\[\overline{M}_{i+1} = \overline{M}_i \cup h_{i+1}.\]
Thus, $\partial \overline{M}_{i+1}$ is the union of 
\[\partial h_{i+1} \setminus \partial_R h_{i+1} \text{  and  } \partial \overline{M}_i \setminus \partial_R h_{i+1}.\]
Let us find the corresponding part in $W_{i+1} = F_{i+1} \times \mathbb{D}^2$. 
The first part, $\partial h_{i+1} \setminus \partial_R h_{i+1}$, is contained in the boundary of $\partial h_{i+1}^{ori}$ and $\partial h_{i+1}^{n-1}$.
As we did above, we see $h_{i+1}^{ori}$ and $h_{i+1}^{n-1}$ as the core Lagrangians of $H_{i+1}^{ori}$ and $H_{i+1}^{n-1}$.
Then, the corresponding boundary part is given as 
\[\left(L_{i+1}, \cos(-c + \theta_*), \sin (-c + \theta_*)\right) \subset F_{i+1} \times \mathbb{D}^2 \simeq W_{i+1}.\]
The second part $\partial \overline{M}_i \setminus \partial_R h_{i+1}$ is denoted as $\Lambda_{i+1}^{comp}$. 
Thanks to the explicit formulas given in Sections \ref{subsubsection push to the corner} and \ref{subsubsection crossing the base}, one can find a Legendrian for the second part.

Since we have a Legendrian corresponding to $\partial \overline{M}_{i+1} \simeq \Lambda_{i+2} \subset \partial_\infty W_{i+1}$ explicitly, we can easily see that {\em (IH 3) and (IH 6) hold for the $(i+1)^{th}$ product Lefschetz fibration $W_{i+1} = F_{i+1} \times \mathbb{D}^2$.}

\begin{remark}
	\label{rmk Lagrangian bones}
	\mbox{}
	\begin{enumerate}
		\item Before discussing (IH 4), we would like to remark what parts of $\mathrm{Skel}(F_{i+1})$ correspond to $\partial \overline{M}_{i+1}$. 
		We note that by (IH 6) of the $i^{th}$ step, we have $\overline{V}_{i+1} \subset \mathrm{Skel}(F_i)$ corresponding to $\partial \overline{M}_i$, whose Legendrian lift is $\partial \overline{M}_i$. 
		And, from the construction above, we need to subtract the disk bundle of $A_{i+1}(S^{k-1})$ from $\overline{V}_{i+1}$, then add $L_{i+1} \subset \mathrm{Skel}(F_{i+1})$.
		The result $\L_{i+1} \sqcup \left(\overline{V}_{i+1} \setminus A_{i+1}(\mathbb{D}^{n-k \times S^{k-1}})\right)$ corresponds to$\partial \overline{M}_{i+1}$.  
		\item For a fixed $i$, one can see that $\mathrm{Skel}(F_i) = \overline{V}_{i+1} \cup \bigcup_{j=0}^i V_j$. 
		Moreover, it is easy to observe that every smooth point $p \in \mathrm{Skel}(F_i)$ is contained in exactly two Lagrangians in $\{V_0, \dots, V_i, \overline{V}_{i+1}\}$. 
		Also, Lagrangians do not have a self-intersecting point, i.e., $V_j$ and $\overline{V}_{i+1}$ are embedded Lagrangian. 
	\end{enumerate}
\end{remark}

Now, let us discuss (IH 4) for the $(i+1)^{th}$ step.
From the above arguments, one can observe that $\partial \overline{M}_{i+1}$ is Legendrian isotoped to a smoothing of the corresponding parts of the Lagrangian skeleton $\mathrm{Skel}(F_{i+1})$.
We note that the singularities of $\mathrm{Skel}(F_{i+1})$ can be seen as the boundary of $L_j$ for $j=1, \dots, i+1$.
In $\partial \overline{M}_{i+1}$, the singular part corresponds to the boundaries of the attaching regions of $h_j$ for $j =0, \dots, i+1$, i.e., $\partial_R h_j$. 
Without loss of generality, one can assume that the attaching sphere of $h_{i+2}$, as a submanifold of $\partial \overline{M}_{i+1}$, is transversal to the boundary of $\partial_R h_j$ for all $j =0, \dots, i+1$. 
It implies that {\em (IH 4) holds for the $(i+1)^{th}$ step}, and it completes the proof of the induction step. 

\begin{remark}
	\label{rmk abstract Lefschtz fibration}
	Let $D:=\{h_0, \dots, h_m\}$ be a handle decomposition of $M$, and let $\pi$ be the Lefschetz fibration of $T^*M$ obtained by applying Theorem \ref{thm main} to $D$. 
	Then, from the inductive constructions given in Section \ref{section the proof of Theorem main}, one can easily see that there is one-to-one relationship between $D$ and the collection of singular values of $\pi$. 
	Moreover, the order of handles in $D$ gives some restriction on the cyclic order of the vanishing cycles. 
	To be more precise, if $\pi$ can be written as an abstract Lefschetz fibration 
	\[(F:V_m, \dots, V_0),\]
	then, one can easily check that for $i = 0, \dots, N$, $V_i$ is a vanishing cycle corresponding to the critical Weinstein handle $H_i^n$, and for $i > N$, the $V_i$ corresponds to the collection of critical Weinstein handles $H_i$.
	Moreover, one can see that if $N< i < j \leq m$, then $V_i$ and $V_j$ are disjoint to each other.
	See Remark \ref{rmk Lagrangian bones} (2). 
	Thus, one can exchange the cyclic order of $V_i$ and $V_j$ without changing the vanishing cycles by Hurwitz move.
	Also, one can consider the order on $\{h_{N+1}, \dots, h_m\}$ in $D$ can be freely reordered, since all handles in the collection have the maximal index. 
\end{remark}

\section{Examples}
\label{section examples}
In Section \ref{section examples}, we will give examples of the inductive construction given in Section \ref{section the proof of Theorem main}. 

\subsection{An example of Theorem \ref{thm main}}
\label{subsect an example of the main theorem}
The example manifold $M$ we consider is the $2$-dimensional torus equipped with a specific handle decomposition $D$.
The given handle decomposition $D$ consists of one $0$-handle, two $1$-handles, and one $2$-handle as described in Figure \ref{figure decomposition of a torus}, a). 
The induced handle decomposition $\tilde{D}$ of $M$ is also descried in Figure \ref{figure decomposition of a torus}, b).

\begin{figure}[h]
	\centering
	\input{decomposition_of_a_torus.txt}		
	\caption{a) The square, both sides (resp.\ the top and the bottom) are identified to each other, is the torus. The torus is decomposed into one $0$-handle $h_0$ (center circle), two $1$-handles $h_1, h_2$ whose boundaries are red and blue lines respectively, and one $2$-handle $h_3$ (the rest). b) It describes the induced handle decomposition $\tilde{D}$ of a torus when $D$ is the given decomposition in a). In other words, for $i=1,2$, an $1$-handle $h_i$ is divided into two $1$-handles $h_i^{ori}, h_i^1$ and one $2$-handle $h_i^2$.}
	\label{figure decomposition of a torus}
\end{figure}
Figure \ref{figure M_i} describes $M_0, \dots, M_3$ defined in Equation \eqref{eqn def of M_i}.
\begin{figure}[h]
	\centering
	\input{M_i.txt}		
	\caption{a) describes $M_0$, i.e., union of $h_0^{ori}$ and $h_0^1$. Similarly, in b), c), and d) describe $M_1, M_2$ and $M_3$, respectively. For each $M_i$, the labeled handles are in $M_i \setminus M_{i-1}$}
	\label{figure M_i}
\end{figure}

\subsubsection{The base step}
The base step is to construct a product space $W_0 =F_0 \times \mathbb{D}^2$ which is equivalent to $T^*M_0$. 
As seen in Section \ref{subsection the base and the final steps}, $F_0 \simeq D^*S^1$ and $W_0 \simeq D^*S^1 \times \mathbb{D}^2$. 

Under the equivalence $T^*M_0 \simeq W_0$, the outer (resp.\ inner) boundary of $M_0$ is identified with the zero section of the fiber $\pi_0^{-1}(1) \simeq F_0 = D^*S^1$ (resp.\ $\pi_0^{-1}(-1)$). 
By using the notation in Section \ref{section the proof of Theorem main}, let $\partial \overline{M}_0$ denote the outer boundary of $M_0$ in $\pi_0^{-1}(1)$.
We would like to a Legendrian isotopy of $\partial \overline{M}_0=:\Lambda_1$, in order to construct $W_1 \simeq F_1 \times \mathbb{D}^2$ from $W_0 \simeq F_0 \times \mathbb{D}^2$. 

We note that $\Lambda_1$ is a Legendrian lift of the exact Lagrangian $L_0$ in the fiber $F_0$, where $L_0$ is the zero section of $F_0 \simeq D^*S^1$. 
Our plan is to take an exact Lagrangian isotopy of $L_0$, instead of $\Lambda_1$.
Then, by lifting the Lagrangian isotopy, one can obtain a Legendrian isotopy starting from $\Lambda_1$.

\subsubsection{Push to the corner of $W_0$.}
\label{subsubsection example 1}
The next step is to push the Legendrian $\Lambda_1$ to the corner of $W_0$, or equivalently, to push the exact Lagrangian $L_0$ to the boundary of $F_0$.

First, we specify the corresponding part of $L_0$ to $\partial_L H_1^{ori}$ and $\partial_L H_1^1$. 
We note that 
\[\partial_L H_1^{ori} = \partial_Rh_1^{ori}, \partial_L H_1^{n-1} = \partial_R h_1^{n-1}.\] 
Also, we recall that $h_1 = h_1^{ori} \cup h_1^{n-1} \cup h_1^2$. 

We remark the following:
Since $h_1$ is an $1$-handle, the attaching boundary is homeomorphic to $S^0 \times \mathbb{D}^1$.
Without loss of generality, one can identify $\partial_R h_1$ with $S^0 \times \mathbb{D}^1_{2\epsilon}$ where $\mathbb{D}^k_r$ means a $k$-dimensional disk of the radius $r$.
From the conditions (i) and (ii) in Section \ref{subsect technical statement}, one can assume that 
\[h_1^2 \cap \partial h_1 \simeq S^0 \times \mathbb{D}^1_\epsilon \subset S^0 \times \mathbb{D}^1_{2 \epsilon} \simeq \partial_R h_1.\] 

Under the identification $\Lambda_1 \simeq L_0$, one could embed $\partial_R h_1$ into $L_0$.
For convenience, let $\overline{j}_1 : S^0 \times \mathbb{D}^1_{2\epsilon} \hookrightarrow F_0$ denote the embedding of $\partial_R h_1 \hookrightarrow \partial L_1 \simeq \partial \Lambda_1$.  
Moreover, one can extend $\overline{j}_1$ slightly. 
Let $\overline{j}_1$ denote the extended embedding 
\[\overline{j}_1 : S^0 \times \mathbb{D}^1_{3\epsilon} \hookrightarrow F_0.\]
Figure \ref{figure fiber F_0} describes this. 
\begin{figure}[h]
	\centering
	\input{fiber_F_0.txt}		
	\caption{a) is the $M_1$ in b) of Figure \ref{figure M_i}. The outer circle of $M_0$ is $\Lambda_1 = \partial \overline{M}_0$ and the red (resp.\ blue) parts of $\Lambda_1$ are $\partial_R h_1^{ori}$ and $\partial_R h_1^1$ (resp.\ $\partial_R h_1^2 \cap \partial h_1$). 
		b) The rectangle is $F_0 \simeq D^*S^1$ and the zero section is $L_0$. Under $\Lambda_1 \simeq L_0$, the red and blue curves in b) correspond to the red and blue in a).}
	\label{figure fiber F_0}
\end{figure}

In order to modify $\mathrm{Im}(\overline{j}_1)$, we fix a function $g$ defined on $\mathrm{Im}(\overline{j}_1)$ as follows:
\begin{gather*}
	g : \mathbb{D}_{3 \epsilon} \times S^0 \to \mathbb{R},\\
	g(p,q) = - \varphi(|q|) \text{  for  } (p,q) \in \mathbb{D}^{n-k}_{3 \epsilon} \times S^0,
\end{gather*}
where $\varphi$ is the auxiliary function defined in Section \ref{subsection Legendrian isotopy}.
We note that in Section \ref{subsection Legendrian isotopy}, we used $\phi(p,q) = \varphi(|q|)$ instead of $g(p,q) = - \varphi(|q|)$.

Let $L_0'$ be the Lagrangian obtained from $L_0$ by replacing $\mathrm{Im}(\overline{j}_1)$ with the graph of $d g$. 
As in Section \ref{subsection Legendrian isotopy}, $L_0$ and $L_0'$ are Hamiltonian isotopic, and the Hamiltonian isotopy connecting $L_0$ and $L_0'$ induces a Legendrian isotopy connecting their Legendrian lifts. 
Let the new Legendrian obtained by isotoping $\Lambda_1$ be denoted by $\Lambda_1'$. 
Figure \ref{figure L_1'}, a) is $L_0'$ in $F_0$ and b) is the projection (to the base) image of $\Lambda_1'$.
\begin{figure}[h]
	\centering
\begingroup%
  \makeatletter%
  \providecommand\color[2][]{%
    \errmessage{(Inkscape) Color is used for the text in Inkscape, but the package 'color.sty' is not loaded}%
    \renewcommand\color[2][]{}%
  }%
  \providecommand\transparent[1]{%
    \errmessage{(Inkscape) Transparency is used (non-zero) for the text in Inkscape, but the package 'transparent.sty' is not loaded}%
    \renewcommand\transparent[1]{}%
  }%
  \providecommand\rotatebox[2]{#2}%
  \newcommand*\fsize{\dimexpr\f@size pt\relax}%
  \newcommand*\lineheight[1]{\fontsize{\fsize}{#1\fsize}\selectfont}%
  \ifx\svgwidth\undefined%
    \setlength{\unitlength}{283.46456693bp}%
    \ifx\svgscale\undefined%
      \relax%
    \else%
      \setlength{\unitlength}{\unitlength * \real{\svgscale}}%
    \fi%
  \else%
    \setlength{\unitlength}{\svgwidth}%
  \fi%
  \global\let\svgwidth\undefined%
  \global\let\svgscale\undefined%
  \makeatother%
  \begin{picture}(1,0.84)%
    \lineheight{1}%
    \setlength\tabcolsep{0pt}%
    \put(0,0){\includegraphics[width=\unitlength,page=1]{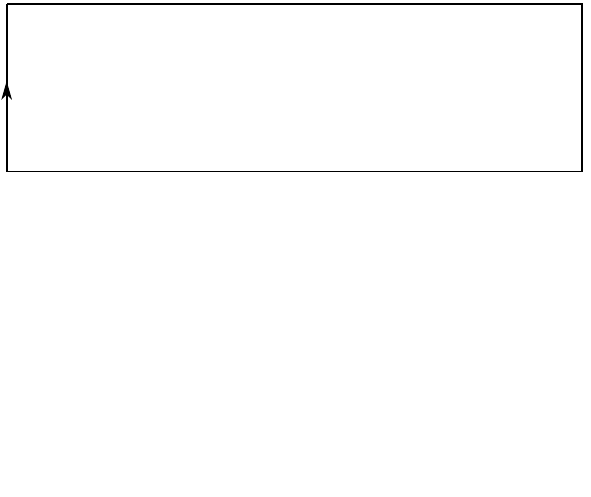}}%
    \put(0.43798977,0.51090279){\makebox(0,0)[lt]{\lineheight{1.25}\smash{\begin{tabular}[t]{l}a. $L_0'$\end{tabular}}}}%
    \put(0,0){\includegraphics[width=\unitlength,page=2]{L_1.pdf}}%
    \put(0.44473812,-0.00106651){\makebox(0,0)[lt]{\lineheight{1.25}\smash{\begin{tabular}[t]{l}b. $\pi_0(\Lambda_0')$\end{tabular}}}}%
  \end{picture}%
\endgroup%

	\caption{a) is $L_0'$ in $F_0$. The colored parts are matched to Figure \ref{figure fiber F_0}. b) is the image of $\Lambda_1'$ under $\pi_0$.}		
	\label{figure L_1'}
\end{figure}

As we did in \eqref{eqn parametrization}, one can formulate the part in $\Lambda_1'$, which corresponds to $\partial_R h_1$, or equivalently, the image of $\overline{j}_1$, by a map $j_1$. 
See \eqref{eqn parametrization}.
\begin{gather*}
	j_1(p,q) = \left((p,q), dg_{(p,q)}\right) \in F_0.
\end{gather*}
Similarly, let $J_1$ is also a map defined on $S^0 \times \mathbb{D}^1_{3 \epsilon}$ such that 
\begin{gather*}
	J_1(p,q) = \big(j_1(p,q), \cos(-c (g\circ j_1)(p,q)), \sin(-c (g\circ j_1)(p,q))\big) \in F_0 \times \mathbb{D}^2,
\end{gather*}
where the first component is a point in $F_0$, and the second and the last components are coordinated by the standard $(x,y)$-coordinates of $\mathbb{D}^2$.

\subsubsection{Crossing the base}
The next step is to take a Legendrian isotopy which makes our Legendrian $\Lambda_1'$ crosses the base. 
Since \eqref{eqn gamma_1} and \eqref{eqn gamma_2} give a Legendrian isotopy explicitly, we skip to explain how to isotope our example. 

Let $\tilde{\Lambda}_1$ denote the Legendrian obtained after taking a such isotopy, again.
Then, Figure \ref{figure tilde_L_1} describes $\pi_0(\tilde{\Lambda}_1)$ and $\pi_0(\tilde{\Lambda}_s)$ for some $s \in (0, 1)$, where $\Lambda_s$ denotes the Legendrian obtained by connecting $\gamma_1^{s\pi}$ and $\gamma_2^{s \pi}$. 
See Equations \eqref{eqn gamma_1} and \eqref{eqn gamma_2} for the definitions of $\gamma_1^{s\pi}$ and $\gamma_2^{s \pi}$. 

\begin{figure}[h]
	\centering
	\input{tilde_L_1.txt}
	\caption{The projection images of $\tilde{\Lambda}_s$ (left) and $\tilde{\Lambda}_1$ (right) are given. The red (resp.\ blue) part corresponds to $\mathrm{Im}(\gamma_1^s)$ and $\mathrm{Im}(\gamma_1^{\pi})$ (resp.\ $\mathrm{Im}(\gamma_2^s)$ and $\mathrm{Im}(\gamma_2^\pi)$). 
		We also note that the red and blue parts in Figure \ref{figure tilde_L_1} are connected by Legendrian isotopy to the red and blue parts in Figure \ref{figure L_1'}.}		
	\label{figure tilde_L_1}
\end{figure}

\subsubsection{Attaching subcritical handles}
\label{subsubsection example 2}
The next step is to attach subcritical handles $H_1^{ori}$ and $H_1^1$. 
We attach them along $\partial_L H_1^{ori}$ and $\partial_L H_1^1$.
More precisely, from the starting data, one has $\partial_L H_1^{ori}, \partial_L H_1^1 \subset \Lambda_1$. 
Let $\phi_t$ be the Legendrian isotopy constructed above such that $\phi_1(\Lambda_1) = \tilde{\Lambda}_1$.
Then, we attach $H_1^{ori}$ and $H_1^1$ along $\phi_1(\partial_L H_1^{ori})$ and $\phi_1(\partial_L H_1^1)$. 

We note that $\phi_1(\partial_L H_1^{ori}), \phi_1(\partial_L H_1^1) \subset \partial F_0 \times \mathbb{D}^2$, i.e., the horizontal boundary of $W_0 = F_0 \times \mathbb{D}^2$. 
Also, we note that 
\[H_1^{ori} \simeq \check{H}_1^{ori} \times \mathbb{D}^2, H_1^1 \simeq \check{H}_1^1 \times \mathbb{D}^2,\]
where $\check{H}_1^{ori}, \check{H}_1^1$ are 2-dimensional index 1 Weinstein handles. 

By attaching subcritical handles to $W_0$, we obtains 
\[W_1 := W_0 \cup H_1^{ori} \cup H_1^1 = \left( F_0 \cup \check{H}_1^{ori} \cup \check{H}_1^1\right) \times \mathbb{D}^2.\]
Since $W_1$ is a product of two Weinstein domains, we have a product Lefschetz fibration $\pi_1 : W_1 \to \mathbb{D}^2$. 
The regular fiber $F_1$ of $\pi_1$ is given in Figure \ref{figure F_1}.
Moreover, the construction of $W_1$ induces that $W_1$ is equivalent to $T^*M_1$. 
\begin{figure}[h]
	\centering
	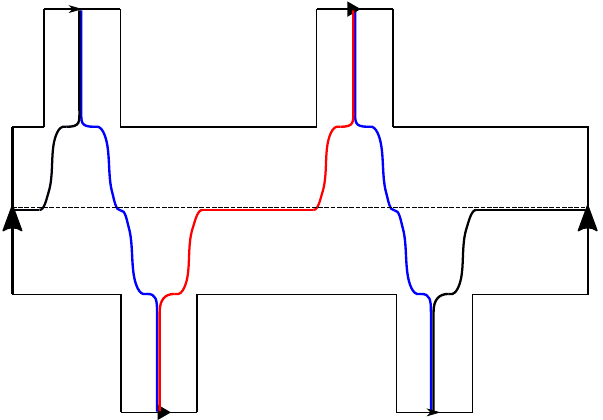
	\caption{The regular fiber $F_1$ is given. The edges with arrows are identified, and the black, blue, red and dashed curves are the images of $\partial M_1$ under $pr_1$.}		
	\label{figure F_1}
\end{figure}

In Figure \ref{figure M_i}, one can observe that $\partial M_1$ has four components.
The images of $\partial M_1$ under $pr_1 : W_1 \to F_1$, after smoothing, are given in Figure \ref{figure F_1}.
Also, one can observe that two of the four components will be used for attaching critical handles $H_0^2, H_1^2$. 

The attaching sphere of $H_0^2$ (resp.\ $H_1^2$) corresponds to the dashed (resp.\ blue) curve in Figure \ref{figure F_1}.
Moreover, by $\pi_1$, the attaching sphere of $H_0^2$ (resp.\ $H_1^2$) is projected to $-1$ (resp.\ a small interval $\left(\pi-\theta_1, \pi - \theta_1-c\right)$), where $\theta_1$ is a constant depending on the choice of the small positive number $c$ and the auxiliary function $\varphi$ above).

To be more precise, let us remark that one can check that $\theta_1 = c \varphi(2 \epsilon)$.
Thus, one can obtain an arbitrarily small $\theta_1$ by choosing sufficiently small $c$. 
By choosing a sufficiently small $\theta_1$, for example, by choosing $\theta_1$ such that $\theta_1 < \tfrac{1}{2} \theta_0$ where $\theta_0$ is a fixed constant satisfying $3 \theta_0 < \pi$, then one can check that the induction hypotheses (IH 2) will hold. 

The other two components of $\partial M_1$ are projected down to the interval 
\[\{ e^{-i\theta} \hs | \hs \theta \in [-\theta_1, 0]\} \subset \partial \mathbb{D}^2,\]
by $\pi_1$.
By a proper Legendrian isotoping, one can move them a little bit so that, after moving, the Legendrians are projected to 
\[\{ e^{-i\theta} \hs | \hs \theta \in [-\theta_0, 0)\} \subset \partial \mathbb{D}^2.\]
The Legendrian isotopy that one needs to apply is taking positive/negative Reeb flows of the Legendrians.

\subsubsection{Construction of $W_2$ from $W_1$}
By applying the inductive step, one can construct $W_2$ from $W_1$. 
Since the procedure is almost the same as the contents of Sections \ref{subsubsection example 1}--\ref{subsubsection example 2}, we omit the details.
See Figure \ref{figure F_2} for the resulting product space $W_2$. 

\begin{figure}[h]
	\centering
\begingroup%
  \makeatletter%
  \providecommand\color[2][]{%
    \errmessage{(Inkscape) Color is used for the text in Inkscape, but the package 'color.sty' is not loaded}%
    \renewcommand\color[2][]{}%
  }%
  \providecommand\transparent[1]{%
    \errmessage{(Inkscape) Transparency is used (non-zero) for the text in Inkscape, but the package 'transparent.sty' is not loaded}%
    \renewcommand\transparent[1]{}%
  }%
  \providecommand\rotatebox[2]{#2}%
  \newcommand*\fsize{\dimexpr\f@size pt\relax}%
  \newcommand*\lineheight[1]{\fontsize{\fsize}{#1\fsize}\selectfont}%
  \ifx\svgwidth\undefined%
    \setlength{\unitlength}{306.14173228bp}%
    \ifx\svgscale\undefined%
      \relax%
    \else%
      \setlength{\unitlength}{\unitlength * \real{\svgscale}}%
    \fi%
  \else%
    \setlength{\unitlength}{\svgwidth}%
  \fi%
  \global\let\svgwidth\undefined%
  \global\let\svgscale\undefined%
  \makeatother%
  \begin{picture}(1,1.26851852)%
    \lineheight{1}%
    \setlength\tabcolsep{0pt}%
    \put(0,0){\includegraphics[width=\unitlength,page=1]{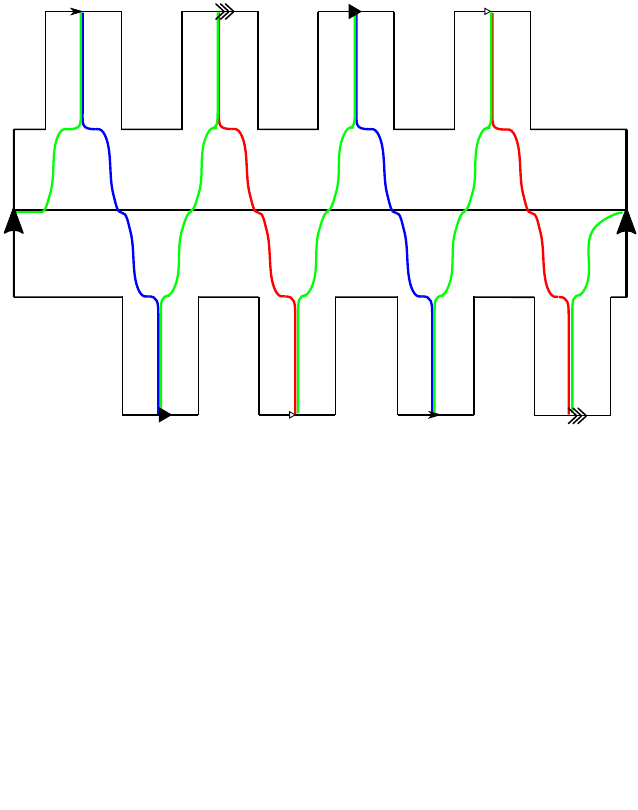}}%
    \put(0.29797354,0.56124041){\makebox(0,0)[lt]{\lineheight{1.25}\smash{\begin{tabular}[t]{l}a. projection images under $pr_2$\end{tabular}}}}%
    \put(0,0){\includegraphics[width=\unitlength,page=2]{F_2.pdf}}%
    \put(0.29796923,0.01306948){\makebox(0,0)[lt]{\lineheight{1.25}\smash{\begin{tabular}[t]{l}b. projection images under $\pi_2$\end{tabular}}}}%
  \end{picture}%
\endgroup%

	\caption{a) is the fiber $F_2$ together with $pr_2(\partial M_2)$. The colored curves are images of $\partial M_2$ under $pr_2$. b) is the base $\mathbb{D}^2$ together with $\pi_2(\partial M_2)$. The images of the same component of $\partial M_2$ are in the same color in a) and b).}		
	\label{figure F_2}
\end{figure}

\subsubsection{Attaching critical handles.}
The product space $W_2 \simeq F_2 \times \mathbb{D}^2$ is equivalent to $T^*M_2$.
Then, $\partial M_2$ are identified with a union of Legendrian spheres.
We note that $\partial M_2$ consists of four circles, thus, one has four Legendrian spheres on $\partial W_2$. 
The projected images of those four Legendrians, under $pr_2$ and $\pi_2$, are given in Figure \ref{figure F_2}.

With Figure \ref{figure F_2}, one can attach critical handles $H_0^2, H_1^2, H_2^2$ and $H_3$ along $\partial M_2$, by \cite[Proposition 8.1]{BEE}.
Then, one has a Lefschetz fibration of $T^*M$ with the fiber $F_2$ and four singular values. 
The four vanishing cycles are each Lagrangian spheres in $F_2$ given in Figure \ref{figure F_2} a), and their cyclic order is determined by Figure \ref{figure F_2} b).

\subsection{More examples}
\label{subsect examples}
In this subsection, we give more examples. 

The first example is the cotangent bundle of $\mathbb{RP}^2$. 
We consider the simplest handle decomposition of $\mathbb{RP}^2$, i.e., the handle decomposition consisting of one $0$-handle, one $1$-handle, and one $2$-handle. 
Let the handle decomposition be denoted by 
\[D = \{h_0, h_1, h_2\}.\]
Then, the induced $\tilde{D}$ and $W_{\tilde{D}}$ are 
\begin{gather*}
	\tilde{D}=\{h_0^{ori}, h_0^1, h_1^{ori}, h_1^1, h_0^2, h_1^2, h_2\}, \\
	W_{\tilde{D}} = \{H_0^{ori}, H_0^1, H_1^{ori}, H_1^1, H_0^2, H_1^2, H_2\}.
\end{gather*}

We note that the resulting Lefschetz fibration has three singular values. 
Also, the regular fiber admits a Weinstein handle decomposition
\[\{\check{H}_0^{ori}, \check{H}_0^1, \check{H}_1^{ori}, \check{H}_1^1\}.\]
The resulting Lefschetz fibration $\pi: T^*\mathbb{RP}^2 \to \mathbb{C}$ is given in Figure \ref{figure RP^2}.

\begin{figure}[h]
	\centering
\begingroup%
  \makeatletter%
  \providecommand\color[2][]{%
    \errmessage{(Inkscape) Color is used for the text in Inkscape, but the package 'color.sty' is not loaded}%
    \renewcommand\color[2][]{}%
  }%
  \providecommand\transparent[1]{%
    \errmessage{(Inkscape) Transparency is used (non-zero) for the text in Inkscape, but the package 'transparent.sty' is not loaded}%
    \renewcommand\transparent[1]{}%
  }%
  \providecommand\rotatebox[2]{#2}%
  \newcommand*\fsize{\dimexpr\f@size pt\relax}%
  \newcommand*\lineheight[1]{\fontsize{\fsize}{#1\fsize}\selectfont}%
  \ifx\svgwidth\undefined%
    \setlength{\unitlength}{306.14173228bp}%
    \ifx\svgscale\undefined%
      \relax%
    \else%
      \setlength{\unitlength}{\unitlength * \real{\svgscale}}%
    \fi%
  \else%
    \setlength{\unitlength}{\svgwidth}%
  \fi%
  \global\let\svgwidth\undefined%
  \global\let\svgscale\undefined%
  \makeatother%
  \begin{picture}(1,1.26851852)%
    \lineheight{1}%
    \setlength\tabcolsep{0pt}%
    \put(0.29797353,0.56124044){\makebox(0,0)[lt]{\lineheight{1.25}\smash{\begin{tabular}[t]{l}a. Fiber of $\pi$ with vanishing cycles\end{tabular}}}}%
    \put(0,0){\includegraphics[width=\unitlength,page=1]{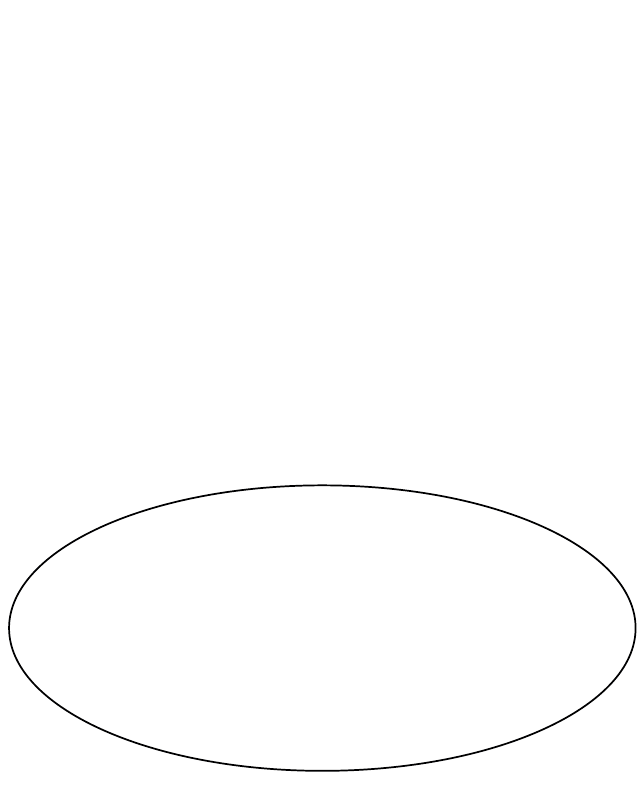}}%
    \put(0.40796922,0.01306944){\makebox(0,0)[lt]{\lineheight{1.25}\smash{\begin{tabular}[t]{l}b. base of $\pi$\end{tabular}}}}%
    \put(0,0){\includegraphics[width=\unitlength,page=2]{RP1.pdf}}%
  \end{picture}%
\endgroup%

	\caption{a) The regular fiber of $\pi$ with vanishing cycles. b) The base of $\pi$ with singular values. Colors indicate the correspondence of singular values and vanishing cycles.}		
	\label{figure RP^2}
\end{figure}

The next example is the three-dimensional torus 
\[\mathbb{D}^3 = S^1 \times S^1 \times S^1 = \mathbb{R} / \mathbb{Z} \times \mathbb{R} / \mathbb{Z} \times \mathbb{R} / \mathbb{Z}.\]
The last term of the above equations says that $\mathbb{T}^3$ is a cube, i.e., 
\[[0,1] \times [0,1] \times [0,1],\] with identified boundary. 
From this, one can easily consider the handle decomposition of $\mathbb{T}^3$ having one $0$-handle (corresponding to the unique vertex of the cube), three $1$-handles (corresponding to edges of the cube), three $2$-handles (corresponding to two-faces of the cube), and one $3$-handle.

Let 
\[D=\{h_0, h_1, h_2, h_3, h_4, h_5, h_6, h_7\}\]
denote the handle decomposition, where $h_0$ is the $0$-handle, $h_1, h_2, h_3$ are $1$-handles corresponding to the edges 
\[\mathbb{R}/\mathbb{Z} \times \{0\} \times \{0\}, \{0\} \times \mathbb{R}/\mathbb{Z}  \times \{0\}, \{0\} \times \{0\} \times \mathbb{R}/\mathbb{Z},\] 
$h_4, h_5, h_6$ are $2$-handles corresponding to the faces
\[\mathbb{R}/\mathbb{Z} \times \mathbb{R}/\mathbb{Z} \times \{0\}, \mathbb{R}/\mathbb{Z} \times \{0\} \times \mathbb{R}/\mathbb{Z},  \{0\} \times \mathbb{R}/\mathbb{Z} \times \mathbb{R}/\mathbb{Z},\] 
and $h_7$ is the $3$-handle.
Then,
\begin{gather*}
	\tilde{D}=\{h_0^{ori}, h_0^2, \dots, h_6^{ori}, h_6^2, h_0^3, \dots, h_6^3, h_7\}, \\
	W_{\tilde{D}} = \{H_0^{ori}, H_0^2, \dots, H_6^{ori}, H_6^2, H_0^3, \dots, H_6^2, H_7\}.
\end{gather*}

By the inductive construction, the resulting Lefschetz fibration has seven singular values and the regular fiber admits a Weinstein handle decomposition 
\[\{\check{H}_0^{ori}, \check{H}_0^2, \dots, \check{H}_6^{ori}, \check{H}_6^2\}.\]
We note that, in the above Weinstein handle decomposition of the fiber $F$, $\check{H}_0^{ori}$ is the unique $0$-handle, $\check{H}_1^{ori}, \dots, \check{H}_3^{ori}$ are $1$-handles, and others are $2$-handles. 
Thus, in a Kirby diagram for $F$, there are three pairs of spheres corresponding to three $1$-handles, and there are also ten Legendrian knots on $\partial \check{H}_0^{ori}$, i.e., the boundary of the zero handle for $F$ (or more precisely, their front projections), corresponding to ten $2$-handles. 

Moreover, we also note that the Legendrian knot corresponding to $\check{H}_0^{ori}$ can be drawn as a Legendrian unknot since $H_0^{ori} \cup H_0^2$ is equivalent to the cotangent bundle of $\mathbb{D}^1 \times S^2$. 
Similarly, for $i= 1, 2, 3$, the Legendrian knot corresponding $\check{H}_i^2$ can be drawn as a Legendrian unknot after sliding to $1$-handle $\check{H}_i^{ori}$. 

For $i=4, 5, 6$, the attaching spheres of $\check{H}_i^{ori}$ and $\check{H}_i^2$ are parallel to each other. 
Moreover, when one consider the handle decomposition of 
\[\{\check{H}_0^{ori}, \check{H}_1^{ori}, \check{H}_2^{ori}, \check{H}_4^{ori}\},\]
one can observe that the decomposition should be a handle decomposition for $T^*\mathbb{T}^2$. 

In order to draw a complete Kirby diagram, one should analyze how Legendrian links for $\check{H}_4^{ori}, \check{H}_5^{ori}, \check{H}_6^{ori}$ are related to each other. 
It could be achieved by analyzing the handle decomposition with respect to the product coordinate of 
\[\mathbb{T}^3 = \mathbb{R} / \mathbb{Z} \times \mathbb{R} / \mathbb{Z} \times \mathbb{R} / \mathbb{Z}.\]
However, we omit the procedure since it would be complicated.

\subsection{The case of cotangent bundles of surfaces}
\label{subsection the case of cotangetn bundles of surfaces}
The construction of Lefschetz fibration given in Section \ref{section the proof of Theorem main} inductively constructs Lefschetz fibrations of cotangent bundles from handle decompositions of the zero sections.
In this subsection, we focus on the case of cotangent bundles of surfaces, and we describe how to read off the resulting fibers and vanishing cycles directly from the input handle decompositions. 
We note that in the literature, there exists a construction of Lefschetz fibrations for cotangent bundles of surfaces, proven by Johns \cite{Johns}.
After introducing the practical recipe, we compare our construction with the construction in \cite{Johns} in this subsection.

Before proceeding further, let us list the reasons why we are describing the direct construction only for the lowest dimensional case.
\begin{itemize}
	\item First of all, in the inductive step for construction $F_{i+1}$ to $F_i$, we added $D^*\left(S^{n-k-1} \times \mathbb{D}^k\right)$ to $F_i$ if the handle $h_{i+1}$ has index $k < n$. 
	If the zero section is $2$-dimensional, then the only possible $k$ is $1$ for the induction step under the assumption that there exists a unique zero handle. 
	Thus, we add the same block  $D^*\left(S^0 \times \mathbb{D}^1\right)$ for each induction step, and it makes the construction simpler. 
	\item More importantly, by handle movements, we can assume that all $1$-handles in the input handle decomposition are attached to the unique zero handle. 
	It means that, when we construct $F_{i+1}$ from $F_i$, we attach the block $D^*\left(S^0 \times \mathbb{D}^1\right)$ to $F_0 \subset F_{i+1}$. 
	Thus, we can describe the attachment to $F_0$ part of $F_i$, in the $(i+1)^{th}$ inductive step, without constructing whole $F_i$. 
\end{itemize}

\noindent {\em Construction of the fiber}:
Now, let $M$ denote a surface and let 
\[\mathcal{H} = \{h_0, h_1, \dots, h_N, h_{N+1}, \dots, h_m\}\]
denote a given handle decomposition of $M$. 
Let us assume also that the input handle decomposition satisfies the following condition:
In the handle decomposition, the number of $0$-handle is one, and all $1$-handles are attached to the unique $0$-handle. 
Then, as we did in the previous sections, $h_0$ is the unique $0$-handle, $h_1, \dots, h_N$ are $1$-handles, and the other handles are $2$-handles. 

We note that by the base step, $F_0 \simeq D^*S^1$ and the zero section $S^1$ can be identified with the boundary of $h_0$. 
We also point out that the attaching sphere for an $1$-handle is $S^0 \simeq \text{two points}$.
Thus, if one identifies $\partial h_0 \simeq S^1 \simeq \mathbb{R} / \mathbb{Z}$, then the attachment of $h_j$ for $j =1, \dots, N$ can be encoded by two numbers 
\[0 \leq \theta_{j,1} < \theta_{j,2} <1, \text{  such that  } \theta_{j,a} \neq \theta_{i, b} \text{  for all } j \neq i \in \{1, \dots, N\}, a, b \in \{1,2\}. \]
a
Now, we choose two small numbers $\epsilon, \delta >0$ such that
\[\left\{[\theta_{j,a}-\epsilon-\delta, \theta_{j,a}+\epsilon+\delta] | j = 1, \dots, N, a =1, 2\right\}, \left\{[\theta_{j,a}+\epsilon-\delta, \theta_{j,a}+\epsilon+\delta] | j = 1, \dots, N, a =1, 2\right\}\] 
are collection of disjoint intervals.

For each $j = 1, \dots, N$, we will attach $D^*(S^0 \times \mathbb{D}^1)$ to $F_0$. 
To describe the attachment, we use the following identifications:
\begin{gather*}
	D^*S^1 \simeq S^1 \times [-1, 1], \text{  and  } D^*(S^0 \times \mathbb{D}^1) \simeq D^*\mathbb{D}^1 \sqcup D^*\mathbb{D}^1 \simeq [-\delta, \delta] \times [1,2] \sqcup [-\delta, \delta] \times [1,2].
\end{gather*}

For $j= 1, \dots, N$, if $h_j$ is attached to $h_0$ without twisting, i.e., the union of two handles $h_0$ and $h_j$ is orientable, then we attach two $[-\delta, \delta] \times [1,2]$ to $S^1 \times [-1, 1]$ by identifying 
\begin{itemize}
	\item $[-\delta,\delta] \times \{1\}$ to $[\theta_{j,1} -\epsilon - \delta, \theta_{j,1} -\epsilon +\delta] \times \{1\} \subset S^1 \times [-1,1]$ and $[-\delta,\delta] \times \{2\}$ to $[\theta_{j,2} +\epsilon - \delta, \theta_{j,2} +\epsilon +\delta] \times \{-1\} \subset S^1 \times [-1,1]$ for the first $[-\delta, \delta] \times [1,2]$, and 
	\item $[-\delta,\delta] \times \{1\}$ to $[\theta_{j,1} +\epsilon - \delta, \theta_{j,1} +\epsilon +\delta] \times \{-1\} \subset S^1 \times [-1,1]$ and $[-\delta,\delta] \times \{2\}$ to $[\theta_{j,2} -\epsilon - \delta, \theta_{j,2} -\epsilon +\delta] \times \{1\} \subset S^1 \times [-1,1]$ for the second $[-\delta, \delta] \times [1,2]$.
\end{itemize}
For $j= 1, \dots, N$, if $h_j$ is attached to $h_0$ with twisting, i.e., the union of two handles $h_0$ and $h_j$ is nonorientable, then we attach two $[-\delta, \delta] \times [1,2]$ to $S^1 \times [-1, 1]$ by identifying 
\begin{itemize}
	\item $[-\delta,\delta] \times \{1\}$ to $[\theta_{j,1} -\epsilon - \delta, \theta_{j,1} -\epsilon +\delta] \times \{1\} \subset S^1 \times [-1,1]$ and $[-\delta,\delta] \times \{2\}$ to $[\theta_{j,2} -\epsilon - \delta, \theta_{j,2} -\epsilon +\delta] \times \{1\} \subset S^1 \times [-1,1]$ for the first $[-\delta, \delta] \times [1,2]$, and 
	\item $[-\delta,\delta] \times \{1\}$ to $[\theta_{j,1} +\epsilon - \delta, \theta_{j,1} +\epsilon +\delta] \times \{-1\} \subset S^1 \times [-1,1]$ and $[-\delta,\delta] \times \{2\}$ to $[\theta_{j,2} +\epsilon - \delta, \theta_{j,2} +\epsilon +\delta] \times \{-1\} \subset S^1 \times [-1,1]$ for the second $[-\delta, \delta] \times [1,2]$.
\end{itemize}
We note that when one attaches two $[-\delta, \delta] \times [1,2]$ to $S^1 \times [-1, 1]$, the resulting space should be orientable. 

\begin{remark}
	\mbox{}
	\begin{enumerate}
		\item We would like to point out that the usage of the letter $\epsilon$ in this subsection is different of that in Sections \ref{section the proof of Theorem main}, \ref{subsect an example of the main theorem}, and \ref{subsect examples}. 
		In the previous sections, $D^*\left(S^{n-k-1} \times \mathbb{D}^k\right)$ is attached by using $(3\epsilon)$-disk-neighborhood of the attaching sphere. 
		Especially, if one compares the notations in Section \ref{subsection Legendrian isotopy} and the present subsection, one can observe that $[\theta_{j,1} -\epsilon - \delta, \theta_{j,1} -\epsilon +\delta]$ and the other intervals in the above paragraph correspond to $[\epsilon, 2\epsilon]$ part of Section \ref{subsection Legendrian isotopy}, see Figure \ref{figure graph}. 
		Thus, one can observe that $\epsilon$ (resp.\ $\delta$) in the present subsection corresponds to $\tfrac{3}{2}\epsilon$ (resp.\ $\tfrac{1}{2}\epsilon$) in the previous sections. 
		\item The reason we use two different letters $\epsilon$ and $\delta$ in the present subsection is that we would like to control the size of $\epsilon$ independently of the size of $\delta$. 
		By doing that, we can observe that our inductive method gives the same answer with \cite{Johns}. 
	\end{enumerate}
\end{remark}

See Figure \ref{figure 2dim_exam}. 
Figure \ref{figure 2dim_exam} a) corresponds to the handle decomposition of torus, which we considered in Section \ref{subsect an example of the main theorem} and Figure \ref{figure F_2}, and b) corresponds to the $\mathbb{RP}^2$ example that we described in Section \ref{subsect examples} and Figure \ref{figure RP^2}.
\begin{figure}[h]
	\centering
	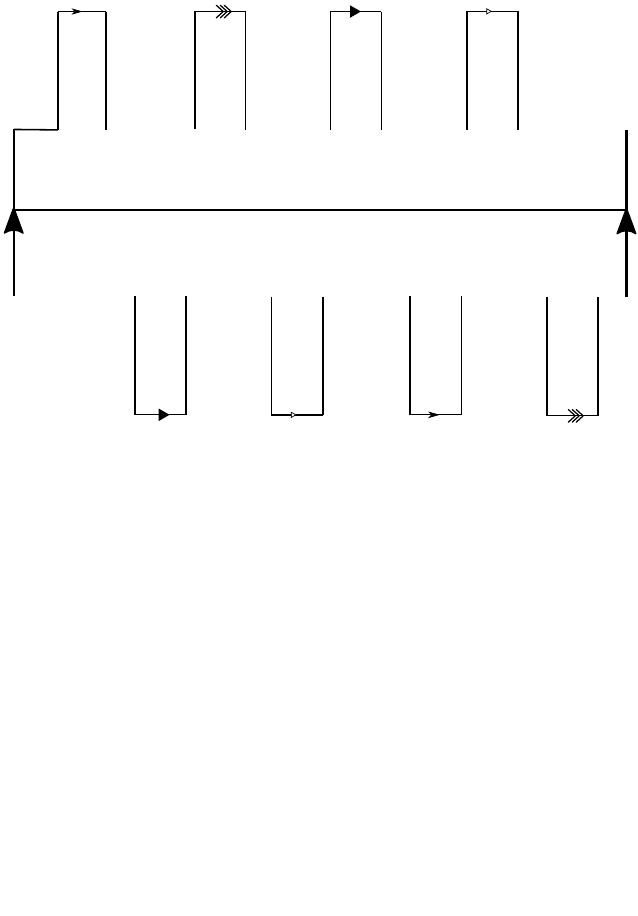
	\caption{a) is the resulting fiber when the input handle decomposition is the one described in Figure \ref{figure decomposition of a torus}. We specify $[\theta_{1,1} -\epsilon - \delta, \theta_{1,1} -\epsilon +\delta] \times \{1\}$ part as the red part. b) is the resulting fiber when the input handle decomposition is the one of $\mathbb{RP}^2$, described in Section \ref{subsect examples}.}
	\label{figure 2dim_exam}
\end{figure}

\noindent {\em Vanishing cycles}:
In order to describe the vanishing cycles, let us note that in the above construction of the fiber, if one takes smaller $\epsilon$, the Weinstein isotopy class of the resulting fiber does not change. 
Moreover, even when one choose $\epsilon = 0$, the resulting fiber can be understood as a plumbing space. 
More precisely, the resulting fiber can be obtained by plumbing one $D^*S^1$ for each $j =1, \dots, N$ at two points $\theta_{j,1}$ and $\theta_{j,2}$. 

Then, for each $j =1, \dots, N$, there exists the zero section of the plumbed $D^*S^1$, and it becomes the vanishing cycle corresponding to $\partial h_j^2$ for $j =1, \dots, N$. 
And, we have the zero section of $F_0 \simeq D^*S^1$.
The zero section becomes the vanishing cycle corresponding to $\partial h_0^2$. 
For convenience, let $L_j$ denote the vanishing cycles corresponding to $\partial h_j^2$.

Finally, we need to find the vanishing cycles corresponding to $\partial h_j$ with $j >N$. 
From the inductive construction, one can easily observe that that part corresponds to the Lagrangian surgery of $L_0$ and $\cup_{j=1}^N L_j$ at every plumbing points $\theta_{j,a}$. 
We note that there are two possible ways of Lagrangian surgeries, and we always take the surgery ``bending $L_0$ to left". 
See Figures \ref{figure F_2} and \ref{figure RP^2}, the vanishing cycles are in green (Figure \ref{figure F_2}) and in red (Figure \ref{figure RP^2}).

\noindent{\em Comparison to the result of \cite{Johns}}: We do not explain the algorithm given in \cite{Johns}, but we note that Johns gave an algorithmic construction of Lefschetz fibrations of cotangent bundles of surfaces. 
See \cite[Section 4]{Johns}. 
The idea of \cite{Johns} is to ``complexifying a Morse function of the zero section". 

It is easy to check that the algorithms given in the present subsection and \cite{Johns} give the same fiber and vanishing cycles. 
We expect that it happens since ``complexifying a Morse function" is equivalent to a canceling pair for each subcritical point of the Morse function, but we do not try to prove in the present paper.

\section{The effects of handle moves}
\label{section the effects of handle moves}
Theorem \ref{thm main} gives infinitely many Lefschetz fibrations on a cotangent bundle $T^*M$. 
In Section \ref{section the effects of handle moves}, we discuss how those Lefschetz fibrations of $T^*M$ are related to each other for the case of $\mathrm{dim}M=2$. 
As the result, we show that all Lefschetz fibrations of $T^*M$ constructed by Theorem \ref{thm main} are connected by four moves which are introduced in Section \ref{subsect four moves}.

\subsection{Four moves}
\label{subsect four moves}
Let $(F;L_1, \dots, L_m)$ be an abstract Lefschetz fibration. 
Then, it is well-known that the total space of $(F ; L_1, \dots, L_m)$ is equivalent to the total space of another abstract Lefschetz fibration obtained by applying one of the following four operations:
\begin{itemize}
	\item {\em Deformation} means a simultaneous Weinstein deformation of $F$ and exact Lagrangian isotopy of $(L_1, \dots, L_m)$. 
	\item {\em Cyclic permutation} is to replace the ordered collection $(L_1, \dots, L_m)$ with $(L_2, \dots, L_m, L_1)$.
	In other words, 
	\[(F; L_1, \dots, L_m) \simeq (F;L_2,\dots, L_m, L_1). \] 
	The equivalence means that their total spaces are equivalent.
	\item {\em Hurwitz moves.} Let $\tau_i$ denote the symplectic Dehn twist around $L_i$. {\em Hurwitz move} is to replace $(L_1, \dots, L_m)$ with either $(L_2, \tau_2(L_1), L_3, \dots, L_n)$ or $(\tau_1^{-1}(L_2), L_1, L_3, \dots, L_m)$, i.e.,
	\[(F; L_1, \dots, L_m) \simeq (F;L_2, \tau_2(L_1), \dots, L_m) \simeq (F; \tau_1^{-1}(L_2), L_1, \dots, L_m). \]
	\item {\em Stabilization.} Let $\operatorname{dim}F = 2n-2$, or equivalently, the total space is of dimension $2n$. For a parameterized Lagrangian disk $D^{n-1} \hookrightarrow F$ with Legendrian boundary $S^{n-2} = \partial D^{n-1} \hookrightarrow \partial F$ such that $0 = [\lambda] \in H^1(D^{n-1},\partial D^{n-1})$ where $\lambda$ is the Liouville form, replace $F$ with $\tilde{F}$, obtained by attaching a $(2n-2)$-dimensional Weinstein $(n-1)$-handle to $F$ along $\partial D^{n-1}$, and replace $(L_1, \dots, L_m)$ with $(\tilde{L}, L_1, \dots, L_m)$, where $\tilde{L} \subset \tilde{F}$ is obtained by gluing together $D^{n-1}$ and the core of the handle.
	In other words,
	\[(F;L_1, \dots, L_m) \simeq (\tilde{F}; \tilde{L}, L_1, \dots, L_m). \]
\end{itemize}  
See \cite[Section 1.2]{GP} for more details. 

\begin{remark}
	\label{rmk oepn quetsion}
	As cited in \cite{GP}, it is natural to ask whether any two Lefschetz fibrations of a fixed Weinstein manifold can be connected by a finite sequence of the above four moves. 
	In the current paper, we do not claim that the four moves are enough to connect every Lefschetz fibrations of $T^*M$, but we claim that they are enough to connect all Lefschetz fibrations obtained by applying Theorem \ref{thm main}, when $\operatorname{dim}M=2$.  
\end{remark}

\subsection{Equivalence of Lefschetz fibrations}
\label{subsect equivalence of Lefschetz fibrations}
We start Section \ref{subsect equivalence of Lefschetz fibrations} with a remark. 
\begin{remark}
	\label{rmk one to one}
	\mbox{}
	\begin{enumerate}
		\item Let $D$ be a handle decomposition of $M$ and let $\pi$ be the Lefschetz fibration of $T^*M$ produced from $D$. Then, there are one-to-one relations between the following three sets:
		\begin{itemize}
			\item the set of handles in $D$, 
			\item the set of critical handles in $W_{\tilde{D}}$, and
			\item the set of vanishing cycles of $\pi$. 
		\end{itemize}
		See also Remark \ref{rmk abstract Lefschtz fibration}.
		This fact will be heavily used in the rest of Section \ref{section the effects of handle moves}.
		\item We also want to remark that since we would like to use the description of the resulting fiber given in Section \ref{subsection the case of cotangetn bundles of surfaces}, we are assuming that every $1$-handle in a given handle decomposition $D$ is attached to the unique zero handle. 
		See the second bulleted item in Section \ref{subsection the case of cotangetn bundles of surfaces}.
	\end{enumerate}
\end{remark}

In Sections \ref{subsect equivalence of Lefschetz fibrations} --\ref{subsect sliding}, we prove the following Proposition. 
\begin{proposition}
	\label{prop handle moves}
	If $M$ is a $2$-dimensional manifold, then all Lefschetz fibration of $T^*M$ obtained by applying Theorem \ref{thm main} are connected to each other by a finite sequence of the four moves in Section \ref{subsect four moves}.
\end{proposition}
\begin{proof}
	It is well-known that any two handle decomposition $D_1$ and $D_2$ of the same manifold $M$ are connected by a finite sequence of three operations, {\em a change of order of handles, a cancellation of a canceling pair} and {\em a handle sliding}.
	Because $\operatorname{dim}M=2$, and because every handle decomposition has only one $0$-handle by Definition \ref{def handle decomposition}, it is enough to consider the following four cases:
	\begin{itemize}
		\item The first case is to change orders of handles. 
		\item The second case is to cancel a canceling pair consisting of a $1$-handle and a $2$-handle.
		\item The third (resp.\ the last) case is to slide a $1$-handle along another $1$-handle without twisting (resp.\ with twisting). 
	\end{itemize}
	The last three cases are described in Figure \ref{figure handle moves}.
	
	\begin{figure}[h]
		\centering
		\input{handle_moves.txt}		
		\caption{The super script means the index of each handle. Note that the figures do not contain the whole $1$-handle $h^1_2$ in b). and c).
			a). the operation is the cancellation of a canceling pair consisting of $h^1$ and $h^2$. b). A $1$-handle $h^1_2$ is sliding along $h^1_1$, a $1$-handle without twisting. c). A $1$-handle $h^1_2$ is sliding along $h^1_1$, a $1$-handle with twisting.}
		\label{figure handle moves}
	\end{figure}
	
	In order to discuss the first case, let $D_1 :=\{h_0, \dots, h_m\}$ be a handle decomposition of $M$.
	If $D_2$ is another handle decomposition of $M$ obtained by switching the order of $h_i$ and $h_j$, one can observe that $h_i$ and $h_j$ have the same index.
	
	Let $\pi_i$ denote the Lefschetz fibration obtained by applying Theorem \ref{thm main} to $D_i$. 
	Also, we assume that $i < j $ without loss of generality.   
	If $h_i$ and $h_j$ are $2$-handles, one can easily check that $\pi_1$ and $\pi_2$ are connected by multiple Hurwitz moves. 
	
	To be more precisely, let $\pi_1$ be the following abstract Lefschetz fibration
	\[\pi_1 = (F: L_m, \dots, L_i, \dots, L_j, \dots, L_0).\]
	Similarly, $\pi_2$ can be written as the following abstract Lefschetz fibration
	\[\pi_2 = (F: L_m, \dots, L_j, \dots, L_i, \dots, L_0).\]
	Moreover, 
	\[L_i \cap L_k =\varnothing, L_j \cap L_k = \varnothing \text{  for all  } k = i+1, \dots, j-1,\]
	since 
	\begin{gather}
		\label{eqn change order}
		\partial h_i \cap \partial h_k =\varnothing, \partial h_j \cap \partial h_k = \varnothing \text{  for all  } k = i+1, \dots, j-1.
	\end{gather}
	We note that for all $k = i, \dots, j$, $h_k$ should be of index $2$. 
	Thus, Equation \eqref{eqn change order} holds. 
	Now, since $L_i \cap L_k = \varnothing$ (resp.\ $L_j \cap L_k = \varnothing$), $\tau_k(L_i) = L_i$ (resp.\ $\tau_k^{-1}(L_j)=L_j$), where $\tau_k$ is a Dehn twist around $L_k$. 
	Thus, by operating Hurwitz moves, one can change the order of $L_i$ (resp.\ $h_j$) and $L_k$ for all $k = i+1, \dots, j-1$, without changing the vanishing cycles. 
	
	If $h_i$ and $h_j$ are $1$-handles, their attaching regions are disjoint subsets of $\partial h_0$.
	See Remark \ref{rmk one to one} (2).
	Thus, from the construction of $\pi_1$, one could observe that the vanishing cycle corresponding to $\partial_L H_i^2 = \partial h_i$, i.e., $L_i$, does not intersect with vanishing cycles corresponding to $\partial_L H_k^2 = \partial h_k^2$, i.e., $L_k$, for all $k = i+1, \dots, j$.  
	Similarly, the vanishing cycles corresponding to $\partial_L H_j^2 = \partial h_j^2$, i.e., $L_j$, does not intersect with vanishing cycles corresponding to $\partial_LH_k^2 = \partial h_k^2$, i.e., $L_k$, for $k = i, \dots, j-1$.
	Then, by the same logic, one can check that $\pi_1$ and $\pi_2$ are connected by a sequence of Hurwitz moves such that the sequence changes the order of vanishing cycles without changing the vanishing cycles. 
	
	Now, it is enough to consider the three cases described in Figure \ref{figure handle moves}.   
	Lemma \ref{lemma handle canceling} cares the handle cancellation, and Lemma \ref{lemma sliding} cares the handle sliding.
	Thus, Lemmas \ref{lemma handle canceling}--\ref{lemma sliding} complete the proof.  
\end{proof}

\begin{lemma}
	\label{lemma handle canceling}
	If a handle decomposition $D_2$ is obtained from $D_1$ by a cancellation of a canceling pair, then $\pi_1$ and $\pi_2$ are connected to each other by four moves.
\end{lemma}

\begin{lemma}
	\label{lemma sliding}
	If a handle decomposition $D_2$ is obtained from $D_1$ by sliding an $1$-handle along another $1$-handle (with or without twisting), then $\pi_1$ and $\pi_2$ are connected to each other by four moves.
\end{lemma}

\begin{remark}
	\label{rmk two facts}
		Before proving Lemmas \ref{lemma handle canceling}--\ref{lemma sliding}, we would like to point out that, according to the algorithm given by Theorem \ref{thm main}, the regular fiber $F$ admits a Weinstein handle decomposition.
		Moreover, by dimension reason, the regular fiber of $\pi_i$ is given by the disc cotangent bundle $D^*S^1$ and Weinstein $1$-handles. 
		This allows us to draw Figures \ref{figure cancel}--\ref{figure sliding2} that play a key role in our proof. 
\end{remark}

\subsection{Proof of Lemma \ref{lemma handle canceling}}
\label{subsect canciling}
The strategy for proving Lemma \ref{lemma handle canceling} is the following:
we start the proof by drawing a local figure of $\pi_1$.
We point out that $\pi_1$ is an abstract Lefschetz fibration, thus, a local figure of $\pi_1$ means a local figure of the fiber $F_1$ together with vanishing cycles. 
Then, we operate a sequence of four moves, and it induces a sequence of Lefschetz fibrations.
At the end, we stop when we have a local figure corresponding to $\pi_2$.
We note that $\pi_i$ is obtained by applying Theorem \ref{thm main} for $D_i$, and $D_1$ (resp.\ $D_2$) is modeled in Figure \ref{figure handle moves}, a) left (resp. right).

\begin{figure}[p]
	\centering
\begingroup%
  \makeatletter%
  \providecommand\color[2][]{%
    \errmessage{(Inkscape) Color is used for the text in Inkscape, but the package 'color.sty' is not loaded}%
    \renewcommand\color[2][]{}%
  }%
  \providecommand\transparent[1]{%
    \errmessage{(Inkscape) Transparency is used (non-zero) for the text in Inkscape, but the package 'transparent.sty' is not loaded}%
    \renewcommand\transparent[1]{}%
  }%
  \providecommand\rotatebox[2]{#2}%
  \newcommand*\fsize{\dimexpr\f@size pt\relax}%
  \newcommand*\lineheight[1]{\fontsize{\fsize}{#1\fsize}\selectfont}%
  \ifx\svgwidth\undefined%
    \setlength{\unitlength}{226.77165354bp}%
    \ifx\svgscale\undefined%
      \relax%
    \else%
      \setlength{\unitlength}{\unitlength * \real{\svgscale}}%
    \fi%
  \else%
    \setlength{\unitlength}{\svgwidth}%
  \fi%
  \global\let\svgwidth\undefined%
  \global\let\svgscale\undefined%
  \makeatother%
  \begin{picture}(1,2.45)%
    \lineheight{1}%
    \setlength\tabcolsep{0pt}%
    \put(0,0){\includegraphics[width=\unitlength,page=1]{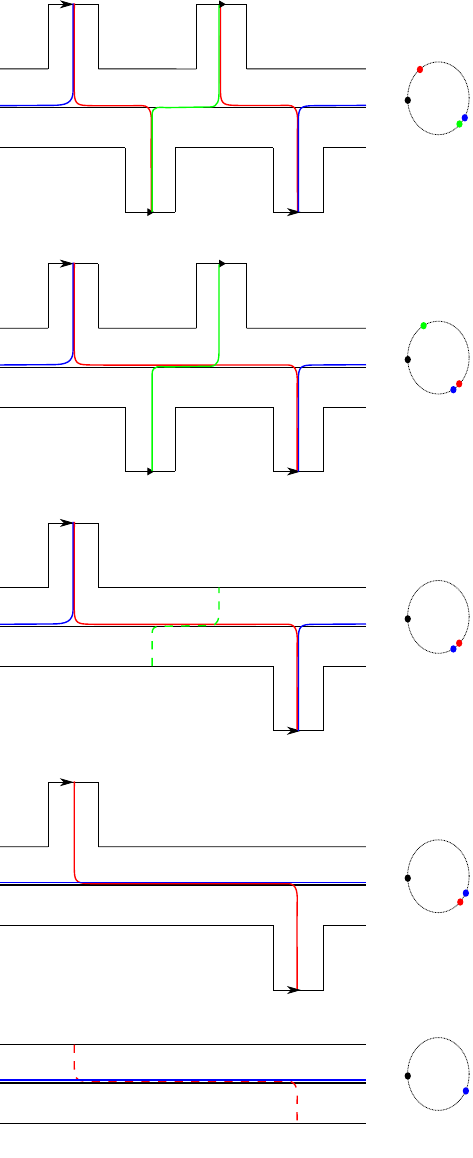}}%
    \put(0.38792658,1.95145353){\makebox(0,0)[lt]{\lineheight{1.25}\smash{\begin{tabular}[t]{l}a).\end{tabular}}}}%
    \put(0.3879286,1.40258889){\makebox(0,0)[lt]{\lineheight{1.25}\smash{\begin{tabular}[t]{l}b).\end{tabular}}}}%
    \put(0.3879286,0.85372441){\makebox(0,0)[lt]{\lineheight{1.25}\smash{\begin{tabular}[t]{l}c).\end{tabular}}}}%
    \put(0.41364189,0.29214902){\makebox(0,0)[lt]{\lineheight{1.25}\smash{\begin{tabular}[t]{l}d).\end{tabular}}}}%
    \put(0.41364439,0.00170999){\makebox(0,0)[lt]{\lineheight{1.25}\smash{\begin{tabular}[t]{l}e).\end{tabular}}}}%
  \end{picture}%
\endgroup%
		
	\caption{By a sequence of four moves, one can convert a) to e). For each of a) -- e), the lefts are local pictures of fibers together with vanishing cycles (colored curves), and the right circles indicate the cyclic order of vanishing cycles.}
	\label{figure cancel}
\end{figure}

Figure \ref{figure cancel}, a) is the local picture for $\pi_1$. 
In the local picture, there are four vanishing cycles that correspond to handles in Figure \ref{figure handle moves}, a) left.
The correspondence are given as follows:
\begin{itemize}
	\item The black curve corresponds to the $0$-handle $h^0$.
	\item The red curve corresponds to the $1$-handle $h^1$. 
	\item The green curve corresponds to the $2$-handle $h^2$.
	\item The blue curve corresponds to the $2$-handle which is adjacent to $h^1$, and which is not $h^2$. 
\end{itemize}
One can also observe that in the cyclic order, the green and blue come the first since they correspond to $2$-handles, the red is the next since it correspond to the $1$-handle, and the black comes the last since it corresponds to the $0$-handle. 
See Remark \ref{rmk abstract Lefschtz fibration}.
We note that the order between blue and green vanishing cycles are not important because they do not intersect each other. 

Figure \ref{figure cancel}, b) is obtained from a) by doing Hurwitz move which applies an inverse Dehn twist around the green to the red. 
We note that the Liouville structure near the black is same as the standard Liouville structure of the cotangent bundle of the black curve, as explained in Remark \ref{rmk two facts}.

On the other hand, Figure \ref{figure cancel}, b) is obtained by stabilizing c) along the green dashed curve in c). 
In order to justifying the stabilization operation, we should check that the integration of the Liouville form on the whole green dashed line is zero.
This corresponds to the condition $0= [\lambda] \in H^1(D,\partial D)$ in the definition of the stabilization. 
One can easily check this since along the green dashed curve, one can assume that the Liouville form is the standard Liouville form on the cotangent bundle of the black.

Figure \ref{figure cancel}, d) is obtained by a Hurwitz move for the red and blue curves.
This is similar to the step between a) and b).
Meanwhile, Figure \ref{figure cancel}, d) can be obtained from e), by operating a stabilization along the red dashed curve.
In order to justify the stabilization procedure, we need the same computation which we did for the step between b) and c). 

Since the local picture corresponding to the handle decomposition $D_2$ is Figure \ref{figure cancel}, e) this completes the proof of Lemma \ref{lemma handle canceling}.
\qed 

\subsection{Proof of Lemma \ref{lemma sliding}}
\label{subsect sliding}

Because of the lengthy of the paper, we prove Lemma \ref{lemma sliding} only for the case of $1$-handle sliding without twisting.
The other case could be proven by the same way. 

\begin{figure}[p]
	\centering
	\input{sliding1.txt}		
	\caption{a). It is the same as Figure \ref{figure handle moves}, b). 
		For b) -- f), the lefts are local pictures of fibers together with vanishing cycles and the right circles indicate the cyclic order of vanishing cycles. We note that the vanishing cycle corresponding to $H_0^2$ is denoted by a black dot in the right circle, but it is omitted in the fiber pictures.}
	\label{figure sliding1}
\end{figure}

\begin{figure}[p]
	\centering
	\input{sliding2.txt}		
	\caption{For each of g) -- l), the lefts are local pictures of fibers together with vanishing cycles (colored curves) and the right circles indicate the cyclic order of vanishing cycles.}
	\label{figure sliding2}
\end{figure}

We prove Lemma \ref{lemma sliding} as similar to the proof of Lemma \ref{lemma handle canceling}.
More precisely, we start from a local picture of $\pi_2$, where $D_2$ is described in Figure \ref{figure handle moves}, b) right. 
We operate a sequence of four moves for the local picture until we get a local picture corresponding to $\pi_1$. 
We note that Figure \ref{figure sliding1}, a) is the same picture as Figure \ref{figure handle moves}, b) except that it is decorated by colored curves. 
The colors explain the one-to-one relation mentioned in Remark \ref{rmk one to one}.

Figures \ref{figure sliding1} and \ref{figure sliding2} give a sequence of four moves.
We omit some details since the omitted details are the same as the proof of Lemma \ref{lemma handle canceling}.

\begin{enumerate}
	\item[b) $\Rightarrow$ c).] We take a stabilization with the dashed orange Lagrangian in b).
	\item[c) $\Rightarrow$ d).] We take a deformation.
	\item[d) $\Rightarrow$ e).] We operate a Hurwitz move changing the order of the orange and the green.
	\item[e) $\Rightarrow$ f).] We operate another Hurwitz move, exchanging the blue and the orange.
	\item[f) $\Leftarrow$ g).] We take a stabilization with the dashed orange Lagrangian in g).
	\item[g) $\Rightarrow$ h).] We take a deformation.
	\item[h) $\Leftarrow$ i).] We operate a stabilization with the dashed orange Lagrangian in i).
	\item[i) $\Rightarrow$ j).] We take a deformation.
	\item[j) $\Rightarrow$ k).] We take two Hurwitz moves, so that the orange goes front of the blue and the purple.
	\item[k) $\Leftarrow$ l).] We take a stabilization with the dashed orange Lagrangian in l).
\end{enumerate}

At the end, we can easily check that Figure \ref{figure sliding2}, l) is the local picture for $\pi_1$ corresponding to the left of Figure \ref{figure sliding1}, a).
This completes the proof.
\qed

We would like to point out that Proposition \ref{prop handle moves} is likely to hold for the case of general-dimensional $M$. 
However, the proof of Proposition \ref{prop handle moves} is based on the ``{\em case by case}'' method.
For higher-dimensional case, a similar proof will work, but there are much more cases. 
\clearpage

\part{Lefschetz fibrations on some plumbings}
\label{part Lefschetz fibrations on some plumbings}
In Part \ref{part Lefschetz fibrations on some plumbings}, we construct Lefschetz fibrations on some plumbings.
In Section \ref{section Lefschetz fibrations on plumbings}, we consider the plumbings of two cotangent bundles, and in Sections \ref{section sketch of the proof of theorem} -- \ref{section proof of plumbing tree theorem}, we consider the plumbings of multiple copies of $T^*S^n$, whose plumbing patterns are trees. 
In Section \ref{section application}, we give possible applications.
Especially, Corollary \ref{cor diffeomorphic family} gives diffeomorphic families of plumbing spaces.
The given diffeomorphic families contain some plumbing spaces with names. 
For example, the Milnor fibers of $A_{2k+1}$-type and $D_{2k+1}$-type are diffeomorphic to each other if their dimension is two times of an odd number.
In \cite{Choa-Karabas-Lee}, it is known that each member of the diffeomorphic families is symplectically different to other members of the same family. 

\section{Plumbing space of two cotangent bundles}
\label{section Lefschetz fibrations on plumbings}
In Section \ref{section Lefschetz fibrations on plumbings}, we prove the following Theorem. 
\begin{theorem}[=Theorem \ref{thm rough statement2}]
	\label{thm plumbing}
	Let $M_1$ and $M_2$ be smooth manifolds of the same dimension.
	Let $P$ be the plumbing of two cotangent bundles $T^*M_1 \# T^*M_2$ at one point.
	Then, there is an algorithm producing a Lefschetz fibration on $P$ from a pair of handle decomposition $D_1$ and $D_2$ of $M_1$ and $M_2$. 
\end{theorem}

In the first subsection, we briefly review the notion of plumbing spaces and we will describe the resulting Lefschetz fibration that Theorem \ref{thm plumbing} produces.
The proof of Theorem \ref{thm plumbing} will appear in Section \ref{subsection proof of plumbing theorem}.

\subsection{Plumbing spaces and their Lefschetz fibrations}
\label{subsection plumbing space}
\mbox{}

\noindent{\em Brief review for the plumbing procedure.}
First, we briefly review the construction of a plumbing space.
See \cite[Chapter 7.6]{Geiges08} or \cite[Section 2.3]{Abouzaid11}, for more details.

When we plumb two disk cotangent bundles $D^*M_1$ and $D^*M_2$ of the same dimension, we start the plumbing procedure by choosing a pluming point $p_i \in M_i$. 
Then, there exists a small neighborhood $U_i \subset M_i$ of $p_i$, such that $U_i \simeq \mathbb{D}^n$, where $n = \dim M_i$.
The disk cotangent bundles of $U_i$, $D^*U_i$ can be seen as $\mathbb{D}^n \times \mathbb{D}^n$.
Then, we identify $D^*U_1$ and $D^*U_2$ via the map 
\[f: D^*U_i \simeq \mathbb{D}^n \times \mathbb{D}^n \to \mathbb{D}^n \times \mathbb{D}^n \simeq D^*U_2, (x,y) \mapsto (y,-x).\]
The resulting plumbing space $P=T^*M_1 \# T^*M_2$ is given as the completion of 
\[D^*M_1 \sqcup D^*M_2 / (x,y) \sim f(x,y) \text{  for all  } (x,y) \in D^*U_1.\]

\begin{remark}
	\label{rmk sign of plumbing point}
	We note that the order of $M_1$ and $M_2$ is important in the above arguments. 
	In the plumbing space $P$, one can see $M_i$ as a submanifold of $P$. 
	Then, two submanifolds $M_1$ and $M_2$ intersect at one point, the plumbing point $p_1 = p_2$. 
	And, near the plumbing point, one can observe that $M_1$ is identified with the zero section of $D^*U_1$, and $M_2$ is identified with the cotangent fiber $D^*_{p_1}U_1$. 
	Let us choose a local Darboux chart $(x_1, \dots, x_n, y_1, \dots, y_n)$ near the plumbing point, satisfying that
	\begin{itemize}
		\item the symplectic form is given as $\sum_{i=1}^n d y_i \wedge d x_i$, and
		\item the base of $D^*U_1$ is coordinated by $(x_1, \dots, x_n)$ and the fiber of $D^*U_1$ is coordinated by $(y_1, \dots, y_n)$.
	\end{itemize}
	Then, as the {\em algebraic} intersection point of two submanifolds $M_1$ and $M_2$, the plumbing point has the sign $(-1)^{\tfrac{1}{2}n(n-1)}$. 
	If we change the order of $M_1$ and $M_2$ in the construction of $P$, one can observe that the sign of the intersection point can change based on the choice of $n$. 
\end{remark}
\vskip0.2in

\noindent{\em Construction of an abstract Lefschetz fibration.}
Theorem \ref{thm plumbing} produces a Lefschetz fibration for a plumbing space $P = T^*M_1 \# T^*M_2$. 
Before proving Theorem \ref{thm plumbing}, we describe the resulting (abstract) Lefschetz fibration of $P$. 

We set the notation first. 
Let 
\begin{gather*}
	D_1 = \{a_0, \dots, a_{m_1}\}, D_2 = \{b_0, \dots, b_{m_2}\}
\end{gather*}
be given handle decompositions of $M_1$ and $M_2$ respectively. 
By Definition \ref{def handle decomposition}, we have $N_i \leq m_i$ such that $a_j$ (resp.\ $b_k$) is a handle of index $<n$ if $j \leq N_1$ (resp.\ $k \leq N_2$). 

Let 
\begin{gather}
	\label{eqn notation}
	(F_1; X_{m_1}, \dots, X_0) \text{  and  } (F_2; Y_{m_2}, \dots, Y_0)
\end{gather} 
denote the abstract Lefschetz fibrations which are obtained by applying Theorem \ref{thm main} to $D_1$ and $D_2$.
We recall that $F_i$ is determined by handles of index $<n$ in $D_i$, i.e., 
\[\{a_0, \dots, a_{N_1}\}, \{b_0, \dots, b_{N_2}\}.\]
More precisely, the above handles of index $<n$ give a Weinstein handle decomposition of the regular fiber $F_i$. 
According to the Weinstein handle decomposition, one can construct $F_i$ by attaching $(2n-2)$-dimensional Weinstein handles to the disk cotangent bundle of $S^{n-1}$.
See Remark \ref{rmk M_0}.

We recall that $a_0$ and $b_0$ are $0$-handles. 
Thus, $\partial a_0, \partial b_0 \simeq S^{n-1}$. 
Let $S_+, S_-$ be the upper and lower hemisphere of $S^{n-1}$. 
Without loss of generality, one can assume that 
\[\partial_R a_i \cap \partial a_0 \subset S_+ \subset S^{n-1} \simeq \partial a_0, \partial_R b_j \cap \partial b_0 \subset S_- \subset S^{n-1} \simeq \partial b_0,\]
for all $i \in [1,N_1]$ and for all $j \in [1,N_2]$. 

Under the assumption, $F_1$ is obtained by attaching $(2n-2)$-dimensional Weinstein handles to 
\[\partial(D^*_{S_+}S^{n-1}) = \left\{(x,y) \in D^*S^{n-1} | x \in S_+, y \in D^*_xS^{n-1}\right\},\]
where $D^*$ means the disk cotangent bundle.
Similarly, $F_2$ is obtained by attaching $(2n-2)$-dimensional Weinstein handles to 
\[\partial(D^*_{S_-}S^{n-1}) = \left\{(x,y) \in D^*S^{n-1} | x \in S_-, y \in D^*_xS^{n-1}\right\}.\]

Now, we construct the regular fiber $F$ for $P$, by attaching $(2n-2)$-dimensional Weinstein handles to $D^*S^{n-1}$ as follows: 
We attach Weinstein handles to $D^*_{S_+}S^{n-1}$ (resp.\ $D^*_{S_-}S^{n-1}$) in the same way as we constructed $F_1$ (resp.\ $F_2$).
Then, one could understand $F_1$ and $F_2$ as subsets of $F$ so that 
\begin{gather*}
	F_ 1 \cup F_2 = F, \\
	F_1 \cap F_2 = D^*S^{n-1}.
\end{gather*} 	

Figure \ref{figure the fiber} is an example. 
The example case is the plumbing of two $T^*\mathbb{T}^2$ where $\mathbb{T}^2$ is the $2$-dimensional torus. 
The handle decompositions $D_1$ and $D_2$ are the same as the handle decomposition described in Figure \ref{figure decomposition of a torus}, a).	 	
Then, the fiber $F$ is given by attaching eight $1$-handles to $D^*S^1$.

\begin{figure}[h]
	\centering
	\input{the_fiber.txt}		
	\caption{The fiber after attaching eight $1$-handles is given. In the picture, the top and bottom line segments are identified. The labels mean that the segments having the same labels should be identified to each other, and the arrow indicates the way of identification. The red and blue curves are Lagrangians in the fiber, which are obtained by modifying the zero sections of $D^*S^1$. According to the proof of Theorem \ref{thm main}, the modified Lagrangians explain how to attach $1$-handles.}
	\label{figure the fiber}
\end{figure}

One can check that an exact Lagrangian in $F_i$ is an exact Lagrangian in $F$. 
Then,  
\begin{gather}
	\label{eqn Lefschetz fibration of plumbing}
	\pi := (F; X_{m_1}, \dots, X_1, Y_{m_2}, \dots, Y_1, X_0 = Y_0)
\end{gather}
is a well-defined abstract Lefschetz fibration. 
We note that the vanishing cycle $X_0 = Y_0$ is the zero section of $D^*S^{n-1} \subset F$. 

We would like to point out that because of the definition of Plumbing space, one can see $T^*M_i$ as a submanifold of their plumbing space $P$. 
In the total space of the abstract Lefschetz fibration given in \eqref{eqn Lefschetz fibration of plumbing}, one can find the submanifold $T^*M_i$ as follows: 
Let $W_i$ be the subset of $W$ such that the restriction of $\pi$ on $W_i$ satisfies 
\begin{itemize}
	\item the regular fiber of $\pi|_{W_i}$ is $F_i \subset F$, and 
	\item the target of the restriction $\pi|_{W_1}$ (resp.\ $\pi|_{W_2}$) is the interior of the red (resp.\ blue) circle given in Figure \ref{figure base plumbing}.
\end{itemize}
Then, $W_i$ can be seen as the total space of the Lefschetz fibrations given in \eqref{eqn notation}, i.e., $W_i$ is equivalent to $T^*M_i$. 

\begin{figure}[h]
	\centering
\begingroup%
  \makeatletter%
  \providecommand\color[2][]{%
    \errmessage{(Inkscape) Color is used for the text in Inkscape, but the package 'color.sty' is not loaded}%
    \renewcommand\color[2][]{}%
  }%
  \providecommand\transparent[1]{%
    \errmessage{(Inkscape) Transparency is used (non-zero) for the text in Inkscape, but the package 'transparent.sty' is not loaded}%
    \renewcommand\transparent[1]{}%
  }%
  \providecommand\rotatebox[2]{#2}%
  \newcommand*\fsize{\dimexpr\f@size pt\relax}%
  \newcommand*\lineheight[1]{\fontsize{\fsize}{#1\fsize}\selectfont}%
  \ifx\svgwidth\undefined%
    \setlength{\unitlength}{226.77165354bp}%
    \ifx\svgscale\undefined%
      \relax%
    \else%
      \setlength{\unitlength}{\unitlength * \real{\svgscale}}%
    \fi%
  \else%
    \setlength{\unitlength}{\svgwidth}%
  \fi%
  \global\let\svgwidth\undefined%
  \global\let\svgscale\undefined%
  \makeatother%
  \begin{picture}(1,0.9875)%
    \lineheight{1}%
    \setlength\tabcolsep{0pt}%
    \put(0,0){\includegraphics[width=\unitlength,page=1]{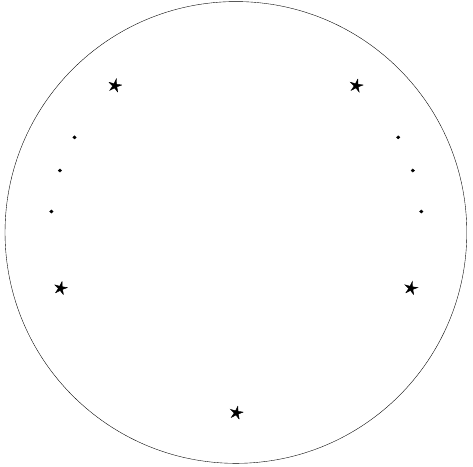}}%
    \put(0.42484981,0.07194898){\makebox(0,0)[lt]{\lineheight{1.25}\smash{\begin{tabular}[t]{l}$x_0 = y_0$\end{tabular}}}}%
    \put(0.83245307,0.33443812){\makebox(0,0)[lt]{\lineheight{1.25}\smash{\begin{tabular}[t]{l}$x_{m_1}$\end{tabular}}}}%
    \put(0.19456418,0.76137366){\makebox(0,0)[lt]{\lineheight{1.25}\smash{\begin{tabular}[t]{l}$y_{m_2}$\end{tabular}}}}%
    \put(0.71954241,0.76560735){\makebox(0,0)[lt]{\lineheight{1.25}\smash{\begin{tabular}[t]{l}$x_1$\end{tabular}}}}%
    \put(0.10352561,0.3325006){\makebox(0,0)[lt]{\lineheight{1.25}\smash{\begin{tabular}[t]{l}$y_1$\end{tabular}}}}%
    \put(0,0){\includegraphics[width=\unitlength,page=2]{base_plumbing.pdf}}%
  \end{picture}%
\endgroup%
		
	\caption{The star marks are singular values. The vanishing cycles corresponding to singular values $x_i$ and $y_j$ are $X_i$ and $Y_j$ respectively. The red and blue circles are boundaries of the targets of $\pi|_{W_1}$ and $\pi|_{W_2}$.}
	\label{figure base plumbing}
\end{figure}

\subsection{Proof of Theorem \ref{thm plumbing}}
\label{subsection proof of plumbing theorem}
We prove Theorem \ref{thm plumbing} in an inductive manner that is similar to the proof of Theorem \ref{thm main}.
Thus, in the proof below, we emphasize the difference between proofs of Theorem \ref{thm main} and Theorem \ref{thm plumbing}, and skip some details if the details are the same as the proof of Theorem \ref{thm main}.
\vskip0.2in 

\noindent{\em Weinstein handle decomposition of the plumbing space.}
In the proof of Theorem \ref{thm main}, we constructed a Weinstein manifold $W_i = F_i \times \mathbb{C}$ by attaching Weinstein handles inductively. 
And the last step of the induction is to construct a Weinstein manifold equivalent to the starting cotangent bundle.
Similarly, we construct the plumbing space $P = T^*M_1 \# T^*M_2$ by attaching Weinstein handles inductively. 
In order to describe the construction, let us set the following notation: 
As we did in the preceding subsection, we let 
\[D_1 = \{a_0, \dots, a_{m_1}\}, D_2 = \{b_0, \dots, b_{m_2}\}\]
denote the given handle decomposition of $M_1$ and $M_2$. 
Also, we have $N_i \leq m_i$ such that $a_j$ (resp.\ $b_k$) is a handle of index $<n$ if $j \leq N_1$ (resp.\ $k \leq N_2$). 

Then, we consider another handle decomposition $\tilde{D}_1$ (resp.\ $\tilde{D}_2$) obtained by dividing $a_1, \dots, a_{N_1}$ (resp.\ $b_1, \dots, b_{N_2}$) into $a_i^{ori}, a_i^{n-1}, a_i^n$ (resp.\ $b_1^{ori}, b_i^{n-1}, b_i^n$). 
We would like to emphasize that {\em we do not divide the index zero handles $a_0$ and $b_0$}, differently from the proof of Theorem \ref{thm main}. 

Applying Theorem \ref{lem handle decomposition of cotangent bundle}, we have a Weinstein handle decomposition of $T^*M_i$, 
\begin{gather*}
	W_{\tilde{D}_1} = \{A_0, A_1^{ori}, A_1^{n-1}, \dots, A_{N_1}^{ori}, A_{N_1}^{n-1}, A_1^n, \dots, A_{N_1}^n, A_{N_1+1}, \dots, A_{m_1}\},\\
	W_{\tilde{D}_2} = \{B_0, B_1^{ori}, B_1^{n-1}, \dots, B_{N_2}^{ori}, B_{N_2}^{n-1}, B_1^n, \dots, B_{N_2}^n, B_{N_2+1}, \dots, B_{m_2}\}.
\end{gather*}
We note that $A_0$ and $B_0$ are the zero handles, so we can simply see that they are equivalent to $\mathbb{D}^{2n}$ embedded in $\mathbb{R}^{2n}$.
Then, we can identify $A_0$ (resp.\ $B_0$) with $\mathbb{D}^{2n}$ so that $a_0 \subset A_0$ (resp.\ $b_0 \subset B_0$) is identified with $\mathbb{D}^{2n} \cap \mathbb{R}^n \times \{(0, \dots, 0)\}$ (resp.\ $\{(0, \dots, 0)\}\times \mathbb{R}^n$).

When one attaches the Weinstein handles 
\begin{itemize}
	\item $A_1^{ori}, A_1^{n-1}, \dots, A_{N_1}^{ori}, A_{N_1}^{n-1}, A_1^n, \dots, A_{N_1}^n, A_{N_1+1}, \dots, A_{m_1}$ along $\Lambda_1 := \partial a_0 \subset \partial \mathbb{D}^{2n}$, and
	\item $B_1^{ori}, B_1^{n-1}, \dots, B_{N_2}^{ori}, B_{N_2}^{n-1}, B_1^n, \dots, B_{N_2}^n, B_{N_2+1}, \dots, B_{m_2}$ along $\Lambda_2 := \partial b_0 \subset \partial \mathbb{D}^{2n}$,
\end{itemize}
one can recover the plumbing space $P$. 
\vskip0.2in

\noindent{\em The base step.}
{\em The base step is one of the biggest difference between two inductive construction of Lefschetz fibrations in the proof of Theorems \ref{thm main} and \ref{thm plumbing}. }
The base step of our inductive construction is a Lefschetz fibration of $\mathbb{D}^{2n} \simeq A_0 \simeq B_0$. 
There exists a well-known (abstract) Lefschetz fibration for $\mathbb{D}^{2n}$ given as 
\[W_0 := \left(F_0 = D^*S^{n-1}; \text{  the zero section  } S^{n-1}\right).\]
For convenience, we call the vanishing cycle as $X_0$. 

Without loss of generality, let us identify the above abstract Lefschetz fibration with a specific Lefschetz fibration $\pi_0: \mathbb{D}^{2n} \to \mathbb{D}^2 \subset \mathbb{C}$ such that whose singular value is located at the center of the base.
See Figure \ref{figure the simple LF}.
Then, the union of vanishing cycles along the red (resp.\ blue) curve in Figure \ref{figure the simple LF} becomes a Lagrangian disk centered at the singular point. 
Moreover, those two Lagrangian disks intersect transversally.
We identify the disk along the red (resp.\ blue) curve with $a_0 \subset A_0 \simeq \mathbb{D}^{2n}$ (resp.\ $b_0 \subset B_0 \simeq \mathbb{D}^{2n}$), and we will attach the other Weinstein handles using the zero sections of $\pi_0^{-1}(\pm 1) \simeq D^*S^{n-1}$. 
For the later use, let $\Lambda_1$ denote the Legendrian sphere in $\pi_0^{-1}(1) \subset \partial_\infty W_0$ corresponding to $\partial a_0$, and $\Lambda_2$ denote the Legendrian sphere in $\pi_0^{-1}(-1) \subset \partial_\infty W_0$ corresponding to $\partial b_0$.

\begin{figure}[h]
	\centering
\begingroup%
  \makeatletter%
  \providecommand\color[2][]{%
    \errmessage{(Inkscape) Color is used for the text in Inkscape, but the package 'color.sty' is not loaded}%
    \renewcommand\color[2][]{}%
  }%
  \providecommand\transparent[1]{%
    \errmessage{(Inkscape) Transparency is used (non-zero) for the text in Inkscape, but the package 'transparent.sty' is not loaded}%
    \renewcommand\transparent[1]{}%
  }%
  \providecommand\rotatebox[2]{#2}%
  \newcommand*\fsize{\dimexpr\f@size pt\relax}%
  \newcommand*\lineheight[1]{\fontsize{\fsize}{#1\fsize}\selectfont}%
  \ifx\svgwidth\undefined%
    \setlength{\unitlength}{255.11811024bp}%
    \ifx\svgscale\undefined%
      \relax%
    \else%
      \setlength{\unitlength}{\unitlength * \real{\svgscale}}%
    \fi%
  \else%
    \setlength{\unitlength}{\svgwidth}%
  \fi%
  \global\let\svgwidth\undefined%
  \global\let\svgscale\undefined%
  \makeatother%
  \begin{picture}(1,0.87777778)%
    \lineheight{1}%
    \setlength\tabcolsep{0pt}%
    \put(0,0){\includegraphics[width=\unitlength,page=1]{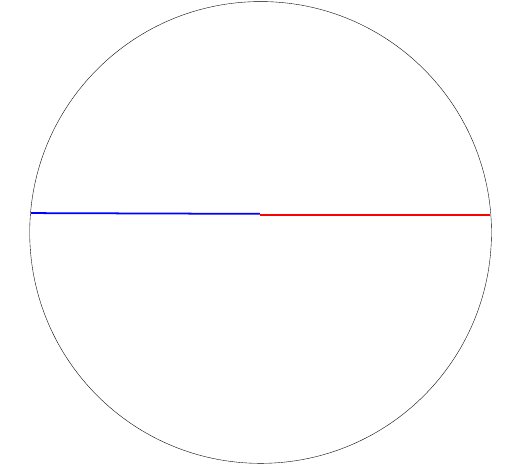}}%
    \put(0.06830221,0.49172177){\color[rgb]{0,0,0}\makebox(0,0)[lt]{\lineheight{1.25}\smash{\begin{tabular}[t]{l}$-1 \in \mathbb{C}$\end{tabular}}}}%
    \put(0.92616062,0.49171931){\color[rgb]{0,0,0}\makebox(0,0)[lt]{\lineheight{1.25}\smash{\begin{tabular}[t]{l}$1 \in \mathbb{C}$\end{tabular}}}}%
    \put(0,0){\includegraphics[width=\unitlength,page=2]{the_simple_LF.pdf}}%
  \end{picture}%
\endgroup%
		
	\caption{The star mark means the singular value at the center of the base, and the red and blue curves have an end point at the singular value, and have another endpoint at $\pm 1$, respectively.}
	\label{figure the simple LF}
\end{figure}
\vskip0.2in

\noindent{\em The differences on the vertical boundary.}
Since the base step is not a product Lefschetz fibration, its vertical boundary should be a mapping torus of the monodromy map, i.e., a generalized Dehn twist along $X_0$.
For more details, see \cite{Bahar20}.
{\em It is also one of the biggest difference from the proof of Theorem \ref{thm main}}. 
For the later convenience, let $\tau$ denote the Dehn twist along $X_0$.

\begin{remark}
	\label{rmk monodromy map}
	The difference will affect when we take a Legendrian isotopy later. 
	We recall that in the proof of Theorem \ref{thm main}, we constructed Legendrian isotopies $\gamma_1^s$ and $\gamma_2^s$ in Section \ref{subsubsection crossing the base}. 
	We will take similar isotopies in the proof of Theorem \ref{thm plumbing}, but, the construction of $\gamma_1^s$ and $\gamma_2^s$ is a bit different from the construction in Section \ref{subsubsection crossing the base}, if the isotopies touch the dashed line in Figure \ref{figure the simple LF}. 
	More precisely, when the isotopies pass the dashed line from the left-handed side to the right-handed side, we apply $\tau$ for the fiber coordinate. 
\end{remark}
\vskip0.2in 

\noindent{\em Attaching subcritical handles for $M_1$.}
The next step is to attach subcritical handles for $M_1$, i.e., $A_1^{ori}, \dots, A_{N_1}^{n-1}$.
For that we first choose a small $\theta_0 >0$ such that $2(N_1+N_2+1) \theta_0 < \pi$. 
We point out that the condition on $\theta_0$ is {\em different from that in Section \ref{subsection induction hypotheses}}, see (IH 2). 

To attach the subcritical handles, we first take a Legendrian isotopy of $\Lambda_1$ so that the resulting Legendrian is located at $\pi_0^{-1}\left(\cos(-\theta_0)+\sqrt{-1}\sin(-\theta_0)\right)$.
Then, we take Legendrian isotopies of the new Legendrian inductively, as we did in the proof of Theorem \ref{thm main}, especially in Sections \ref{subsubsection push to the corner}--\ref{subsubsection attaching subcritical handles}.
In other words, we push the Legendrian to corners and then cross the Legendrian over the other side of the base. 

The contents of Section \ref{subsubsection push to the corner} works without any modification and it pushes the Legendrian to the corner. 
We also note that the construction of Legendrian isotopies in Section \ref{subsubsection crossing the base} gives Legendrian isotopies crossing the base. 
More precisely, the constructed Legendrian isotopies, i.e., $\gamma_1^s$ and $\gamma_2^s$, meet the dashed line only at the central singular fiber in Figure \ref{figure the simple LF}.
See Figure \ref{figure goodcase}.
However, at the singular fiber, the image of isotopies are lying on the horizontal boundary, which is still given as the product of the boundary of the fiber and the base disk, because of the {\em regularity along $\partial^h W$ condition} in Definition \ref{def Lefschetz fibration}.
Thus, the construction in Section \ref{subsubsection crossing the base} still gives Legendrian isotopies.
\begin{figure}[h]
	\centering
	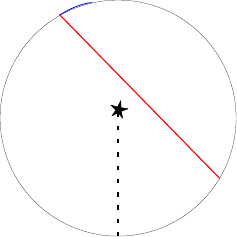		
	\caption{The red and blue curves are the projection-to-the-base images of $\gamma_1^s$ and $\gamma_2^s$, constructed by the construction in Section \ref{subsubsection crossing the base}.}
	\label{figure goodcase}
\end{figure}

Because we have Legendrian isotopies pushing $\Lambda_1$ to the corner and crossing the base, we can attach subcritical handles $A_i^{ori}$ and $A_i^{n-1}$ in each inductive step.
Then, after $N_1$-many inductive steps, we have a fiber $F_1$ and Legendrian spheres corresponding to $\partial a_i^n$ for $i =1, \dots, N_1$ and $\partial a_j$ for $j = N_1+1, \dots, m_1$.
Moreover, the Legendrian spheres corresponding to $\partial a_i^n$ for $i =1, \dots, N_1$ (resp.\ $\partial a_j$ for $j = N_1+1, \dots, m_1$) should lie in the region 
\[\left[\pi-(N_1+1)\theta_0,\pi-\theta_0\right], \text{  (resp.\ $\left[-(N_1+1)\theta_0,-N_1\theta_0\right]$)} \subset \mathbb{R}/\mathbb{Z} = S^1 = \partial \mathbb{D}^2. \]

The next step is to attach subcritical handles for $M_2$. 
Before that, we would like to take simple Legendrian isotopies of the obtained Legendrian spheres, so that after the isotopy, Legendrians corresponding to $\partial a_i^n$ for $i =1, \dots, N_1$ (resp.\ $\partial a_j$ for $j = N_1+1, \dots, m_1$) are lying in the region 
\[\left[\pi-(N_1+N_2 + 1) \theta_0, \pi-(N_2+1) \theta_0 \right] \text{(resp.\ $\left[\theta_0, (N_1+1)\theta_0\right]$)}.\]
We note that since $0 < \theta_0 < \tfrac{\pi}{2(N_1+N_2+1)}$, 
\[(N_1+1) \theta_0< \pi- (N_1+N_2+1) \theta_0.\]
Thus, after the isotopy, Legendrians corresponding to $\partial a_i^n$ for $i =1, \dots, N_1$ and Legendrians corresponding to $\partial a_j$ for $j = N_1+1, \dots, m_1$ do not intersect. 

\begin{remark}
	\label{rmk vanishing cylce 1}
	The resulting Legendrian spheres, corresponding to $\partial a_i^n$ and $\partial a_j$ for proper $i$ and $j$, will give the vanishing cycles of the resulting Lefschetz fibration. 
	Because their construction is the same as the case of $T^*M_1$, the corresponding vanishing cycles are $X_1, \dots, X_{m_1} \subset F_1$. 
\end{remark}
\vskip0.2in

\noindent{\em Attaching subcritical handles for $M_2$.}
Similar to the above step and Sections \ref{subsubsection push to the corner}--\ref{subsubsection attaching subcritical handles}, in order to attach the subcritical handles of $P$ from $M_2$, we need to push $\Lambda_2$ to the corner, and then need to cross the Legendrian over the base. 
We note that pushing-to-the-corner step is the same as before, but {\em the crossing-the-base step should be different from what we did before}, because of the following two reasons: 
The first reason is that the Legendrian in our interest is lying at the fiber at $-1$, which is different from before. 
The second, and more important reason, is that, in the plumbing case, we do not have a product Lefschetz fibration in each inductive step. 

As considering the two reasons, we construct the following Legendrian isotopies $\gamma_3^s$ and $\gamma_4^s$ with the notations used in Section \ref{section the proof of Theorem main}:
\begin{gather}
	\label{eqn gamma_3,4}
	\begin{gathered}
		\gamma_3^s : \partial \mathbb{D}^{n-k}_{2\epsilon} \times S^{k-1} \times [0,1] \to \partial (F_i \times \mathbb{D}^2),\\
		(p, q, t) \mapsto \big(\boldsymbol{\Psi_i^{t,s}}(p,q), (1-t) \cos (-c g(p,q)+\theta_*+\boldsymbol{\pi}) + t \cos(-c g(p,q)+\theta_*+s+\boldsymbol{\pi}), \\
		(1-t) \sin (-c g(p, q) + \theta_*+\boldsymbol{\pi}) + t \sin(-c g(p,q)+\theta_*+s+\boldsymbol{\pi})\big)\\
		\gamma_4^s : \mathbb{D}^{n-k}_{2\epsilon} \times S^{k-1} \to \partial (F_i \times \mathbb{D}^2), \\
		(p,q) \mapsto \Big(\boldsymbol{\Psi_i^{1,s}}(p,q), \cos\big(-c g(p,q) + \theta_* + s + h_s(p,q)+\boldsymbol{\pi}\big), \sin\big(-c g(p,q) + \theta_* +s + 	h_s(p,q)+\boldsymbol\pi\big) \Big),
	\end{gathered}
\end{gather}
where 
\[\Psi_i^{t,s}(p,q) := \begin{cases}
	\left(\Phi_i^{-\tfrac{1}{c}t\sin s}\circ j_{i+1}\right)(p,q), \text{  if  } (1-t)\cos(-cg(p,q) + \pi+\theta_*)+t \cos(-cg(p,q) + \pi+\theta_*+s) <0, \\
	\left(\tau \circ \Phi_i^{-\tfrac{1}{c}t\sin s}\circ j_{i+1}\right)(p,q), \text{  if  } (1-t)\cos(-cg(p,q) + \pi+\theta_*)+t \cos(-cg(p,q) + \pi+\theta_*+s) \geq 0.
\end{cases}\]
The above $\gamma_3^s$ and $\gamma_4^s$ in \eqref{eqn gamma_3,4} are different from $\gamma_1^s$ and $\gamma_2^s$ in \eqref{eqn gamma_1} and \eqref{eqn gamma_2} by the parts written in bold font. 
We added $\pi$ because of the first reason, and we replace $\Phi_i^{-\tfrac{1}{c}t\sin s}$ with $\Psi_i^{t,s}$ because of the second reason. 

We note that $\tau$ is a generalized Dehn twist along $X_0$, thus $\tau$ is compactly supported symplectomorphism.
Thus, $\gamma_1^s$ and $\gamma_3^s$ are the same except the bold-$\pi$-parts since on the domain of $\gamma_1^s$ and $\gamma_3^s$, $(p,q)$ lies on the boundary of the fiber. 
And, since the vertical boundary is not a product space, $\text{the fiber} \times S^1$, but the mapping torus of $\tau$, $\gamma_4^s$ becomes a Legendrian isotopy. 
See the left picture of Figure \ref{figure badcase}, which describes the projection images of $\gamma_3^s$ and $\gamma_4^s$ on the base.

Through the Legendrian isotopies $\gamma_3^s$ and $\gamma_4^s$, we can deform the Legendrian in each inductive step, and we can attach the subcritical handles $B_i^{ori}$ and $B_i^{n-1}$.
We note again that $\tau$ is a generalized Dehn twist and is compactly supported, thus on the horizontal boundary, the Legendrian isotopy is the same as what we took in Section \ref{section the proof of Theorem main}. 
Thus, the resulting fiber is the same as $F$ defined in Section \ref{subsection plumbing space}, which is obtained by attaching $\{B_i^{ori}, B_i^{n-1}|i =1, \dots, N_2\}$ to $F_1$. 

\begin{figure}[h]
	\centering
	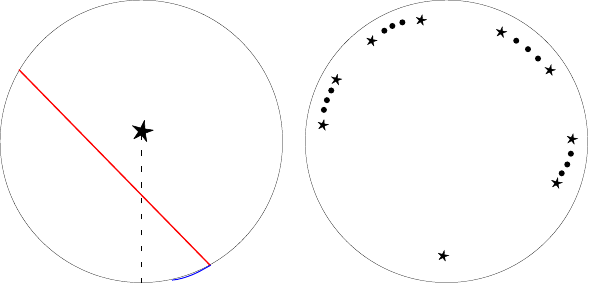		
	\caption{The left describes the projection images of the Legendrian isotopies $\gamma_3^s$ (red) and $\gamma_4^s$ (blue). The right describes the base of resulting Lefschetz fibration with singular values. The star marked points are singular values. Each singular value is labeled by the corresponding vanishing cycle.}
	\label{figure badcase}
\end{figure}

We also have the Legendrian spheres along which the critical handles $B_i^n$ for $i=1, \dots, N_2$ and $B_j$ for $j=N_2+1, \dots, m_2$ will be attached.
These Legendrian spheres satisfy the following: 
\begin{itemize}
	\item By the Legendrian isotopy we took above, the Legendrian spheres are projected down to two regions of the boundary of the base.
	The first region is $[-N_2\theta_0, 0]$ where the attaching spheres of $B_i^n$ for $i =1, \dots, N_2$ are lying, and the second region is $[\pi-N_2\theta_0, \pi]$ where the attaching spheres of $B_j$ for $j= N_2 +1, \dots, m_2$ are lying. 
	\item We note that in Section \ref{subsection plumbing space}, we described the Lefschetz fibration of $T^*M_2$, which is written as 
	\[(F_2; Y_{m_2}, \dots, Y_1, Y_0).\]
	Among the vanishing cycles, if $i=1, \dots, N_2$, $Y_i$ corresponds to the attaching sphere of $B_i^n$. 
	From the above construction, if one finds the vanishing cycle in the fiber $F$, corresponding to the attaching sphere of $B_i^n$, it must be $\tau(Y_i)$.
	\item Similarly, one can observe that for $j= N_2+1, \dots, m_2$, the vanishing cycle corresponding to the attaching sphere of $B_j$ is $Y_j$. 
\end{itemize}
\vskip0.2in

\noindent{\em Attaching critical handles.}
The final step of the inductive construction is to attach the critical handles along the Legendrian spheres obtained by above procedure. 
Then, from the above arguments, the resulting Lefschetz fibration can be written as the following abstract Lefschetz fibration:
\[\left(F; X_{m_1}, \dots, X_1, Y_{m_2}, \dots, Y_{N_2+1}, X_0 = Y_0, \tau(Y_{N_2}), \dots, \tau(Y_1)\right).\]
See the right picture of Figure \ref{figure badcase}.

Moreover, by taking Hurwitz moves, one can obtain an abstract Lefschetz fibration 
\[(F; X_{m_1}, \dots, X_1, Y_{m_2}, \dots, Y_1, X_0 = Y_0),\]
which is already given in Equation \eqref{eqn Lefschetz fibration of plumbing}.
\qed

\section{Sketch of the proof of Theorem \ref{thm rough statement3}}
\label{section sketch of the proof of theorem}

In Section \ref{section Lefschetz fibrations on plumbings}, we described a way of constructing Lefschetz fibrations defined on plumbings of two cotangent bundles. 
The idea in Section \ref{section Lefschetz fibrations on plumbings} is to use a simple Weinstein handle decomposition of a plumbing space, which one can construct by identifying the index zero handles of two cotangent bundles. 
This idea can be generalized to more plumbing spaces easily if the plumbings have a tree as their plumbing patters.
The modified idea is to identify the unique zero handle of a cotangent bundle and a {\em critical} handle of another cotangent bundle. 

To be more precise, let us recall that a critical handle has the unique zero of the inherent Liouville vector flow. 
With respect to the negative Liouville flow, the stable and unstable manifolds of the unique zeros are both disk $\mathbb{D}^n$. 
The boundary of the stable manifold is the attaching sphere of the critical handle. 

On the other hand, let us consider a Weinstein handle decomposition of a cotangent bundle is obtained by applying Lemma \ref{lem handle decomposition of cotangent bundle}.
Then, the Weinstein handle decomposition has a unique zero handle, and its core Lagrangian is a disk $\mathbb{D}^n$. 
One can easily check that if a handle in the Weinstein handle decomposition is attached to the unique zero handle, then the handle is attached along the {\em boundary of the core Lagrangian} of the unique zero handle. 
In other words, there exists a Legendrian sphere on the unique zero handle such that other handles are attached along only the Legendrian sphere. 

Now, when one plumbs two cotangent bundles $T^*M_1$ and $T^*M_2$, one can choose a critical handle of $T^*M_1$. 
And, one can attach the Weinstein handles of $T^*M_2$, having positive indexes, along the {\em boundary of the unstable manifold of the unique zero}.
Then, the procedure gives us a Weinstein handle decomposition of the plumbing space.
We note that the resulting Weinstein handle decomposition is not the sense of Definition \ref{def Weinstein handle decomposition}, since it allows a handle to be attached on a critical handle. 

Comparing this idea to the idea given in Section \ref{section Lefschetz fibrations on plumbings}, i.e., the idea identifying two index $0$ handles, one can observe that the new idea can be applied to the plumbing of three or more cotangent bundles.
Since each cotangent bundle has a unique zero handle, the idea in Section \ref{section Lefschetz fibrations on plumbings} is applied to plumbings of two cotangent bundles.
However, since a Weinstein manifold has as many critical handles as we need by adding canceling pairs, we can apply the new idea to the plumbings of three or more cotangent bundles if the plumbings has a tree as their plumbing patterns.

The idea gives us a Weinstein handle decomposition of a plumbing space, and the proof of Theorem \ref{thm plumbing} or its modification can construct a Lefschetz fibration from the Weinstein handle decomposition.
However, the construction seems quite complicated when we consider the plumbings of multiple cotangent bundles, and would be challenging to write in a formal manner. 
Thus, in the rest of the current paper, we discuss plumbings of a simple type.
In other words, we will prove Theorem \ref{thm rough statement3}.

We will prove Theorem \ref{thm rough statement3} in Section \ref{section proof of plumbing tree theorem}, but we give a sketch of the proof in the present section. 
\begin{proof}[Sketch of the proof of Theorem \ref{thm rough statement3}]
	Let $T$ be a tree, and let $P$ denote the plumbing of $T^*S^n$ whose plumbing pattern is $T$. 
	First, we will give an order on the set of vertices of $T$ in Section \ref{section an order on a tree}.
	Then, we will prove Theorem \ref{thm rough statement3} by an induction on the ordered set of vertices.
	
	Let $\{v_1, \dots, v_m\}$ be the ordered set of vertices of $T$. 
	Then, we consider a subtree $T^{(k)}$ of $T$ such that the vertices of $T^{(k)}$ is $\{v_1, \dots, v_k\}$. 
	
	Let $P_k$ denote the plumbing space of $k$ copies of $T^*S^n$, along the plumbing pattern $T^{(k)}$. 
	The induction is to construct a Lefschetz fibration of $P_{k+1}$ from that of $P_k$. 
	This inductive step can be proven by a similar argument with Section \ref{section Lefschetz fibrations on plumbings}.
	More detailed proof will be given in Section \ref{section proof of plumbing tree theorem}.
	
	Then, the induction on $k$ will completes the proof. 
\end{proof}

\section{An order on a tree}
\label{section an order on a tree} 
Let $T$ be a tree, i.e., a graph without a cycle. 
For convenience, we define the following notation. 

\begin{definition}
	\label{def V(T) and E(T)}
	\mbox{}
	\begin{enumerate}
		\item For a tree $T$, let {\em $V(T)$} and {\em $E(T)$} denote the set of vertices and the set of edges of $T$, respectively. 
		\item For any $v \in V(T)$, let {\em $E_v(T)$} denote the set of edge $e$ such that $v$ is an end point of $e$.
		\item A vertex $v \in V(T)$ is a {\em boundary vertex of $T$} if there is only one edge $e \in E(T)$ such that $v$ is an end point of $e$, i.e., $|E_v(T)|=1$. 
	\end{enumerate}	
\end{definition}

The goal of this section is to set notations which we will use in Sections \ref{section the algorithm for the plumbings along trees} and \ref{section proof of plumbing tree theorem} and to fix an order of $V(T)$ satisfying the following: 
Let 
\[V(T):= \{ v_0, v_1, \dots, v_m\}\]
be an ordered set.
Then, there is a sequence of subtree $T^{(k)}$ such that
\begin{itemize}
	\item $V(T^{(k)}) = \{v_0, \dots, v_k\}$, and
	\item $T^{(1)} \subset T^{(2)} \subset \dots \subset T^{(m)} = T$.
\end{itemize}
A such order will be defined in Definition \ref{def order of vertices}, and to state it, we would like to define the notion of {\em embedded tree} (Definition \ref{def embedded tree}) and {\em rooted tree} (Definition \ref{def rooted tree}).

We point out that every tree is planer, i.e., there exists an embedding $f: T \to \mathbb{R}^2$. 

\begin{definition}
	\label{def embedded tree}
	An {\em embedded tree} is a pair 
	\[(T, f : T \to \mathbb{R}^2),\]
	such that $f$ is an embedding of a tree $T$.  
\end{definition}
We simply say that $T$ is an embedded tree without mentioning a specific embedding $f$. 

When one has an embedded tree $T$, there exists a natural cyclic order on the set $E_v(T)$ for any $v \in V(T)$.
See Figure \ref{figure cyclic order}. 

Similarly, one can define a cyclic order on the set of all boundary vertices of $T$.
For more detailed explanation, we choose a closed subset $D \subset \mathbb{R}^2$ such that 
\begin{itemize}
	\item $D$ is homeomorphic to a topological disk $\mathbb{D}^2$, 
	\item $T \subset D$, and
	\item $\partial D \cap T$ is the set of all boundary vertices. 
\end{itemize}
Then, the orientation on $\partial D$ induces a cyclic order on the set of boundary vertices. 
Figure \ref{figure cyclic order} gives an example of the cyclic orders on $E_v(T)$ and the set of boundary vertices. 

 \begin{figure}[h]
	\centering
\begingroup%
  \makeatletter%
  \providecommand\color[2][]{%
    \errmessage{(Inkscape) Color is used for the text in Inkscape, but the package 'color.sty' is not loaded}%
    \renewcommand\color[2][]{}%
  }%
  \providecommand\transparent[1]{%
    \errmessage{(Inkscape) Transparency is used (non-zero) for the text in Inkscape, but the package 'transparent.sty' is not loaded}%
    \renewcommand\transparent[1]{}%
  }%
  \providecommand\rotatebox[2]{#2}%
  \newcommand*\fsize{\dimexpr\f@size pt\relax}%
  \newcommand*\lineheight[1]{\fontsize{\fsize}{#1\fsize}\selectfont}%
  \ifx\svgwidth\undefined%
    \setlength{\unitlength}{141.73228346bp}%
    \ifx\svgscale\undefined%
      \relax%
    \else%
      \setlength{\unitlength}{\unitlength * \real{\svgscale}}%
    \fi%
  \else%
    \setlength{\unitlength}{\svgwidth}%
  \fi%
  \global\let\svgwidth\undefined%
  \global\let\svgscale\undefined%
  \makeatother%
  \begin{picture}(1,0.88)%
    \lineheight{1}%
    \setlength\tabcolsep{0pt}%
    \put(0,0){\includegraphics[width=\unitlength,page=1]{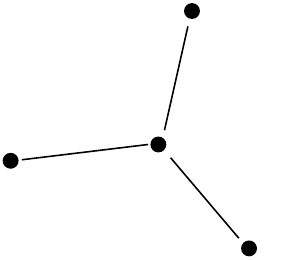}}%
    \put(-0.00541607,0.25265553){\makebox(0,0)[lt]{\lineheight{1.25}\smash{\begin{tabular}[t]{l}$a$\end{tabular}}}}%
    \put(0.25172848,0.30044001){\makebox(0,0)[lt]{\lineheight{1.25}\smash{\begin{tabular}[t]{l}$e_1$\end{tabular}}}}%
    \put(0.68305322,0.23142807){\makebox(0,0)[lt]{\lineheight{1.25}\smash{\begin{tabular}[t]{l}$e_2$\end{tabular}}}}%
    \put(0.60714001,0.62134557){\makebox(0,0)[lt]{\lineheight{1.25}\smash{\begin{tabular}[t]{l}$e_3$\end{tabular}}}}%
    \put(0.86938538,0.02094162){\makebox(0,0)[lt]{\lineheight{1.25}\smash{\begin{tabular}[t]{l}$b$\end{tabular}}}}%
    \put(0.67615195,0.82493083){\makebox(0,0)[lt]{\lineheight{1.25}\smash{\begin{tabular}[t]{l}$c$\end{tabular}}}}%
    \put(0.57953523,0.37290258){\makebox(0,0)[lt]{\lineheight{1.25}\smash{\begin{tabular}[t]{l}$v$\end{tabular}}}}%
  \end{picture}%
\endgroup%
		
	\caption{An embedded tree $T=D_4$ with three boundary vertices $\{a, b, c\}$, a non-boundary vertex $v$, and three edges $\{e_1, e_2, e_3\}$. 
	We note that $E_v(T)= \{e_1, e_2, e_3\}$ (resp.\ the set of boundary vertices $\{a, b, c\}$) is a cyclically ordered set.}
	\label{figure cyclic order}
\end{figure} 

\begin{remark}
	\label{rmk equivalent embedded trees}
	The cyclic orders are given by the orientation of $\mathbb{R}^2$. 
	Thus, for two embedded trees $(T,f_1)$ and $(T,f_2)$ of the same tree $T$, if there is an orientation preserving diffeomorphism $h : \mathbb{R}^2 \to \mathbb{R}^2$ such that $h \circ f_1 = f_2$, then two embedded trees induce the same cyclic orders. 
\end{remark}

We also define the notion of {\em rooted tree}.
\begin{definition}
	\label{def rooted tree}
	Let $T$ be a tree.
	\begin{enumerate}
		\item A {\em root} of $T$ is a pair $(v \in V(T), e \in E(T))$ such that $v$ is an end point of $e$. 
		\item A {\em rooted tree} $T$ is a tree equipped with a specific choice of root. 
	\end{enumerate}
\end{definition}

For simplicity, by a tree $T$, we mean an embedded tree $T$ with root $(v_0,e_0)$.

Let $T$ be a rooted tree with root $(v_0, e_0)$. 
Then, it is easy to check the following facts. 
\begin{itemize}
	\item For any $e \in E(T)$, one can equip a natural orientation on $e$ so that $e$ is ``going away'' from the root vertex $v_0$. 
	According to the natural orientation on $e \in E(T)$, one can see an edge $e$ as an arrow. 
	We use terms heads/tails of edges from this view-point. 
	\item For any $v \in V(T)$, there is a unique sequence of edges $e_1, \dots, e_k$ connecting $v$ and the root $v_0$ such that the tail of $e_1$ is the root $v_0$, the head of $e_i$ is the same as the tail of $e_{i+1}$, and the head of $e_k$ is $v$. 
\end{itemize}

From the above facts, we define a {\em distance function}.
\begin{definition}
	\label{def distance}
	Let $T$ be a tree and let $(v_0,e_0)$ be the chosen root.
	Then, there is a well-defined function 
	\[dist : V(T) \to \mathbb{Z}_{\geq 0}, \hs v \mapsto k,\]
	where $v$ is the number of edges in the unique sequence connecting $v$ and $v_0$ as descried above.
\end{definition}

Let $T$ be a tree, and let $(v_0,e_0)$ be the root of $T$. 
We would like to make $E_v(T)$ (resp.\ the set of boundary vertices) an ordered set, not just a cyclically ordered set. 
We note that it is enough to set the first or the last element of each of sets.  
Thus, the following will give an order on each sets:
\begin{itemize}
	\item If $v$ is the root vertex $v_0$, then we set the root edge $e_0$ as the {\em first} edge of $E_{v_0}(T)$. 
	\item If $v$ is not a root vertex $v_0$, then there is a unique edge $e \in E_v(T)$ such that the head of $e$ is $v$. 
	We set the unique $e$ as the {\em last} edge of $E_v(T)$. 
	\item The {\em first} element of the set of the boundary vertices is the boundary vertex that is connected to the root edges only using the {\em first edges}.
	More precisely, a boundary vertex $v$ is the first boundary vertex if there exists a finite sequence of edges $\{e_0, e_1, \dots, e_k\}$ such that $e_0$ is the root edge, $e_{i+1}$ is the first edge of $E_v(T)$ where $v$ is the head of $e_i$ for all $i = 0, \dots, k-1$, and the head of $e_k$ is the boundary vertex $v$. 
\end{itemize} 

Now, we define the {\em height} function.
The height function and the distance function (Definition \ref{def distance}) will be used in Definition \ref{def order of vertices} in order to define an order on $V(T)$.

For an arbitrary vertex $v \in V(T)$, there exists a unique finite sequence of edges $\{f_1, f_2, \dots, f_k\} \subset E(T)$ such that 
\begin{itemize}
	\item $f_1$ is the first edge of $E_v(T)$, 
	\item $f_{i+1}$ is the first edge of $E_{v_i}(T)$ where $v_i$ is the head of $f_i$, and
	\item the head of the last edge $f_k$ is a boundary vertex.
\end{itemize}
Let us assume that the head of $f_k$ be the $j^{th}$ element of the set of boundary vertices. 
Then, we define the height of $v$ as follows.
\begin{definition}
	\label{def height}
	The {\em height} of a vertex $v \in V(T)$ is $j \in \mathbb{N}$ that is the order of the unique boundary vertex in the above argument. 
	This defines a function 
	\[height : V(T) \to \mathbb{N}, v \mapsto j.\]
\end{definition}
We note that since the sequence $\{f_1, \dots, f_k\}$ in the above argument is unique, Definition \ref{def height} is well-defined. 

For a tree $T$, Definition \ref{def order of vertices} gives an order on $V(T)$.
\begin{definition}
	\label{def order of vertices} 
	\mbox{}
	\begin{enumerate}
		\item Let $v, w \in V(T)$. 
		Then, we say {\em $v < w$},  
		\begin{itemize}
			\item if $height(v) < height(w)$, or
			\item if $height(v) = height(w)$ and $dist(v) < dist(w)$.  
		\end{itemize} 
		For convenience, we label elements of
		\[V(T) = \{v_0, v_1, \dots, v_m\}\] 
		so that $v_i < v_j$ if $i <j$. 
		\item Let $e, f \in E(T)$.
		Then, we say {\em $e <f$} if the head of $e$ is less than the head of $f$. 
		We label 
		\[E(T) = \{e_0, \dots, e_{m-1}\},\]
		so that the head of $e_i$ is $v_{i+1}$. 
		We note that $m = |V(T)|$.
	\end{enumerate}
\end{definition}

The followings are obvious: 
\begin{itemize}
	\item the root vertex $v_0$ is the first element in $V(T)$, i.e., for any vertex $v \in V(T)$, $v_0 \leq v$, and
	\item $v_k$ is connected to $\{v_0, \dots, v_{k-1}\}$ by an edge $e_{k-1}$. 
\end{itemize}

We note that for a tree $T$ with a specific embedding $f_0: T \to \mathbb{R}^2$, there is another embedding $f:T \to \mathbb{R}^2$ such that 
\[f(v):= \left(dist(v), height(v)\right) \text{  for all  } v \in V(T).\]
As mentioned in Remark \ref{rmk equivalent embedded trees}, two embedding $f_0$ and $f$ will give the same orders on $V(T)$ and $E(T)$. 
In the rest of the paper, we assume that every tree $T$ admits an embedding satisfying
\[f(v):= \left(dist(v), height(v)\right).\]

We note that 
\[f(\text{the root vertex  }v_0) = (0,1) \text{  and  } f(\text{the root edge  } e_0) = \{(t,1) \in \mathbb{R}^2 | t \in [0,1]\}.\]
Thus, by giving the embedded image $f(T)$, one can specify the root. 

Figure \ref{figure example of embedded trees} gives two examples of different rooted and embedded trees such that their abstract trees, i.e., trees without roots and embeddings, are the same.

\begin{figure}[h]
	\centering
	\input{example_of_embedded_trees.txt}		
	\caption{Two different embedded trees are described in Figure \ref{figure example of embedded trees}. As trees, not embedded trees, they are the same tree which is the Dynkin diagram of $D_4$ type.} 
	\label{figure example of embedded trees}
\end{figure} 

We end the present Section by defining the followings for the later use.
\begin{definition}
	\label{def subtress}
	Let $T$ be a tree with ordered sets 
	\[V(T) = \{v_0, \dots, v_m\}, E(T) = \{e_0, \dots, e_{m-1}\}.\]
	\begin{enumerate}
		\item Let {\em $T^{(k)}$} be the subtree of $T$ consisting of 
		\[V(T^{(k)}) = \{v_0, \dots, v_k\}, E(T^{(k)}) = \{e_0, \dots, e_{k-1}\},\]
		for $1 \leq k \leq m$.
		\item Let {\em $\overline{T}$} be the tree obtained from $T$ by shrinking the first edges of $E_v(T)$, for all vertices $v \in V(T)$.
	\end{enumerate}
\end{definition}

An example of $\overline{T}$ is given in Figure \ref{figure example of overlineT}.
\begin{figure}[h]
	\centering
	\input{example_of_overlineT.txt}		
	\caption{The left is a rooted embedded tree $T$ and the right is the corresponding $\overline{T}$.} 
	\label{figure example of overlineT}
\end{figure}

\begin{remark}
	\label{rmk subtrees}
	\mbox{}
	\begin{enumerate}
		\item It is easy to check that 
		\[T^{(1)} \subset \dots \subset T^{(m)} = T.\]
		We will operate an induction on the sequence $\{T^{(k)}\}$ in order to prove Theorem \ref{thm rough statement3}.
		\item For any rooted embedded tree $T$, there is a natural quotient map $q : T \to \overline{T}$.
		We note that 
		\begin{itemize}
			\item $q$ sends a vertex $v \in V(T)$ to a vertex of $\overline{T}$,
			\item for any $v \in V(T)$, $q$ sends the first edge of $E_v(T)$ to a vertex $q(v) \in V\left(\overline{T}\right)$, 
			\item for the edges $e \in E(T)$ such that $e$ is not a first edge of $E_v(T)$ for any $v \in V(T)$, $q$ sends $e$ to an edge of $\overline{T}$.
		\end{itemize}
	\end{enumerate}
\end{remark}

\section{The algorithm for the plumbings along trees}
\label{section the algorithm for the plumbings along trees}
As defined in Definition \ref{def abstract Lefschetz fibration}, an abstract Lefschetz fibration consists of two things, a regular fiber, and a cyclically ordered collection of exact Lagrangian spheres in the fiber, i.e., the vanishing cycles. 
In the present section, we give a three-step algorithm producing an abstract Lefschetz fibration from an embedded, rooted tree $T$.
We will show that the plumbing space whose plumbing pattern is $T$ is the total space of the constructed Lefschetz fibration in Section \ref{section proof of plumbing tree theorem}. 

Let $T$ be a tree.  
In the first step, we construct the fiber.
The fiber $F$ is a plumbing of multiple copies of $T^*S^{n-1}$ along $\overline{T}$ defined in Definition \ref{def subtress}, (2).
The second step is to choose a vanishing cycle for each vertex of the input tree. 
The last step is to choose a vanishing cycles for some edges, but not all edges of the tree. 
We describe each of three steps abstractly below. 
Also, we will consider a specific example case given in Figure \ref{figure example} at the end of each steps.

\begin{figure}[h]
	\centering
	\input{example.txt}		
	\caption{An example tree $T$ is the right, and the corresponding $\overline{T}$ is the left.} 
	\label{figure example}
\end{figure}

\vskip0.2in

\noindent {\em Step 1. The fiber:}
Let $T$ be a given tree with 
\[V(T) = \{ v_0, v_1, \dots, v_m\}, E(T) = \{e_0, e_1, \dots, e_{m-1}\}.\]
We note that $V(T)$ and $E(T)$ are ordered sets, and the subscription above respects the order. 

First, we set a notation. 
\begin{definition}
	\label{def P_n(T)}
	Let {\em $P_n(T)$} denote the plumbing space of copies of $T^*S^n$, whose plumbing pattern is $T$.
\end{definition}

For a given $T$, and for $n \geq 2$, we set $P_{n-1}(\overline{T})$ as the fiber of the resulting abstract Lefschetz fibration. 

When we apply the first step to the example case given in Figure \ref{figure example}, the fiber should be $P_{n-1}(D_4)$ where $D_4$ means the Dynkin diagram of $D_4$ type. 
\vskip0.2in

\noindent {\em Step 2. Vanishing cycles corresponding to vertices:}
For each vertex $v_i \in V(T)$, we add one vanishing cycle.

We would like to note that the fiber $P_{n-1}(\overline{T})$ is a plumbing space. 
Thus, there is an exact Lagrangian sphere corresponding to each vertices of $\overline{T}$. 
For such Lagrangians, we set a notation. 

\begin{definition}
	\label{def L_v}
	For a tree $T$ and $v \in V(T)$, let {\em $L_v$} denote the Lagrangian sphere in $P_n(T)$ corresponding to the given vertex $v$.
\end{definition} 

For $v_i \in V(T)$, we choose the exact Lagrangian $L_{q(v_i)} \subset P_{n-1}(\overline{T})$. 
We note that the quotient map $q$ is defined in Remark \ref{rmk subtrees}.
Then, we have a cyclically ordered collection of exact Lagrangian spheres
\begin{gather}
	\label{eqn the first step abstract}
	\{L_{q(v_0)}, L_{q(v_1)}, \dots, L_{q(v_m)}\}.
\end{gather}
For the example case in Figure \ref{figure example}, we have 
\begin{gather}
	\label{eqn the first step}
	(P_{n-1}(D_4); L_{q(v_0)}, L_{q(v_1)}, L_{q(v_2)}, L_{q(v_3)}, L_{q(v_4)}, L_{q(v_5)}).	
\end{gather}

\begin{remark}
	We note that $L_{q(v_0)} = L_{q(v_1)} = L_{q(v_2)}$. 
	This gives some matching cycles of the resulting Lefschetz fibration.
	After finishing the algorithm, the matching cycles will correspond to $L_{v_1}, L_{v_2} \subset P_n(T)$. 
\end{remark}
\vskip0.2in

\noindent {\em Step 3. Vanishing cycles corresponding to edges:}
For each edge $e \in E(T)$ such that 
\begin{enumerate}
	\item[(i)] $e$ is not the first edge of $E_v(T)$ for the vertex $v \in V(T)$ such that $v$ is the tail of $e$, or
	\item[(ii)] $e$ is the root edge $e_0$,
\end{enumerate}
we add a vanishing cycle. 

We note that for adding a vanishing cycle, we need to choose two things, one is an exact Lagrangian in the fiber $P_{n-1}(\overline{T})$, and the other is the location where the vanishing cycle is in the cyclically ordered collection of vanishing cycles. 

In order to choose the exact Lagrangian, let $h \in V(T)$ denote the {\em head} of an edge $e$ satisfying either (i) or (ii). 
Then, the exact Lagrangian sphere corresponding to $e$ is $L_{q(h)} \in P_{n-1}(\overline{T})$.

In order to choose the location of the vanishing cycle $L_{q(h)}$ (corresponding to $e$), let us assume that $e$ satisfies (i).
Let the tail of $e$ be $v_i$, i.e., $(i+1)^{th}$ vertex in $V(T)$. 
Then, the vanishing cycle for $e$ will be located between the vanishing cycles corresponding to $v_{i-1}$ and $v_i$, i.e., $L_{q(v_{i-1})}$ and $L_{q(v_i)}$ in \eqref{eqn the first step abstract}.

We note that by the above procedure, one can add multiple vanishing cycles between $v_{i-1}$ and $v_i$. 
To explain that let $\{e_{i_1}, \dots, e_{i_k}\}$ be the set of edges such that their tails are $v_i \neq v_0$, and such that they satisfy (i).
In other words, 
\[\{e_{i_1}, \dots, e_{i_k}\} = E_{v_i} \setminus \left\{\text{the first edge starting at  } v_i\right\}.\]
Moreover, let $i_1 < i_2 < \dots < i_k$. 
Then, every vanishing cycle corresponding to one of the edges is located between $L_{q(v_{i-1})}$ and $L_{q(v_i)}$.  
And, we put the vanishing cycles in the {\em reverse} order, i.e., the vanishing cycle corresponding to $e_{i_k}$ comes first, the vanishing cycle corresponding to $e_{i_{k-1}}$ comes second, and so on.

Lastly, we should consider the edge satisfying (ii), i.e., the root edge $e_0$.
For the root edge $e_0$, we add the vanishing cycle $L_{q(v_1)}$ at the first position of the abstract Lefschetz fibration. 
This is the last step of the algorithm.

For our example, one can observe that we need to add vanishing cycles for edges $e_0, e_2, e_3, e_4$ because $e_2, e_3$ and $e_4$ satisfy (i), and $e_0$ satisfies (ii).
The corresponding vanishing cycles are $L_{q(v_1)}, L_{q(v_3)}, L_{q(v_4)}, L_{q(v_5)}$. 
The positions of $L_{q(v_3)}$ and $L_{q(v_4)}$ (corresponding to $e_2$ and $e_3$) are located between $L_{q(v_0)}$ and $L_{q(v_1)}$ since the tails of $e_2$ and $e_3$ are the same vertex $v_1$. 
Since $e_2 < e_3$, $L_{q(v_4)}$ comes earlier than $L_{q(v_3)}$. 

The position of $L_{q(v_5)}$ (corresponding to $e_4$) is located at the front of $L_{q(v_0)}$ in Equation \eqref{eqn the first step} since the tail of edge $e_4$ is $v_0$. 
After that, we add $L_{q(v_1)}$ corresponding to the root edge $e_0$, at the first of the collection of vanishing cycles. 
At the end, we construct the following abstract Lefschetz fibration: 
\[(P_{n-1}(D_4); L_{q(v_1)}, L_{q(v_5)}, L_{q(v_0)}, L_{q(v_4)}, L_{q(v_3)}, L_{q(v_1)}, L_{q(v_2)}, L_{q(v_3)}, L_{q(v_4)}, L_{q(v_5)}). \]

We set notation for the later use. 
\begin{definition}
	\label{def LF}
	Let $T$ be a tree. 
	We denote the {\em Lefschetz fibration obtained from $T$} in Section \ref{section the algorithm for the plumbings along trees} by {\em $LF(T)$}.
\end{definition}
\vskip0.2in

\noindent {\em Matching cycles:}
For a given tree $T$, let $LF(T)$ be the abstract Lefschetz fibration constructed by the above algorithm.
Let $E$ be the total space of $LF(T)$, and let $\pi: E \to \mathbb{C} = \mathbb{R}^2$ be a specific Lefschetz fibration.
  
Without loss of generality, one can assume that all the singular values of $\pi$ are contained in the unit circle in $\mathbb{R}^2 \simeq \mathbb{C}$.
Then, for any vertex $v \in V(T)$, there are the following matching cycles:
\begin{itemize}
	\item If $v$ is the root vertex, i.e., $v = v_0$, then on the base, we consider the straight line segment connecting two singular values corresponding to the vertex $v_0$ and the root edge $e_0$. 
	We note that the vanishing cycles corresponding to $v_0$ and $e_0$ are the same because both are $L_{q(v_0)} = L_{q(v_1)}$. 
	More precisely, by definition, $v_0$ and $v_1$ are connected by the root edge of $T$, and the root edge $e_0$ is shrunk in $\overline{T}$. 
	Thus, $q(v_0) = q(v_1)$.
	\item For $i \geq 1$, if $v_i$ is connected to the vertex $v_{i-1}$, or equivalently, the edge whose head is $v_i$ is the first edge of $E_{v_{i-1}}(T)$, then two singular values corresponding to vertices $v_{i-1}$ and $v_i$ have the same Lagrangian sphere as their vanishing cycles. 
	We consider the straight line segment connecting two singular values corresponding to the vertex $v_i$ and $v_{i-1}$, then the line segment becomes a matching cycle.
	\item For $i \geq 1$, if $v_i$ is `not' connected to the vertex $v_{i-1}$, then there is a singular value corresponding to the edge $e_{i-1}$ since $e_{i-1}$ satisfies (i) of the third step.  
	Then, the line segment connecting two singular values corresponding to the vertex $v_i$ and the edge $e_{i-1}$ becomes a matching cycle. 
\end{itemize} 

We note that in the above, we fix a matching cycle for each of vertices $v \in V(T)$. 
In other words, for each vertex $v \in V(T)$, we fix a Lagrangian sphere in $E$.
Using this, we define the following: 
\begin{definition}
	\label{def corresponding Lagrangian}
	Let $T$ be a tree and let $v \in V(T)$. 
	And, we use $E$ to denote the total space of the Lefschetz fibration obtained from $T$. 
	Then, the {\em Lagrangian sphere corresponding to $v$} is the Lagrangian sphere corresponding to the fixed matching cycle for $v \in V(T)$. 
	We denote the Lagrangian sphere corresponding to $v$ by $S_v$. 
\end{definition}
In Section \ref{section proof of plumbing tree theorem}, we will prove that $P_n(T) \simeq E$, and that $L_v \subset P_n(T)$, the zero section of $T^*S^n$ corresponding to $v$, will be identified with $S_v \subset E$ by the identification.

In other words, we prove Theorem \ref{thm plumbing along tree}.
\begin{theorem}[=Theorem \ref{thm rough statement3}]
	\label{thm plumbing along tree}
	Let $T$ be an abstract tree.
	For any embedding of $T$ into $\mathbb{R}^2$, and for any root of $T$, the algorithm given in Section \ref{section the algorithm for the plumbings along trees} produces an abstract Lefschetz fibration whose total space is equivalent to $P_n(T)$ up to symplectic completion.
	Moreover, the Lagrangian sphere $L_v \in P_n(T)$ for any $v \in V(T)$ is Hamiltonian equivalent to $S_v$ above.
\end{theorem}

\begin{remark}
	\label{rmk different lefschetz fibration}
	We point out that the plumbing space $P_n(T)$ depends only on the abstract tree $T$. 
	Since the input of the algorithm is a rooted, embedded tree, we can observe that by choosing different roots and embeddings of $T$, one can produce different Lefschetz fibrations on the same Weinstein manifold $P_n(T)$. 
\end{remark}

\section{Proof of Theorem \ref{thm plumbing along tree}}
\label{section proof of plumbing tree theorem}
Let $T$ be a given tree with ordered sets
\[V(T) = \{ v_0, v_1, \dots, v_m\}, E(T) = \{e_0, e_1, \dots, e_{m-1}\}.\]

As mentioned in Section \ref{section sketch of the proof of theorem}, we prove Theorem \ref{thm plumbing along tree} by induction on the following increasing sequence
\[T^{(1)} \subset \dots \subset T^{(m)}.\]
We note that the increasing sequence is defined in Definition \ref{def subtress}. 
The induction hypothesis is the following:
\begin{itemize}
	\item For $T^{(k)}$, the total space of the Lefschetz fibration obtained from $T^{(k)}$, $LF(T^{(k)})$, is $P_n(T^{(k)})$, and by the identification between the total space and $P_n(T^{(k)})$, $S_v$ and $L_v$ are identified for each $v \in V(T^{(k)})$. 
\end{itemize}
\vskip0.2in

\noindent{\em The base step.}
For the base case, we operate the algorithm for $T^{(1)}$. 
By definition, $T^{(1)}$ is a tree with two vertices $v_0$ and $v_1$ connected by an edge $e_0$, i.e., the Dynkin diagram of $A_2$-type. 

When one operates the algorithm for $T^{(1)}$, one obtains 
\[LF(T^{(1)}) = (T^*S^{n-1}: S^{n-1}, S^{n-1}, S^{n-1}),\]
where $S^{n-1}$ denotes the zero section of $T^*S^{n-1}$.
It is well-known that the total space of $LF(T^{(1)})$ is the Milnor fiber of $A_2$-type, i.e., $P_n(T^{(1)})$.
Also, it is well-known that the matching cycle condition also holds for the given Lefschetz fibration.
This completes the proof of the base case.

For the inductive step, let us assume that $LF(T^{(k)})$ satisfies the induction hypothesis.  
We note that by definition, $T^{(k+1)}$ is obtained by adding a vertex $v_{k+1}$ and an edge $e_k$ to $T^{(k)}$. 
Thus, the difference between $LF(T^{(k)})$ and $LF(T^{(k+1)})$ occurs by $v_{k+1}$ and $e_k$. 

One can observe that there are two cases.
The first (resp.\ second) case is that $e_k$ is (resp.\ is not) the first edge of $E_v(T)$ if $v$ denotes the tail of $e_k$.
We discuss those two cases separately.  
\vskip0.2in

\noindent {\em The first case of the induction step.}
Before starting the proof for the first case, let us point out the difference between $LF(T^{(k)})$ and $LF(T^{k+1})$.
For the first case, it is easy to check that the added vertex $v_{k+1}$ is connected to $v_k$, and it induces the following two facts:
\begin{itemize}
	\item We note that $LF(T^{(k)})$ and $LF(T^{k+1})$ have the same fiber $P_{n-1}(\overline{T^{(k)}}) = P_{n-1}(\overline{T^{(k+1)}})$.
	\item $LF(T^{(k+1)})$ is obtained by adding one vanishing cycle 
	\[L_{q(v_{k}) = q(v_{k+1})},\]
	at the end of the ordered collection of vanishing cycles. 
\end{itemize}

Now, we compare the plumbing spaces $P_n(T^{(k)})$ and $P_n(T^{(k+1)})$. 
The later is obtained by plumbing a $T^*S^n$ to the former. 
The new plumbing point is the intersection point between $L_{v_{k+1}}$ and $L_{v_k}$. 

One can understand the plumbing procedure as attaching a new critical Weinstein handle. 
To attaching the critical handle, we need an attaching Legendrian. 
The corresponding attaching Legendrian could be seen as the asymptotic boundary Legendrian sphere of $D^*_pL_{v_k} \subset P_n(T^{(k)})$.
We would like to find the Lagrangian disk under the induction hypothesis.

From the induction hypothesis, one can identify the $P_n(T^{(k)})$ and the total space of $LF(T^{(k)})$. 
Based on the identification, one can see that the Lefschetz thimble of the singular value corresponding to $v_k$ corresponds to a Lagrangian disk.
Moreover, the Lagrangian disk transversely intersects with $S_{v_k} \simeq L_{v_k}$. 
And, from the Lefschetz type critical points condition (see Definition \ref{def Lefschetz fibration}), one can see the Lefschetz thimble as the cotangent fiber we would like to find, i.e., $D^*_pL_{v_k}$. 
See Figure \ref{figure example1}.

From the above argument, one can get the Lefschetz fibration by attaching a critical handle along the boundary of the Lefschetz thimble. 
Then, by the construction of $LF(T^{(k+1)})$, the resulting Lefschetz fibration is $LF(T^{(k+1)})$. 
Moreover, the construction of $LF(T^{(k+1)})$ explains that $L_{v_{k+1}}$ and $S_{v_{k+1}}$ are identified. 

\begin{figure}[h]
	\centering
\begingroup%
  \makeatletter%
  \providecommand\color[2][]{%
    \errmessage{(Inkscape) Color is used for the text in Inkscape, but the package 'color.sty' is not loaded}%
    \renewcommand\color[2][]{}%
  }%
  \providecommand\transparent[1]{%
    \errmessage{(Inkscape) Transparency is used (non-zero) for the text in Inkscape, but the package 'transparent.sty' is not loaded}%
    \renewcommand\transparent[1]{}%
  }%
  \providecommand\rotatebox[2]{#2}%
  \newcommand*\fsize{\dimexpr\f@size pt\relax}%
  \newcommand*\lineheight[1]{\fontsize{\fsize}{#1\fsize}\selectfont}%
  \ifx\svgwidth\undefined%
    \setlength{\unitlength}{283.46456693bp}%
    \ifx\svgscale\undefined%
      \relax%
    \else%
      \setlength{\unitlength}{\unitlength * \real{\svgscale}}%
    \fi%
  \else%
    \setlength{\unitlength}{\svgwidth}%
  \fi%
  \global\let\svgwidth\undefined%
  \global\let\svgscale\undefined%
  \makeatother%
  \begin{picture}(1,0.48)%
    \lineheight{1}%
    \setlength\tabcolsep{0pt}%
    \put(0,0){\includegraphics[width=\unitlength,page=1]{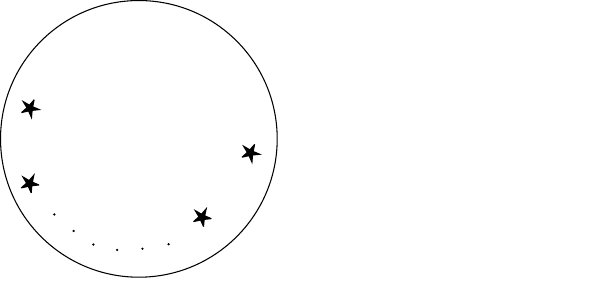}}%
    \put(0.07766871,0.27998571){\makebox(0,0)[lt]{\lineheight{1.25}\smash{\begin{tabular}[t]{l}$e_0$\end{tabular}}}}%
    \put(0.36943279,0.21264093){\makebox(0,0)[lt]{\lineheight{1.25}\smash{\begin{tabular}[t]{l}$v_{k}$\end{tabular}}}}%
    \put(0,0){\includegraphics[width=\unitlength,page=2]{example1.pdf}}%
    \put(0.58715034,0.27998569){\makebox(0,0)[lt]{\lineheight{1.25}\smash{\begin{tabular}[t]{l}$e_0$\end{tabular}}}}%
    \put(0.78784492,0.38767369){\makebox(0,0)[lt]{\lineheight{1.25}\smash{\begin{tabular}[t]{l}$v_{k+1}$\end{tabular}}}}%
    \put(0.87891441,0.21264091){\makebox(0,0)[lt]{\lineheight{1.25}\smash{\begin{tabular}[t]{l}$v_{k}$\end{tabular}}}}%
    \put(0,0){\includegraphics[width=\unitlength,page=3]{example1.pdf}}%
  \end{picture}%
\endgroup%
		
	\caption{The left is the base of $LF(T^{(k)})$ together with a Lefschetz thimble. The right is the base of $LF(T^{(k+1)})$. On the right picture, the red interval is the matching cycle for $S_{v_{k+1}} \simeq L_{v_{k+1}}$. The red circle part corresponds to $P_n(T^{(k)})$ that is a subset of $P_n(T^{(k+1)})$.} 
	\label{figure example1}
\end{figure}
\vskip0.2in

\noindent {\em The second case of the induction step.}
For the second case, we observe that $v_{k+1}$ is connected to $v_j$ such that $0 \leq j < k-1$ by $e_k$.
When one compares the fibers of $LF(T^{(k)})$ and $LF(T^{(k+1)})$, one can observe the following two differences:
\begin{itemize}
	\item First, the fiber of $LF(T^{(k+1)})$, or equivalently $P_{n-1}(\overline{T^{(k+1)}})$, is obtained by plumbing $T^*S^{n-1}$ and the fiber of $LF(T^{(k)})$. 
	Especially, we note that $P_{n-1}(\overline{T^{(k)}}) \subsetneq P_{n-1}(\overline{T^{(k+1)}})$.
	\item Second, $LF(T^{(k+1)})$ has two more singular values than $LF(T^{(k)})$. 
	These two singular values correspond to $v_{k+1}$ and $e_k$. Since the head of $e_k$ is $v_{k+1}$, they have the same vanishing cycles. The same vanishing cycles are the zero section of $T^*S^{n-1}$ in the first item. Or equivalently, the vanishing cycles are $L_{q(v_{k+1})} \subset P_{n-1}(\overline{T^{(k+1)}})$.
\end{itemize}

As same as the first case, $P_n(T^{(k+1)})$ is obtained by plumbing one $T^*S^n$ to $P_n(T^{(k)})$, or equivalently, attaching an extra critical Weinstein handle along the asymptotic boundary of a cotangent fiber $D^*_p L_{v_j}$.
Thus, our strategy is to find a Lefschetz thimble corresponding to $T^*_pL_{v_j}$, as we did above.

To do that, under the induction hypothesis, we consider a {\em stabilization} of $LF(T^{(k)})$. 
See Section \ref{section the effects of handle moves}. 
A stabilization of a Lefschetz fibration can be seen as adding a canceling pair of index $(n-1,n)$ handles, and it changes both of the fiber and the collection of singular values. 
For the fiber, we need to attach a critical Weinstein handle, i.e., index $(n-1)$ Weinstein handle of dimension $2(n-1)$, to the original fiber $P_{n-1}(\overline{T^{(k)}})$. 
And as the new fiber, we obtain $P_{n-1}(\overline{T^{k+1})})$. 
It corresponds to attach the index $(n-1)$ handle in the added canceling pair.

Moreover, we need to add one more singular value and it corresponds to attach the critical handle in the canceling pair. 
We note that the vanishing cycle of the new singular value should be $L_{q(v_k)}$ because of the definition of stabilization. 
The location of the added singular value could be anywhere, but we choose the place of the singular value for $e_k$ in $LF(T^{(k+1)})$. 
See the left one picture of Figure \ref{figure example2}.

\begin{figure}[h]
	\centering
	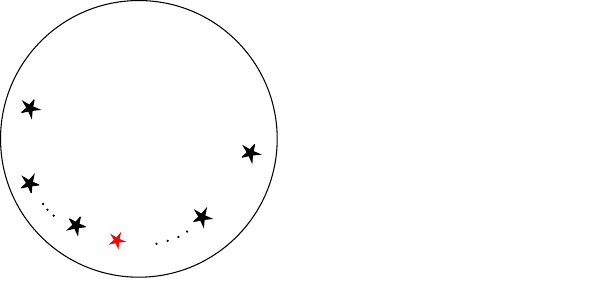		
	\caption{The left is the picture of the base of $LF(T^{(k)})$ after the stabilization. The added singular value is the red star marked point, and the red line segment corresponds to the Lefschetz thimble. The right is the picture of $LF(T^{(k+1)})$. The red circle part corresponds to $P_n(T^{(k)})$ that is a subset of $P_n(T^{(k+1)})$.}  
	\label{figure example2}
\end{figure} 

Now similar to the first case, we consider a Lefschetz thimble ending at the added singular value. 
Then, it intersects the matching cycle corresponding to $S_{v_j}$ at a regular point. 
See Figure \ref{figure example2}.
Moreover, since the matching cycle and the Lefschetz thimble have different vanishing cycles, and since their vanishing cycles intersect at one point transversally, the Lefschetz thimble transversally intersects $S_{v_j}$.
Thus, one can see the Lefschetz thimble as the Lagrangian disk we would like to find. 

Finally, same as the first case, we can attach a critical handle along the boundary of the Lefschetz thimble. 
Then, it gives a new Lefschetz fibration, and the new Lefschetz fibration is $LF(T^{(k+1)})$. 
In other words, one can see that the total space of $LF(T^{(k+1)})$ and $P_n(T^{(k+1)})$ are the same Weinstein manifold.
Moreover, the matching cycle connecting two singular values corresponding to $e_k$ and $v_{k+1}$ should be the zero section of added $T^*S^n$, i.e., $L_{v_{k+1}}$. 
It completes the proof of inductive step.
\qed

\section{An application}
\label{section application} 
In this section, we introduce an application of Theorem \ref{thm plumbing along tree}. 
The application is to construct diffeomorphic families of Weinstein manifolds.
Weinstein manifolds which we are considering are plumbings of multiple copies of $T^*S^n$ along trees $T_m^j$ that are defined in Definition \ref{def T_m^j}.

\begin{definition}
	\label{def T_m^j}
	For any $m \in \mathbb{N}$ and any $ 1 \leq j \leq m$, let {\em $T_m^j$} denote the tree which is given in Figure \ref{figure T_m^j}. 
	\begin{figure}[h]
		\centering
\begingroup%
  \makeatletter%
  \providecommand\color[2][]{%
    \errmessage{(Inkscape) Color is used for the text in Inkscape, but the package 'color.sty' is not loaded}%
    \renewcommand\color[2][]{}%
  }%
  \providecommand\transparent[1]{%
    \errmessage{(Inkscape) Transparency is used (non-zero) for the text in Inkscape, but the package 'transparent.sty' is not loaded}%
    \renewcommand\transparent[1]{}%
  }%
  \providecommand\rotatebox[2]{#2}%
  \newcommand*\fsize{\dimexpr\f@size pt\relax}%
  \newcommand*\lineheight[1]{\fontsize{\fsize}{#1\fsize}\selectfont}%
  \ifx\svgwidth\undefined%
    \setlength{\unitlength}{226.77165354bp}%
    \ifx\svgscale\undefined%
      \relax%
    \else%
      \setlength{\unitlength}{\unitlength * \real{\svgscale}}%
    \fi%
  \else%
    \setlength{\unitlength}{\svgwidth}%
  \fi%
  \global\let\svgwidth\undefined%
  \global\let\svgscale\undefined%
  \makeatother%
  \begin{picture}(1,0.225)%
    \lineheight{1}%
    \setlength\tabcolsep{0pt}%
    \put(0,0){\includegraphics[width=\unitlength,page=1]{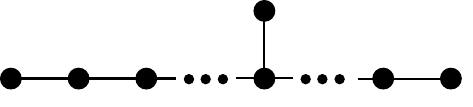}}%
    \put(0.01285024,0.00623575){\makebox(0,0)[lt]{\lineheight{1.25}\smash{\begin{tabular}[t]{l}$v_1$\end{tabular}}}}%
    \put(0.1554882,0.00623723){\makebox(0,0)[lt]{\lineheight{1.25}\smash{\begin{tabular}[t]{l}$v_2$\end{tabular}}}}%
    \put(0.29928831,0.00623723){\makebox(0,0)[lt]{\lineheight{1.25}\smash{\begin{tabular}[t]{l}$v_3$\end{tabular}}}}%
    \put(0.54762176,0.00623723){\makebox(0,0)[lt]{\lineheight{1.25}\smash{\begin{tabular}[t]{l}$v_j$\end{tabular}}}}%
    \put(0.78614463,0.00623723){\makebox(0,0)[lt]{\lineheight{1.25}\smash{\begin{tabular}[t]{l}$v_{m-1}$\end{tabular}}}}%
    \put(0.94503508,0.00623723){\makebox(0,0)[lt]{\lineheight{1.25}\smash{\begin{tabular}[t]{l}$v_m$\end{tabular}}}}%
    \put(0.58697062,0.18741217){\makebox(0,0)[lt]{\lineheight{1.25}\smash{\begin{tabular}[t]{l}$v_{m+1}$\end{tabular}}}}%
  \end{picture}%
\endgroup%
		
		\caption{Tree $T_m^j$.} 
		\label{figure T_m^j}
	\end{figure} 
\end{definition} 

Corollary \ref{cor diffeomorphic family} can be easily obtained from Theorem \ref{thm plumbing along tree} and arguments in \cite{Maydanskiy, MS, Maydanskiy-Seidel15}.
\begin{corollary}
	\label{cor diffeomorphic family}
	For odd $n \geq 5$ (resp.\ $n =3$), $m \in \mathbb{N}$, $P_n(T_m^j)$ and $P_n(T_m^{j+4})$ (resp.\ $P_n(T_m^{j+2})$) are diffeomorphic.
\end{corollary}
\begin{proof}
	We prove Corollary \ref{cor diffeomorphic family} for the case of odd $n \geq 5$, and the same argument will work for the case of $n=3$.
	 
	We apply Theorem \ref{thm plumbing along tree}, then it gives us an abstract Lefschetz fibration whose total space is $P_n(T_m^j)$.
	For any $m$ and $j$, the resulting Lefschetz fibrations have the same fiber $P_{n-1}(A_2)$, where $A_2$ means the Dynkin diagram of $A_2$-type.
	
	We note that $P_{n-1}(A_2)$ is a plumbing of two $T^*S^{n-1}$.
	Thus, there are two exact Lagrangian spheres corresponding to the zero sections of $T^*S^{n-1}$. 
	Let $\alpha$ (resp.\ $\beta$) denote the Lagrangian spheres $L_{q(v_1)} = \dots = L_{q(v_m)}$ (resp.\ $L_{q(v_{m+1})}$).
	We note that $v_i$ is a vertex in Figure \ref{figure T_m^j}, $q(v_i)$ is a vertex of $\overline{T^j_m}$, and $L_{q(v_i)}$ is defined in Definition \ref{def L_v}. 
	Then, the Lefschetz fibrations for $P_n(T_m^j)$ and $P_n(T_m^{j+4})$ are 
	\begin{gather}
		\label{eqn Lef T_m^j}
		(P_{n-1}(A_2); \alpha, \alpha, \dots, \alpha = L_{q(v_{j-1})}, \beta, \alpha= L_{q(v_j)}, \dots, \alpha = L_{q(v_m)}, \beta = L_{q(v_{m+1})}),\\
		\label{eqn Lef T_m^{j+2}}
		(P_{n-1}(A_2); \alpha, \alpha, \dots, \alpha = L_{q(v_{j+3})}, \beta, \alpha= L_{q(v_{j+4})}, \dots, \alpha = L_{q(v_m)}, \beta = L_{q(v_{m+1})}).
	\end{gather}
	
	We note that the middle $\beta$ in the Lefschetz fibration for $P_n(T_m^j)$ (resp.\ $P_n(T_m^{j+4})$) is located at $(j+1)^{th}$ (resp.\ $(j+5)^{th}$) position in the collection of vanishing cycles. 
	By taking the Hurwitz move, one can move the middle $\beta$ in Equation \eqref{eqn Lef T_m^j} to right.
	When we operate the Hurwitz move {\em four times}, then the vanishing cycle becomes $(\tau_\alpha)^4(\beta)$, and it is located at the $(j+5)^{th}$ position, where $\tau_\alpha$ denotes a Dehn twist along $\alpha$ on $P_{n-1}(A_2)$. 
	In other words, we have the following Lefschetz fibration for $P_n(T_m^j)$. 
	\begin{gather}
		\label{eqn Lef T_m^j modi}
		(P_{n-1}(A_2); \alpha, \alpha, \dots, \alpha = L_{q(v_{j+3})}, (\tau_\alpha)^4(\beta), \alpha= L_{q(v_j+4)}, \dots, \alpha = L_{q(v_m)}, \beta = L_{q(v_{m+1})}).
	\end{gather}

	One can observe that Equations \eqref{eqn Lef T_m^j modi} and \eqref{eqn Lef T_m^{j+2}} are the same except the $(j+5)^{th}$ vanishing cycles.
	The vanishing cycles are $\beta$ in \eqref{eqn Lef T_m^j modi} and $(\tau_\alpha)^4(\beta)$ in \eqref{eqn Lef T_m^{j+2}}.
	
	We recall that an abstract Lefschetz fibration gives a Weinstein handle decomposition of its total space. 
	Since we are interested in the smooth structure of $P_n(T^j_m)$ and $P_n(T^{j+4}_m)$, it is enough to show that $\beta$ and $(\tau^4)(\beta)$ induce Legendrian spheres in $\partial_\infty \left(P_{n-1}(A_2) \times \mathbb{C}\right)$ satisfying 
	\begin{itemize}
		\item two Legendrian spheres are isotopic as smooth spheres, and 
		\item the conformal symplectic normal bundles of two Legendrian spheres give the same framing under the isotopy connecting them.
	\end{itemize}
	
	Since \cite{MS, Maydanskiy-Seidel15} prove the above, it completes the proof of the case of odd $n \geq 5$. 
	
	For the case of $n =3$, it is simpler since the formal Legendrian structures on Legendrian sphere are unique in $\mathbb{R}^5$ as stated in \cite[Proposition A.4]{Murphy}.
\end{proof}

\begin{remark}
	\label{rmk exotic}
	\mbox{}
	\begin{enumerate}
		\item Corollary \ref{cor diffeomorphic family} gives diffeomorphic families, not just pairs. 
		For example, if $m =4k+2$ and if $n \geq 5$ is odd, then Corollary \ref{cor diffeomorphic family} gives the following diffeomorphic families of Weinstein manifolds 
		\[\{P_n(T_{4k+2}^1), P_n(T_{4k+2}^5), \dots, P_n(T_{4k+2}^{4k+1})\}.\]
		Since $P_n(T_{4k+2}^1)$ (resp.\ $P_n(T_{4k+2}^{4k+1})$) is the plumbing space whose plumbing pattern is the Dynkin diagram of $A_{4k+3}$-type (resp.\ $D_{4k+3}$-type), the Milnor fibers of $A_{4k+3}$ and $D_{4k+3}$-types are diffeomorphic to each other. 
		Similarly, the Milnor fibers of $A_8$ and $E_8$-types are diffeomorphic to each other. 
		
		It would be natural to ask whether those diffeomorphic families are exotic families or not as Weinstein manifolds. 
		It is answered in \cite{Choa-Karabas-Lee}.
		\item We note that we considered some restricted cases in Corollary \ref{cor diffeomorphic family}, but by using the same method, one can construct more diffeomorphic families of plumbing spaces whose plumbing patterns are not $T^j_m$.
	\end{enumerate}
\end{remark}

We end the present paper by mentioning another possible application. 
The possible application is to study symplectic automorphisms on $P_n(T)$. 
To be more precise, we note that since $P_n(T)$ is obtained by plumbing multiple copies of $T^*S^n$, $P_n(T)$ has at least $|V(T)|$ many Lagrangian spheres. 
Thus, there exist generalized Dehn twists along them. 

On the base of the Lefschetz fibration which Theorem \ref{thm plumbing along tree} gives, one has matching cycles corresponding to the Lagrangian spheres. 
Then, it is well-known that on the base of the Lefschetz fibration, a Dehn twist along a Lagrangian sphere can be descried a braid move related to the corresponding matching cycle. 
From the well-known fact, one can study the Dehn twists along Lagrangian spheres by using the Lefschetz fibration.
Especially, we expect that this recovers the results of the author's thesis \cite{Lee} which constructs a higher-dimensional stable/unstable Lagrangian laminations.

\bibliographystyle{abbrv}
\bibliography{lefschetz}
\end{document}